\documentclass[11pt]{amsart}
\usepackage{graphicx}
\usepackage{amsmath}
\usepackage{amssymb}
\usepackage{amsfonts}
\usepackage{verbatim}
\usepackage{scalefnt}

\usepackage{pstricks,tikz}

\usepackage[notref,notcite]{}
\usepackage{fancyhdr}
\usepackage{hyperref}

\begin{document}

\def\COMMENT#1{}

\newtheorem{theorem}{Theorem}[section]
\newtheorem{lemma}[theorem]{Lemma}
\newtheorem{proposition}[theorem]{Proposition}
\newtheorem{corollary}[theorem]{Corollary}
\newtheorem{conjecture}[theorem]{Conjecture}
\newtheorem{claim}[theorem]{Claim}

\def\eps{{\varepsilon}}
\newcommand{\cP}{\mathcal{P}}
\newcommand{\cT}{\mathcal{T}}
\newcommand{\cL}{\mathcal{L}}
\newcommand{\ex}{\mathbb{E}}
\newcommand{\eul}{e}
\newcommand{\pr}{\mathbb{P}}

\title{APPROXIMATE HAMILTON DECOMPOSITIONS OF ROBUSTLY EXPANDING REGULAR DIGRAPHS}
\author{DERYK OSTHUS AND KATHERINE STADEN}
\thanks{Deryk Osthus was supported by the EPSRC, grant no.~EP/J008087/1.}

\begin{abstract}
We show that every sufficiently large $r$-regular digraph $G$ which 
has linear degree and
is a robust outexpander has an approximate decomposition into edge-disjoint Hamilton cycles,
i.e.~$G$ contains a set of $r-o(r)$ edge-disjoint Hamilton cycles.
Here $G$ is a robust outexpander if for every set $S$ which is not too small and not too large, the `robust' outneighbourhood of $S$ is a little larger than $S$.
This generalises a result of K\"uhn, Osthus and Treglown on approximate Hamilton decompositions of dense regular oriented graphs.
It also generalises a result of Frieze and Krivelevich on approximate Hamilton decompositions of quasirandom (di)graphs.
In turn, our result is used as a tool by K\"uhn and Osthus to prove that any sufficiently large $r$-regular digraph $G$ which has linear degree and is a robust outexpander even has a Hamilton decomposition.
\end{abstract}

\date{\today}
\maketitle

\section{Introduction}
A Hamilton decomposition of a graph or digraph $G$ is a set of edge-disjoint Hamilton cycles which together cover all the edges of $G$. 
The first result in the area was proved by Walecki in 1892, who showed that a complete graph $K_n$ has a Hamilton decomposition if and only if $n$ is odd
(see e.g. \cite{lucas}, \cite{alspach}, \cite{abs}).
Tillson~\cite{till} solved the analogous problem for complete digraphs in 1980.
Though the area is rich in beautiful conjectures, until recently there were few general results.

Starting with a result of Frieze and Krivelevich~\cite{fk}, a very successful recent direction of research has been to find `approximate' Hamilton decompositions, 
i.e.~a set of edge-disjoint Hamilton cycles which cover almost all the edges of the given (di)graph.
The result in \cite{fk} concerns dense quasirandom graphs and digraphs. Hypergraph versions of this result were proved by Frieze, Krivelevich and Loh \cite{FKL} as well as Bal and Frieze \cite{BF}. 
Also, K\"uhn, Osthus and Treglown~\cite{kot} proved an approximate version of Kelly's conjecture.
This long-standing conjecture (see \cite{moon}) states that
every regular tournament has a Hamilton decomposition. 
In fact, the result in \cite{kot} is much more general, 
namely it states that every regular oriented graph on $n$ vertices whose in- and outdegree is slightly larger than $3n/8$ has an approximate Hamilton decomposition.
Here an oriented graph is a digraph with at most one edge between each pair of vertices (whereas a digraph may have one edge in each direction between a pair of vertices).
%Here a tournament is an orientation of the complete graph.

Our main result in turn is a far reaching generalisation of the result in~\cite{kot}.
Instead of a degree condition, it involves an expansion condition that has recently been shown to have a close connection with Hamiltonicity. 
This notion was introduced by K\"uhn, Osthus and Treglown in~\cite{KOTchvatal}. 
The condition states that for every set $S$ which is not too small and not too large, 
its `robust' outneighbourhood is at least a little larger than $S$ itself.
More precisely, suppose that $G$ is a digraph of order $n$ and $S \subseteq V(G)$. The $\nu$-\emph{robust outneighbourhood} $RN_{\nu , G}^+(S)$ of $S$ is  the set of 
vertices with at least $\nu n$ inneighbours in $S$. 
We say that $G$ is a \emph{robust $(\nu,\tau)$-outexpander}
if 
$$
|RN^+_{\nu,G}(S)|\ge |S|+\nu n \ \mbox{ for all } \ S\subseteq V(G) \ \mbox{ with } \ \tau n\le |S|\le  (1-\tau)n.
$$
Our main result states that every sufficiently large robustly outexpanding regular digraph has an approximate Hamilton decomposition.

\begin{theorem} \label{main}
For every $ \alpha >0$ there exists $\tau>0$ such that for all $\nu, \eta>0$ there exists $n_0=n_0 (\alpha,\nu,\tau,\eta)$ for which the following holds. Suppose that
\begin{itemize}
\item[{\rm (i)}] $G$ is an $r$-regular digraph on $n \ge n_0$ vertices, where $r\ge \alpha n$;
\item[{\rm (ii)}] $G$ is a robust $(\nu,\tau)$-outexpander.
\end{itemize}
Then $G$ contains at least $(1 -\eta)r$ edge-disjoint Hamilton cycles. Moreover, this set of Hamilton cycles can be found in time polynomial in $n$.
\end{theorem}

As observed in Lemma 12.1 of~\cite{KOKelly}, every oriented graph whose in- and outdegrees are all at least slightly larger than $3n/8$ is a robust outexpander, so this does generalise the main result of~\cite{kot}.
Moreover, it turns out that one can relax condition (i) to the requirement that $G$ is `almost regular'.
This is due to the fact (observed in~\cite{KOKelly}) that every almost regular robustly expanding digraph contains a spanning regular digraph of similar degree.

\begin{corollary} \label{cor}
For every $\alpha>0$ there exists $\tau >0$ such that for all $\nu, \eta > 0$ there exist $n_0 = n_0(\alpha , \nu, \tau , \eta)$ and $\gamma = \gamma(\alpha, \nu, \tau, \eta) > 0$ for which the following holds. Suppose that
\begin{itemize}
\item[(i)] $G$ is a digraph on $n \geq n_0$ vertices with $(\alpha - \gamma)n \leq d^\pm_G(x) \leq (\alpha + \gamma)n$ for every $x$ in $G$;
\item[(ii)] $G$ is a robust $(\nu, \tau)$-outexpander.
\end{itemize}
Then $G$ contains at least $(\alpha - \eta)n$ edge-disjoint Hamilton cycles. Moreover, this set of Hamilton cycles can be found in time polynomial in $n$.
\end{corollary}

The result in~\cite{kot} extends to almost regular oriented graphs in the same way, but is inherently non-algorithmic
(see Section~\ref{sec:sketch}). Since, for dense digraphs, the condition of being a robust outexpander is much weaker than that of being quasirandom, Corollary~\ref{cor} is much more general than the result in \cite{fk} mentioned earlier. Moreover, it is best possible in the sense that, for an almost regular digraph, an approximate Hamilton decomposition is obviously the best one can hope for.

Theorem~\ref{main} is used as an essential tool by K\"uhn and Osthus in~\cite{KOKelly} to prove the following result, which (under the same conditions) guarantees not only an approximate decomposition, but a Hamilton decomposition.

\begin{theorem} \label{decomp}
For every $ \alpha >0$ there exists $\tau>0$ such that for every $\nu > 0$ there exists $n_0=n_0 (\alpha,\nu,\tau)$ 
for which the following holds. Suppose that
\begin{itemize}
\item[{\rm (i)}] $G$ is an $r$-regular digraph on $n \ge n_0$ vertices, where $r\ge \alpha n$;
\item[{\rm (ii)}] $G$ is a robust $(\nu,\tau)$-outexpander.
\end{itemize}
Then $G$ has a Hamilton decomposition.
Moreover, this decomposition can be found in time polynomial in $n$.
\end{theorem}

So as a special case, Theorem~\ref{decomp} implies that Kelly's conjecture holds for all sufficiently large regular tournaments.
It also implies a conjecture of Erd\H{o}s on Hamilton decompositions of regular tournaments.
However, it turns out that the notion of robust (out)expansion extends far beyond the class of tournaments and many further applications 
of Theorem~\ref{main} are explored by K\"uhn and Osthus in~\cite{KellyII}.
For example, the notion of robust expansion can be extended to undirected graphs in a natural way
and one can deduce a version of Theorem~\ref{decomp} for undirected graphs.
In~\cite{KellyII} this in turn is used to prove an approximate version of a conjecture of Nash-Williams on Hamilton decompositions of dense regular graphs.
Random regular graphs of linear degree as well as $(n,d,\lambda)$-graphs (for appropriate values of these parameters) are further examples of robustly expanding graphs.
In combination with a result of Gutin and Yeo~\cite{GutinYeo}, Theorem~\ref{decomp} can also be used to solve a problem of Glover and Punnen~\cite{GP} as well as Alon, Gutin and Krivelevich~\cite{AGK} on TSP tour domination (see~\cite{KOKelly} for details).
For this application, it is crucial that the Hamilton decomposition can be found in polynomial time. 

Roughly speaking, the argument leading to Theorem~\ref{decomp} uses  Theorem~\ref{main} in the following way:
let $G$ be a robustly expanding digraph. The first step is to remove a `robustly decomposable' spanning regular digraph $H$ from $G$ to obtain $G'$.
$H$ will be sparse compared to $G$ and will have the property that it has a Hamilton decomposition even if we add the edges of a digraph $H'$,
which is very sparse compared to $H$ and also regular (on the same vertex set) but otherwise arbitrary.
Now $G'$ is still a robust outexpander, so one can apply Theorem~\ref{main} to $G'$ obtain an approximate Hamilton decomposition of $G'$.
Let $H'$ denote the set of edges not contained in any of the Hamilton cycles of this approximate decomposition of $G'$.
Then the fact that $H$ is robustly decomposable implies that $H \cup H'$ has a Hamilton decomposition.
Together with the approximate decomposition of $G'$, this yields a Hamilton decomposition of the entire digraph $G$.
Note that the above approach means that for Theorem~\ref{decomp} to be algorithmic, one needs Theorem~\ref{main} to be algorithmic too.

This paper is organised as follows.
In the next section, we give a brief outline of the argument.
We then collect the necessary tools in Section~\ref{sec:regularity} (which is mostly concerned with Szemer\'edi's regularity lemma)
and Section~\ref{sec:tools} (which mainly collects properties of robust outexpanders).
We then prove Theorem~\ref{main} in Section~\ref{sec:mainproof}.
In Section~\ref{strengthening}, we deduce Corollary~\ref{cor} from Theorem~\ref{main}.

%%%%%%%%%%%%%%%%%%%%%%%%%%%%%%%%%%%%%%%%%%%%%%%%%%%%%%%%%%%%%%%%%%%%%%%%%%%%%%%%

\section{Sketch of the proof of Theorem~\ref{main}} \label{sec:sketch}

Roughly speaking, the strategy of the proof of Theorem~\ref{main} is the following.
Suppose that a digraph $G$ satisfies the conditions of Theorem~\ref{main}. 
First remove the edges of a carefully chosen spanning sparse subdigraph $H$ from $G$
and let $G'$ consist of the remaining edges of $G$.
Next, find an approximate decomposition of $G'$ into edge-disjoint $1$-factors $F_i$
(where a $1$-factor is a spanning union of vertex-disjoint cycles).
Finally, the aim is to transform each $F_i$ into a Hamilton cycle by removing
some of its edges and adding some edges of $H$.
One immediate obstacle to a na\"ive implementation of this approach is that the $F_i$
might consist of many cycles, so turning each of them into 
a Hamilton cycle might require more edges from $H$ than one can afford.
In~\cite{fk,kot}, this was overcome (loosely speaking) by choosing the $1$-factors
$F_i$ randomly. It turns out that this has the advantage
that the $F_i$ will have few cycles, i.e.~they are already close to being Hamilton
cycles. One disadvantage is that this approach is inherently non-algorithmic
(and does not seem derandomisable).

A second problem is how to make sure that $H$ contains the edges that are required
to transform each $F_i$ into a Hamilton cycle.
We overcome this by choosing $H$ and the $1$-factors $F_i$ according to the vertex
partition of $G$ obtained from Szemer\'edi's regularity lemma.
More precisely, we apply the regularity lemma to partition $G$ into clusters $V_1
, \ldots , V_L$ of vertices such that almost all ordered pairs of clusters
induce a pseudorandom subdigraph of $G$, together with a small (but typically
troublesome) exceptional set $V_0$. 
We define the `reduced multidigraph' $R(\beta)$ whose vertices are the clusters
$V_j$ with (multiple) edges from $V_j$ to $V_k$ 
if the corresponding subdigraph of $G$ is pseudorandom and dense. Here the number of
edges from $V_j$ to $V_k$ is proportional to the density of $G[V_j , V_k]$. 
So each edge of $R(\beta)$ corresponds to a bipartite pseudorandom digraph between the
corresponding pair of clusters in $G$ (where all these pseudorandom digraphs have
same density $\beta$). 
$R(\beta)$ inherits many of the properties of $G$, in particular it is an almost
regular robust outexpander with large minimum semidegree. 

The next step is to use the Max-Flow-Min-Cut Theorem to find a spanning regular subdigraph of
$R(\beta)$ which contains almost all edges of $R(\beta)$. 
We can now (arbitrarily) partition this regular subdigraph into a collection of edge-disjoint
1-factors ${F}_i$ of $R(\beta)$ (see Section~\ref{begin}).
Each of the ${F}_i$  corresponds to a vertex-disjoint collection of `blown-up'
cycles which spans most of $V(G)$.
We will denote each of these collections by $G_i$ and call $G_i$ the $i$th slice of
$G$. Note that the $G_i$ are all edge-disjoint.

Roughly speaking, the aim is to add a small number of edges (which do not lie in any
of the other slices)
to each $G_i$ to transform $G_i$ into a regular digraph which has an approximate
Hamilton decomposition.
Together, these approximate Hamilton decompositions of the slices then yield an approximate Hamilton decomposition of
$G$.
In Section~\ref{sec:H}, we put aside three sparse subdigraphs $H_0, H_1, H_2$ which we
will use to add the required edges to each $G_i$.
So together, $H_0$, $H_1$ and $H_2$ play the role of the digraph $H$ mentioned earlier. 

So far we have ignored the exceptional vertices, but to obtain a regular spanning
subdigraph we need to incorporate them into each slice $G_i$.
For convenience, we call any exceptional vertex $x \in V_0$ and each edge incident
with $V_0$ `red'.
In Sections~\ref{exceptional} and~\ref{exceptional2}, we will add red edges to each $G_i$
in such a way that the resulting slice $G_i$ is almost regular and only a small part
of each cluster is incident to any red edges. 
Some of these edges come from $H_1$ and the others will be edges of $G$ which are
not contained in any of the $H_j$ or any of the $G_i$ constructed so far. 

Together with these red edges, each $G_i$ is now an almost regular digraph consisting
mainly of a union of blown-up cycles. 
On the other hand, $G_i$ may not even be connected. But to guarantee many edge-disjoint Hamilton
cycles in $G_i$, we clearly need to have sufficiently many edge-disjoint 
paths between these blown-up cycles. For this, we define an ordering of the cycles
$D_1,\dots,D_\ell$ of $F_i$ and specify `bridge vertices' $x_{i,j}$ (one for each
successive pair of cycles)
so that $x_{i,j}$ has many inneighbours in $D_j$ and many outneighbours in
$D_{j+1}$. We find the edges incident to these $x_{i,j}$ within $H_0$ (see
Section~\ref{connect}).

We would now like to find a spanning regular subdigraph in each $G_i$ whose degree is
almost as large as that of $G_i$.
Trivially, this regular subdigraph would then have a decomposition into $1$-factors.  
However, as we have little control over the red edges added so far, they may prevent
us from finding a regular subdigraph
(see Section \ref{shadow} for a discussion and an example).
For this reason, we add extra (red) edges to $G_i$ from $H_2$ to balance out the
existing red edges. 
In this way, we can ensure that for each cluster $V$ of a blown-up cycle $D$, the
number of edges leaving $V$ in $G_i$ 
equals the number of edges entering its successor $V^+$ on $D$. This is achieved in
Sections~\ref{shadow} and \ref{balance},
by considering an auxiliary reduced digraph $R^*$ which also turns out to be a
robust outexpander
(the latter property is crucial here).

As indicated above, in Section~\ref{aldecomp}, we can now
find a spanning $\kappa$-regular subdigraph $G_i^*$ of each $G_i$ (for a suitable $\kappa$). We now decompose each $G_i^*$
into
$1$-factors $f_{i,1},\dots,f_{i,\kappa}$.
Our aim is to transform each $f_{i,j}$ into a Hamilton cycle by adding and removing
a few edges. 
The edges we add will be taken from a very sparse digraph $H_{3,i}$ which we removed
from $G_i$ earlier
(so $H_{3,i}$ can also be viewed as a union of blown-up cycles).
The key point of the proof is that we can achieve this transformation by using a very small number
of edges from $H_{3,i}$ for each $f_{i,j}$.
The reason for this is that we can guarantee that the red edges added in the course
of the proof are `localised' within each $G_i$,
i.e.~on each blown-up cycle of each $F_i$ there are long intervals of clusters
which are not incident to any red edges.
This means that for each $1$-factor $f_{i,j}$, its subdigraph induced by any such
interval $I$ consists of long paths.
If some of these paths lie on different cycles of $f_{i,j}$, we can merge these into
a single cycle by adding and removing edges of $H_{3,i}$
which are induced by just a single pair of consecutive clusters on $I$.
Crucially, this enables us to use the bipartite subdigraphs of $H_{3,i}$ induced by
other pairs of consecutive clusters on $I$ to transform  
other $1$-factors $f_{i,j'}$ of the slice $G_i$.
Repeating this process until we have merged all cycles of $f_{i,j}$ into a single
cycle eventually  transforms the $f_{i,j}$ into 
$\kappa$ edge-disjoint Hamilton cycles, as required (see Lemma~\ref{trick} and Section~\ref{mergeH}).

An approach based on the regularity lemma was already used in~\cite{kot}. However, as
mentioned earlier, the argument there relied on a random choice of the
$1$-factors, which did not translate into an algorithm. This problem is overcome by
the above `localisation' idea, which automatically produces $1$-factors
which are `well behaved' with respect to red edges in the sense described above.
However, this `localisation property' is quite difficult to achieve and relies on
additional ideas such as a refinement of the original regularity partition and a special `unwinding' of blown-up cycles
(see Section~\ref{unwinding} and Lemma~\ref{unwind}).

\section{Notation and the diregularity lemma} \label{sec:regularity}

\subsection{Notation}

Throughout we will omit floors and ceilings where the argument is unaffected. The constants in the hierarchies used to state our results are chosen from right to left.
For example, if we claim that a result holds whenever $0<1/n\ll a\ll b\ll c\le 1$ (where $n$ is the order of the graph or digraph), then 
there are non-decreasing functions $f:(0,1]\to (0,1]$, $g:(0,1]\to (0,1]$ and $h:(0,1]\to (0,1]$ such that the result holds
for all $0<a,b,c\le 1$ and all $n\in \mathbb{N}$ with $b\le f(c)$, $a\le g(b)$ and $1/n\le h(a)$. 
Hierarchies with more constants are defined in a similar way. 
Note that $a \ll b$ implies that we may assume in the proof that e.g. $a < b$ or $a < b^2$.
We write $a = b \pm \eps$ for $a \in [b - \eps , b + \eps]$. 

For an undirected graph $G$ containing a vertex $x$ we write $N_G(x)$ for the neighbourhood of $x$ and $d_G(x)$ for its degree. For a digraph $G$ we write $xy$ for the edge directed from $x$ to $y$ and write $N^+_G(x)$ for the \emph{outneighbourhood}, the set of vertices receiving an edge from $x$, and write $d^+_G(x) := \vert N_G^+(x) \vert$ for the \emph{outdegree} of $x$. We define the \emph{inneighbourhood} $N^-_G(x)$ and \emph{indegree} $d^-_G(x)$ similarly. 
For a collection of vertices $U \subseteq V(G)$ we write $d_G^+(U)$ for the total number of edges sent out by the vertices in $U$. We define $d^-_G(U)$ analogously.
We will omit the $G$ subscript in the above and in similar situations elsewhere if this is unambiguous.
Denote the minimum outdegree by $\delta^+(G)$ and the minimum indegree by $\delta^-(G)$. Let the \emph{minimum semidegree} $\delta^0(G)$ be the minimum of $\delta^+(G)$ and $\delta^-(G)$. Denote the maximum outdegree by $\Delta^+(G)$ and define $\Delta^-(G)$ and analogously. Let $\Delta^0(G)$ denote the maximum of $\Delta^+(G)$ and $\Delta^-(G)$. If $G$ is a multidigraph then neighbourhoods are multisets. For any positive integer $r$, an $r$-\emph{regular} digraph on $n$ vertices is such that every vertex has exactly $r$ outneighbours and $r$ inneighbours. 
A \emph{1-factor} of a multidigraph $G$ is a 1-regular spanning digraph; that is, a collection of vertex-disjoint cycles that together contain all the vertices of $G$.

If $G$ is a multidigraph and $U \subseteq V(G)$, we write $G[U]$ for the sub-multidigraph of $G$ \emph{induced} by $U$. That is, the digraph with vertex set $U$ and edge set obtained from $E(G)$ by including only those edges with both endpoints contained in $U$.
If $G[U]$ has empty edge set, we say that $U$ is an \emph{isolated} subset of $G$.
If $G$ is a digraph and $U \subseteq V(G)$ we write $G \setminus U$ for the digraph with vertex set $V(G) \setminus U$ and edge set obtained from $E(G)$ by deleting all edges incident to a vertex of $U$.

Given a digraph $R$ and a positive integer $r$, the \emph{$r$-fold blow-up} $r \otimes R$ of $R$ is the digraph obtained from $R$ by replacing every vertex $x$ of $R$ by $r$ vertices and replacing every edge $xy$ of $R$ by the complete bipartite graph $K_{r,r}$ between the two sets of $r$ vertices corresponding to $x$ and $y$ such that all the edges of $K_{r,r}$ are oriented towards the $r$ vertices corresponding to $y$. 
We say that any edge in this $K_{r,r}$ is \emph{contained in the blow-up of} $xy$. 
Now consider the case when $V_1, \ldots, V_k$ is a partition of some set $V$ of vertices and $R$ is a digraph whose vertices are $V_1, \ldots, V_k$. If $R$ is a directed cycle, say $R = C = V_1 \ldots V_k$, and $G$ is a digraph with $V(G) \subseteq V = V_1 \cup \ldots \cup V_k$, we say that \emph{(the edges of) $G$ wind(s) around $C$} if, for every edge $xy$ of $G$, there exists an index $j$ such that $x \in V_j$ and $y \in V_{j+1}$.

\subsection{A Chernoff bound and its derandomisation}
In the proof of Claims~\ref{claim1} and~\ref{claim2}, we will use the following standard Chernoff type bound (see e.g.~Corollary 2.3 in~\cite{JLR} and Theorem~2.2 in~\cite{SrSt}).

\begin{proposition} \label{chernoff} 
Suppose $X$ has binomial distribution and $0<a<1$. Then
$$
\mathbb{P}(X  \ge (1+a)\mathbb{E}X) \le e^{-\frac{a^2}{3}\mathbb{E}X}
\mbox{ and } \mathbb{P}(X  \le (1-a)\mathbb{E}X) \le e^{-\frac{a^2}{3}\mathbb{E}X}.
$$
\end{proposition}

To obtain an algorithmic version of Theorem~\ref{main}, we need to `derandomise' our applications of Proposition~\ref{chernoff}.
This can be done via the well known `method of conditional probabilities.'
The following result of Srivastav and Stangier (Theorem~2.10 in~\cite{SrSt}) provides a convenient way to apply this method.
It implies that any construction based on a polynomial number of applications of Proposition~\ref{chernoff} can
be derandomised to provide a polynomial time algorithm.

Suppose we are given $N$ independent $0/1$ random variables $X_1,\dots,X_N$ where $\mathbb{P} (X_j=1)= p$ and $\mathbb{P} (X_j =0)=1-p$
for some rational $0 \le p \le 1$. Suppose that $1 \le i \le m$. Let $w_{ij} \in \{0,1\}$.
Denote by $\phi_i$ the random variables $\phi_i:=\sum_{j=1}^N w_{ij}X_j$. Fix $\beta_i$ with $0 < \beta_i < 1$.
Now let $E_i^+$ denote the event that $\phi_i \ge (1+\beta_i)\ex [ \phi_i]$ and 
let $E_i^-$ denote the event that $\phi_i \le (1-\beta_i)\ex [ \phi_i]$.
Let $E_i$ be  either $E_i^+$ or $E_i^-$. 

\begin{theorem}{\cite{SrSt}} \label{derandom}
Let $E_1,\dots,E_m$ be events such that
\begin{equation*} \label{eq:3.3}
\sum_{i=1}^m e^{- \beta^2_i \ex(\phi_i)/3} \le 1/2.
\end{equation*}
Then $$\mathbb{P} \left(\bigcap_{i=1}^m E_i\right) \ge 1/2$$
and a vector $x \in  \bigcap_{i=1}^m E_i$ can be constructed in time $O (m N^2 \log (mN) )$.
\end{theorem}

In general, it will usually be clear that the proofs can be translated into polynomial time algorithms.  
We do not prove an explicit bound on the time needed to find the 
set of edge-disjoint Hamilton cycles guaranteed by Theorem~\ref{main}, apart from the fact that the time is polynomial in $n$.

\subsection{The diregularity lemma} 

We will use the directed version of Szemer\'edi's regularity lemma. To state it we need some definitions. We write $d_G(A,B)$ for 
the \emph{density} $\frac{e_G(A,B)}{\vert A \vert \vert B \vert}$ of an undirected bipartite graph $G$ with vertex classes $A$ and $B$. 
Given $\eps > 0$ we say that $G$ is 
$\eps$\emph{-regular} if every $X \subseteq A$ and $Y \subseteq B$ with $\vert X \vert \geq \eps \vert A \vert$ 
and $\vert Y \vert \geq \eps \vert B \vert$ satisfy $\vert d(A,B) - d(X,Y) \vert \leq \eps$.  
Given $\eps, d \in (0,1)$ we say that $G$ is $(\eps, d)$-regular if $G$ is 
$\eps$-regular and $d_G(A,B) = d \pm \eps$. We say that $G$ is $(\eps , d )$\emph{-superregular} if both of the following hold:
\begin{itemize} 
\item $G$ is $(\eps,d)$-regular;
\item $d(a) =(d \pm \eps) \vert B \vert$, $d(b) =(d \pm \eps) \vert A \vert$ for all $a \in A$, $b \in B$.
\end{itemize}
Given disjoint vertex sets $A$ and $B$ in a digraph $G$, write $(A,B)_G$ for the oriented bipartite subgraph of $G$ whose vertex classes are $A$ and $B$ and whose edges are all those from $A$ to $B$ in $G$. We say that $(A,B)_G$ has any of the regularity properties above if the requirements hold for the underlying undirected bipartite graph of $(A,B)_G$.

The diregularity lemma is a variant of the regularity lemma for digraphs due to Alon and Shapira~\cite{AS}. 
We will use the degree form which can be derived from the standard version in the same manner as the undirected degree form. 
The proof of the diregularity lemma itself is similar to the undirected version.

\begin{lemma} \label{direg}
\emph{(Degree form of the diregularity lemma)} For every $\eps \in (0,1)$ and every integer $M^\prime$ there are integers $M$ and $n_0$ such that if $G$ is a digraph on $n \geq n_0$ vertices and $d \in [0,1]$ is any real number, then there is a partition of the vertex set of $G$ into $V_0 , \ldots , V_L$ and a spanning subdigraph $G^\prime$ of $G$ such that the following holds:
\begin{itemize}
\item $M^\prime \leq L \leq M$;
\item $\vert V_0 \vert \leq \eps n$;
\item $\vert V_1 \vert = \ldots = \vert V_L \vert =: m$;
\item $d^+_{G^\prime}(x) > d^+_{G}(x) - (d + \eps) n$ and $d^-_{G^\prime}(x) > d^-_{G}(x) - (d + \eps) n$ for all vertices $x \in V(G)$;
\item For all $1 \leq i \leq L$ the digraph $G^\prime[V_i]$ is empty;
\item For all $1 \leq j \leq L$ with $i \neq j$ the pair $(V_i , V_j)_{G^\prime}$ is $\eps$-regular and has density either 0 or at least $d$.
\end{itemize}
\end{lemma}

We call $V_1 , \ldots , V_L$ \emph{clusters}, $V_0$ the \emph{exceptional set} and the vertices in $V_0$ \emph{exceptional vertices}. We refer to $G^\prime$ as the \emph{pure digraph}. The last condition of the lemma says that all pairs of clusters are $\eps$-regular in both directions (but possibly with different densities). The \emph{reduced digraph} $R$ of $G$ with parameters $\eps$, $d$ and $M^\prime$ is the digraph whose vertices are $V_1 , \ldots, V_L$ and in which $V_i V_j$ is an edge precisely when $(V_i , V_j)_{G^\prime}$ is $\eps$-regular and has density at least $d$. For each edge $V_i V_j$ of $G$ we write $d_{ij}$ for the density of $(V_i , V_j)_{G^\prime}$. Suppose $0 < 1/M^\prime \ll \eps \ll \beta \ll d \ll 1$. The \emph{reduced multidigraph} $R(\beta)$ of $G$ with parameters $\eps, \beta, d, M^\prime$ is obtained from $R$ by setting $V(R(\beta)) := V(R)$ and adding $\lfloor d_{ij} / \beta \rfloor$ directed edges from $V_i$ to $V_j$ whenever $V_i V_j
  \in E(R)$. These digraphs inherit some of the key properties of $G$, as the next few results show (which are variants of well known observations, see e.g. Lemma 11 in \cite{kot} for the next result).

\begin{lemma} \label{reduced}
Let $0 < 1/n_0 \ll 1/M^\prime \ll \eps \ll \beta \ll d \leq d^\prime \ll c_1 \leq c_2 < 1$ and let $G$ be a digraph of order $n \geq n_0$ with $\delta^0(G) \geq c_1 n$ and $\Delta^0(G) \leq c_2 n$. Apply Lemma~\ref{direg} with parameters $\eps, d$ and $M^\prime$ to obtain a pure digraph $G^\prime$ and a reduced digraph $R$ of $G$ and let $R^\prime$ denote the subdigraph of $R$ whose edges correspond to pairs of density at least $d^\prime$. Let $R(\beta)$ denote the reduced multidigraph of $G$ with parameters $\eps , \beta, d$ and $M^\prime$ and let $R^\prime(\beta)$ be the multidigraph obtained from $R(\beta)$ by including only those edges which also correspond to an edge of $R^\prime$. Let $L := \vert R \vert = \vert R(\beta) \vert$. Then
\begin{enumerate}
\item[(i)] $\delta^0(R^\prime) \geq (c_1 - 3d^\prime)L$.
\item[(ii)] $\delta^0(R'(\beta)) \geq (c_1 - 4d') \dfrac{L}{\beta}\ \ \mbox{ and }\ \ \Delta^0(R'(\beta)) \leq (c_2 + 2 \eps) \dfrac{L}{\beta}$.
\end{enumerate}
\end{lemma}

\begin{proof}
To prove (i), we consider the weighted digraph $R^\prime_w$ obtained from $R^\prime$ by giving each edge $V_i V_j$ of $R^\prime$ weight $d_{ij}$. Given a cluster $V_i$, we write $w^+(V_i)$ for the sum of the weights of all edges sent out by $V_i$ in $R^\prime_w$. We define $w^-(V_i)$ similarly and write $w^0(R_w^\prime)$ for the minimum of $\min \lbrace w^+(V_i) , w^-(V_i) \rbrace$ over all clusters $V_i$. Note that $\delta^0(R^\prime) \geq w^0(R^\prime_w)$. Moreover, Lemma~\ref{direg} implies that $d^\pm_{G^\prime \setminus V_0}(x) > (c_1 - 2d)n$ for all $x \in V(G^\prime \setminus V_0)$. Thus each $V_i \in V(R^\prime)$ satisfies
\[
(c_1 - 2d)nm \leq e_{G^\prime}(V_i , V(G^\prime) \setminus V_0) \leq m^2 w^+(V_i) + (d^\prime m^2)L
\]
and so $w^+(V_i) \geq (c_1 - 2d - d^\prime)L \geq (c_1 - 3d^\prime)L$. Arguing in the same way for inweights gives us $\delta^0(R^\prime) \geq w^0(R^\prime_w) \geq (c_1 - 3d^\prime)L$. We can deduce the first part of (ii) by noting that
$$
d_{R'(\beta)}^+(V_i) = \sum\limits_{V_j \in N_{R'}^+(V_i)}{\lfloor d_{ij}/\beta \rfloor} \geq w^+(V_i)/\beta - L > (c_1 - 4d')\frac{L}{\beta}.
$$ 
Similar arguments can be used to show the remaining bounds.%
\COMMENT{from Andrew's paper}
\end{proof}

\begin{lemma} \label{inherit}
Let $M^\prime , n_0$ be positive integers and let $\eps , d, \nu, \tau$ be positive constants such that $1/ n_0 \ll 1/M' \ll \eps \ll d \leq d^\prime \leq \nu \leq \tau < 1$ and $d^\prime \leq \nu/20$. Let $G$ be a digraph on $n \geq n_0$ vertices such that $G$ is a robust $(\nu , \tau)$-outexpander. Let $R$ be the reduced digraph of $G$ with parameters $\eps, d$ and $M^\prime$ with clusters of size $m$ and let $R^\prime$ be the subdigraph of $R$ whose edges correspond to pairs of density at least $d^\prime$. Then 
$R^\prime$ is a robust $(\nu/4 , 3 \tau)$-outexpander.
\end{lemma}

\begin{proof}
Let $G^\prime$ denote the pure digraph, $L := \vert V(R) \vert$, and $V_1 , \ldots , V_L$ be the clusters of $G$, and $V_0$ the exceptional set. Let $m := \vert V_1 \vert = \ldots = \vert V_L \vert$. Suppose $S \subseteq V(R^\prime)$ has $3 \tau L \leq \vert S \vert \leq (1- 3\tau) L$. Let $S_G$ denote the union of all vertices in clusters in $S$. So $S_G \subseteq V(G)$ and $2 \tau n \leq \vert S_G \vert \leq (1-2\tau)n$. For every $x \in RN^+_{\nu , G}(S_G)$ we have that $\vert N^-_{G^\prime}(x) \cap S_G \vert \geq \vert N_G^-(x) \cap S_G \vert - (d+\eps)n \geq (\nu - d - \eps)n \geq \nu n/2$.
This implies that
\[
\vert RN^+_{\nu/2 , G^\prime}(S_G) \vert \geq \vert RN_{\nu,G}^+(S_G) \vert \geq \vert S_G \vert + \nu n \geq \vert S \vert m + \nu L m
\]
and every vertex $x \in RN_{\nu/2 , G^\prime}^+(S_G)$ has at least $\nu n/2$ inneighbours in $S_G$. Suppose, for a contradiction, that $\vert RN_{\nu/4 , R^\prime}^+(S) \vert < \vert S \vert + \nu L/4$. Let $RN^\prime$ denote the union of all vertices in clusters in $RN^+_{\nu/4,R^\prime}(S)$ and let $T := RN_{\nu /2, G^\prime}^+(S_G) \setminus RN^\prime$; then $\vert T \vert \geq \nu n/4$. 

Note that by definition, for all $V$ outside $RN^+_{\nu/4,R'}(S)$, there exists a collection $\mathcal{V}$
of at least $|S| - \nu L/4$ clusters $U \in S$ so that there is no edge from $U$ to $V$ in $R'$.
So by assumption such a $\mathcal{V}$ exists for any $V$ which has non-empty intersection with $T$.

We say that a vertex $x \in V$ is \emph{bad} if it has indegree at least $2 d^\prime m$ in at least $\sqrt{\eps} L$ of the clusters in $\mathcal{V}$. 
The final property of Lemma~\ref{direg} implies that
there are at most $\eps m$ vertices in $V$ that have indegree at least $2 d^\prime m$ in some fixed cluster of $\mathcal{V}$.
So by double counting the number of such vertex-cluster pairs, we see that any cluster contains at most $\sqrt{\eps} m$ bad vertices.

Say that a cluster $V$ is \emph{significant} if $\vert V \cap T \vert  \geq \eps^{1/3} m$.%
\COMMENT{$\nu n/4 \leq \eps^{1/3} m \times no. non-significant clusters + m \times no. significant clusters$.} Then there are at least $\nu L/5$ significant clusters and we write $V^\prime := V \cap T$. Consider any $x \in V^\prime$, where $V$ is significant. We say that a cluster $U$ in $S$ is \emph{rich for x} if $x$ has at least $\nu m/10$ inneighbours in $U$. Since $x$ has at least $\nu n/2$ inneighbours in $S_G$, there are at least $\nu L/3$ clusters in $S$ which are rich for $x$.
So there are at least $\nu L/12 \geq \sqrt{\eps} L$ clusters in $\mathcal{V}$ which are rich for $x$.%
 \COMMENT{$\nu n/2 \leq no. bad clusters \times \nu m/10 + no. good clusters \times m$} 
Since $d^\prime \leq \nu/20$, this means that every $x$ in $V'$ is bad.
Thus $V$ contains at least $\eps^{1/3} m$ bad vertices, a contradiction.
\end{proof}

The following simple observation is well known, the version given here is proved as Proposition~4.3(i) and (iii) in~\cite{KOKelly}.

\begin{proposition} \label{superslice} 
Suppose that $0<1/m \ll \eps \le d' \le d \ll 1$. Let $G$ be a bipartite graph with
vertex classes
$A$ and $B$ of size $m$. Suppose that $G'$ is obtained from $G$ by removing at most
$d'm$
vertices from each vertex class and at most $d'm$ edges incident to each vertex from
$G$.
\begin{itemize}
\item[(i)]
If $G$ is $(\eps,d)$-regular then $G'$ is $(2\sqrt{d'},d)$-regular.
\item[(ii)] If $G$ is $(\eps,d)$-superregular then $G'$ is $(2\sqrt{d'},d)$-superregular.
\end{itemize}
\end{proposition}

The next result 
%(similar to Lemma 4.10(iii) and (iv) in~\cite{KOKelly})
shows that we can 
partition an $\eps$-(super)regular pair into edge-disjoint $\eps^\prime$-(super)regular spanning subgraphs.
%remove an $\eps^\prime$-(super)regular spanning subgraph from an $\eps$-(super)regular pair. 
The proof is almost identical to that of Lemma 4.10(iii) and (iv) in \cite{KOKelly} (which covers the case $K=2$) so we omit it here.

\begin{lemma}\label{randomregslice}
Let $K$ be an integer and let $0<1/m\ll \eps \ll \gamma_1 , \ldots, \gamma_K \ll 1$ such that $\gamma_1 + \ldots + \gamma_K \leq d \leq 1$.
\begin{itemize}
\item[{\rm (i)}] If $G$ is an $(\eps,d)$-regular bipartite graph with vertex classes $X,Y$ of size $m$, then it contains $K$ edge-disjoint spanning subgraphs $J_1, \ldots, J_K$ such that for each $1 \leq k \leq K$ we have that $J_k$ is $(\eps^{1/12},\gamma_k)$-regular.
Moreover, if $x \in X$ satisfies $d_G(x) =(d \pm \eps)m$, then $d_{J_k}(x) =(\gamma_k \pm \eps^{1/12})m$ for each $1 \leq k \leq K$.
\item[{\rm (ii)}] If $G$ is an $(\eps,d)$-superregular bipartite graph with vertex classes of size $m$, then it contains $K$ edge-disjoint spanning subgraphs $J_1, \ldots, J_K$ such that for each $1 \leq k \leq K$ we have that $J_k$ is $(\eps^{1/12},\gamma_k)$-superregular.
\end{itemize}
Moreover, the spanning subgraphs can be found in time polynomial in $m$.

\end{lemma}%
\COMMENT{in Section~\ref{begin} we apply this lemma with $d \approx d_{ij}$ for each $1 \leq i < j \leq L$. We have $d_{ij} \geq d$ but $d_{ij}$ could be as large as 1. We need each $\gamma_k$ to be small but the sum of the $\gamma_k$s could be as large as one.}

The proof of Theorem~\ref{main} begins by decomposing our digraph into `blown-up' 1-factors and we will need the following well known and 
easy fact that allows us to extract almost spanning blown-up 1-factors in which pairs are superregular.

\begin{lemma} \label{supreg}
Let $0 < \eps \leq \gamma \leq 1 \leq m$ and 
let $D$ be a digraph with vertex clusters $V_1, \ldots, V_k$ each of size $m$ such that $(V_j,V_{j+1})_D$ is $(\eps,\gamma)$-regular for $1 \leq j \leq k$, where $V_{k+1} := V_1$.
Then there exists a subdigraph $D'$ of $D$ with vertex clusters $V_1', \ldots, V_k'$ where $V_j' \subseteq V_j$, $|V_j'| = (1-2\eps)m$ and $(V_j',V_{j+1}')_{D'}$ is $(4\eps,\gamma)$-superregular for $1 \leq j \leq k$, where $V_{k+1}' := V_1'$.
\end{lemma}

\begin{proof}
For each $1 \leq j \leq k$, each $V_i$ contains at most $2\eps m$ vertices whose outdegree or indegree in $D$ is 
either at most $(\gamma - 2\eps)m$ or at least $(\gamma + 2\eps)m$. Deleting exactly $2\eps m$ vertices including these from each cluster gives us $D^\prime$.
\end{proof}

We will use the following crude version of the fact that every $\eps$-regular pair contains a subgraph of given maximum degree $\Delta$ whose average degree is close to $\Delta$,
which is Lemma~13 in~\cite{kot}.
\begin{lemma} \label{embed}
Suppose that $0 < 1/m \ll \eps^\prime, \eps \ll d_0 \leq d_1 \ll 1$ and that $(A,B)$ is an $(\eps,d_1)$-regular pair with $m$ vertices in each class. Then $(A,B)$ contains a subgraph $H$ whose maximum degree is at most $d_0 m$ and whose average degree is at least $d_0 m/8$.
\end{lemma}

The proof proceeds by greedily removing matchings and so $H$ can be found in polynomial time.
Part (ii) of the following observation is proved as Lemma 5.3 in~\cite{KOKelly}; (i) is immediate from the definition.
\begin{lemma}\label{expanderblowup}
Let $r\ge 3$ and let $G$ be a robust $(\nu,\tau)$-outexpander with $0<3\nu\le
\tau<1$. Let $G'$ be the $r$-fold blow-up of
$G$. Then
\begin{itemize}
\item[(i)] $\delta^0(G') = r\delta^0(G)$.
\item[(ii)] $G'$ is a robust $(\nu^3,2\tau)$-outexpander.
\end{itemize}
\end{lemma}

\subsection{Uniform refinements} \label{unirefs}

We will also need to partition each vertex cluster into equal parts in such a way that the in- and outneighbourhood of each vertex restricted to each part is roughly the size we expect it to be.
This is very similar to Lemma 4.7 in~\cite{KOKelly}. To state the result, we need the following definitions.
Let $G$ be a digraph and let $\mathcal{P}$ be a partition of $V(G)$ into an exceptional set $V_0$ and clusters of equal size. Suppose that $\mathcal{P}'$ 
is another partition of $V(G)$ into an exceptional set $V'_0$ and clusters of equal size. We say that $\mathcal{P'}$ is an \emph{$\ell$-refinement of $\mathcal{P}$} 
if $V_0=V'_0$ and if the clusters in $\mathcal{P}'$ are obtained by partitioning each cluster in $\mathcal{P}$ into $\ell$ subclusters of equal size. 
(So if $\mathcal{P}$ contains $k$ clusters then $\mathcal{P}'$ contains $k\ell$ clusters.) $\mathcal{P}'$ is an \emph{$\eps$-uniform $\ell$-refinement of} $\mathcal{P}$ \emph{with respect to} $G$ if it is an $\ell$-refinement of $\mathcal{P}$ which satisfies the following condition:
\begin{itemize}
\item[(URef)] Whenever $x$ is a vertex of $G$, $V$ is a cluster in $\mathcal{P}$ and $|N^+_G(x)\cap V|\ge \eps |V|$ then $|N^+_G(x)\cap V'|=(1\pm \eps)|N^+_G(x)\cap V|/\ell$ for each cluster $V'\in \mathcal{P}'$ with $V'\subseteq V$. The inneighbourhoods of the vertices of $G$ satisfy an analogous condition.
\end{itemize}
Let $\mathcal{G}$ be a collection of digraphs on the same vertex set. If $\mathcal{P}$ is a refinement with respect to $G$ for all $G \in \mathcal{G}$ then we say that it is a refinement \emph{with respect to} $\mathcal{G}$.

\begin{lemma} \label{equipartition}
Suppose that $0<1/m \ll 1/k,\eps \ll \eps', d,1/\ell, 1/t \le 1$ and $m/\ell\in\mathbb{N}$. Suppose that $\mathcal{G}$ is a collection of $t$ digraphs on the same set $V^*$ of $n\le 2km$ vertices and that $\mathcal{P}$ is a partition of $V^*$ into an exceptional set $V_0$ and $k$ clusters of size $m$. Then there exists an $\eps$-uniform $\ell$-refinement of $\mathcal{P}$ with respect to $\mathcal{G}$. Moreover, any $\eps$-uniform $\ell$-refinement $\mathcal{P}'$ of $\mathcal{P}$ automatically satisfies the following conditions for all $G \in \mathcal{G}$:
\begin{itemize}

\item[(i)] Suppose that $V$, $W$ are clusters in $\mathcal{P}$ and $V',W'$ are clusters in $\mathcal{P}'$ with $V'\subseteq V$ and $W'\subseteq W$. If $G[V,W]$ is $(\eps,d)$-superregular then $G[V',W']$ is $(\eps',d)$-superregular.
\item[(ii)] Suppose that $V$, $W$ are clusters in $\mathcal{P}$ and $V',W'$ are clusters in $\mathcal{P}'$ with $V'\subseteq V$ and $W'\subseteq W$. If $G[V,W]$ is $(\eps, d)$-regular  then $G[V',W']$ is $(\eps', d)$-regular.
\end{itemize}
\end{lemma}

The proof proceeds by considering a random partition of $V^*$ (which can be derandomised by Theorem~\ref{derandom}). We omit the proof as it is almost the same as Lemma 4.7 in \cite{KOKelly}.

Let $\eps>0$ and let $\mathcal{P}$ be a partition of $V(G)$ into an exceptional set $V_0$ and clusters of size $m$. Let $\mathcal{P}'$ be another partition of $V(G)$ into an exceptional set $V_0'$ and clusters of size $m'$ where $m\geq m'$ and $|m-m'| \leq 2\eps m'$. We say that $\mathcal{P}$ and $\mathcal{P}'$ are $\eps$\emph{-close} if $|V_0 \cap V_0'| \geq (1-\eps)|V_0'|$ and if for each cluster $U$ in $\mathcal{P}'$ there is a cluster $V$ in $\mathcal{P}$ such that $|U \cap V| \geq (1 -\eps)m'$. In this case we say that $U$ and $V$ are \emph{associated}. Note that $V$ is unique when $\eps < 1/2$. 
Suppose that $R$ is a multidigraph whose vertices are the clusters of $\mathcal{P}$. Let $R'$ be the multidigraph obtained from $R$ by relabelling $V$ by $V'$ for each $V \in \mathcal{P}$ associated with $V' \in \mathcal{P}'$. So $R'$ has vertex set consisting precisely of the clusters of $\mathcal{P}'$.
Moreover, for each edge $E$ from $U$ to $V$ in $R$, there is a unique edge $E'$ from $U'$ to $V'$ in $R'$ which is \emph{associated} with $E$.
The following lemma states that refinements of $\eps$-close partitions are still $\eps'$-close with a slightly bigger parameter $\eps'$.   

\begin{lemma} \label{equipartition2}
Suppose that $0<1/m \ll 1/k,\eps_1,\eps_2 \ll \eps', d,1/\ell \le 1$ and that $m/\ell\in\mathbb{N}$. Suppose that $G$ is a digraph on $n\le 2km$ vertices and that $\mathcal{P}$ is a partition of $V(G)$ into an exceptional set $V_0$ and $k$ clusters of size $m$. Let $\mathcal{P}'$ be an $\eps_1$-uniform $\ell$-refinement of $\mathcal{P}$. Suppose that $\mathcal{R}$ is another partition of $V(G)$ into an exceptional set $V_0'$ and clusters of size $m'$ that is $\eps_2$-close to $\mathcal{P}$. 
Then,
in time polynomial in $m$, one can find
an $\eps'$-uniform $\ell$-refinement $\mathcal{R}'$ of $\mathcal{R}$ which is $\eps'$-close to $\mathcal{P}'$.  
\end{lemma}
\begin{proof}
Let $U$ be a cluster of $\mathcal{P}$ and let $V$ be the cluster of $\mathcal{R}$ associated with $U$. Then, for each $U'$ in $\mathcal{P}'$ such that $U' \subseteq U$ we have that $|U' \cap V| \geq m'/\ell - \eps_2 m'$, so we can pick a subset $V'$ of $U' \cap V$ of size exactly $(1-\eps_2 \ell)m'/\ell$. There are now exactly $\eps_2 \ell m'$ vertices of $V$ which do not lie in any subcluster $V'$. Distribute these among the $V'$ so that every subcluster has equal size $m'/\ell$. Together with $V_0'$, these subclusters form the partition $\mathcal{R}'$. Clearly $U'$ and $V'$ are associated clusters of $\mathcal{P}'$ and $\mathcal{R}'$ respectively and $|U' \cap V'| \geq (1-\eps')m'/\ell$.
It is easy to see that $\mathcal{R}'$ has the required properties.
\begin{comment}
It remains to prove that $\mathcal{R}'$ is indeed an $\eps'$-uniform $\ell$-refinement of $\mathcal{R}$. Let $x$ be a vertex of $G$ and $V$ a cluster in $\mathcal{R}$ with $|N_G^+(x) \cap V| \geq \eps'|V|$. Let $V' \in \mathcal{R}'$ be a subcluster of $V$, and let $U' \in \mathcal{P}'$ and $U \in \mathcal{P}$ be their respective associated clusters. Then, since $\mathcal{P}'$ is an $\eps_1$-uniform $\ell$-refinement of $\mathcal{P}$ we have that
\begin{align*}
|N_G^+(x) \cap V'| &\geq |N_G^+(x) \cap U'| - \eps_2 m'\\
&\geq (1 - \eps_1)|N_G^+(x) \cap U|/\ell - \eps_2 m'\\
& \geq (1-\eps_1) \left( |N_G^+(x) \cap V| - \eps_2 m'\right)/\ell - \eps_2 m'\\
&\geq (1-\eps')|N_G^+(x) \cap V|
\end{align*}
and the upper bound follows similarly. Therefore $\mathcal{R}'$ satisfies (URef) (and hence also (i) and (ii) in Lemma~\ref{equipartition}) as required.
\end{comment} 
\end{proof}

Observe that if $\eps_1 \leq \eps_2$ then any $\eps_1$-uniform refinement is also an $\eps_2$-uniform refinement, and two $\eps_1$-close partitions are also $\eps_2$-close. 

Let $\mathcal{P}_2$ denote the partition obtained by taking an $\eps$-uniform $\ell_1$-refinement $\mathcal{P}_1$ of a partition $\mathcal{P}$ and then taking an 
$\eps$-uniform $\ell_2$-refinement of $\mathcal{P}_1$. Then $\mathcal{P}_2$ is a $3\eps$-uniform $\ell_2 \ell_1$-refinement of 
$\mathcal{P}$. Indeed, whenever $x$ is a vertex of $G$, $V$ is a cluster in $\mathcal{P}$ and $\vert N^+_G(x) \cap V \vert \geq \eps \vert V \vert$, then
for each cluster $V^\prime \in \mathcal{P}_2$ with $V^\prime \subseteq V$, we have
\begin{equation} \label{doublerefinement}
\vert N_G^+(x) \cap V^\prime \vert = (1 \pm \eps)^2\vert N_G^+(x) \cap V \vert / \ell_2 \ell_1 = (1 \pm 3\eps)\vert N_G^+(x) \cap V \vert / \ell_2 \ell_1,
\end{equation}
and similarly for the inneighbourhoods.

%%%%%%%%%%%%%%%%%%%%%%%%%%%%%%%%%%%%%%%%%%%%%%%%%%%%%%%%%%%%%%%%%%%%%%%%%%%%%%%%%%%%%%%%%%%%%%

\section{Tools for finding subgraphs, $1$-factors and Hamilton cycles} \label{sec:tools}

\subsection{Almost regular spanning subgraphs}

The following result (which is proved as Lemma~5.2 in~\cite{KOKelly}) shows that in a robust outexpander, we can guarantee a spanning subdigraph with a given degree sequence
(as long as the required degrees are not too large and do not deviate too much from each other). 
If $x$ is a vertex of a multidigraph $Q$, we write $d^+_Q(x)$ for the number of edges in $Q$ whose initial vertex is $x$
and $d^-_Q(x)$ for the number of edges in $Q$ whose final vertex is $x$.

\begin{lemma}\label{regrobust} Let $q\in\mathbb{N}$.
Suppose that $0<1/n\ll\eps\ll \nu \le \tau \ll \alpha<1$ and that $1/n \ll \rho \le q\nu^2/3$. Let~$G$ be a digraph on~$n$ vertices with
$\delta^0(G)\ge \alpha n$ which is a robust $(\nu,\tau)$-outexpander. Suppose that $Q$ is a multidigraph on $V(G)$ such that
whenever $xy\in E(G)$ then $Q$ contains at least $q$ edges from $x$ to $y$. For every vertex $x$ of $G$, let $n^+_x,n^-_x\in\mathbb{N}$
be such that $(1-\eps)\rho n\le n^+_x, n^-_x\le (1+\eps)\rho n$ and such that $\sum_{x\in V(G)} n^+_x=\sum_{x\in V(G)} n^-_x$.
Then $Q$ contains a spanning submultidigraph $Q'$ such that $d^+_{Q'}(x)=n^+_x$ and $d^-_{Q'}(x)=n^-_x$ for every $x\in V(G)=V(Q)$.
\end{lemma}

The next result (Lemma~16 in~\cite{kot}) is an analogue of the previous one where we consider superregular pairs instead of robust outexpanders.
In both cases, the proof is algorithmic (as it is based on the Max-Flow-Min-Cut Theorem).

\begin{lemma} \label{almostreg2}
Let $0 < 1/n \ll \eps \ll \beta \ll \alpha^\prime \ll \alpha \ll 1$. Suppose that $G = (A,B)$ is an $(\eps , \beta + \eps)$-superregular pair where $\vert A \vert = \vert B \vert = n$. Define $\kappa := (1 - \alpha) \beta n$. Suppose we have a non-negative integer $m^+_a \leq \alpha^\prime \beta n$ associated with each $a \in A$ and a non-negative integer $m^-_b \leq \alpha^\prime \beta n$ associated with each $b \in B$ such that $\sum_{a \in A}{m^+_a} = \sum_{b \in B}{m^-_b}$. Then $G$ contains a spanning subgraph $H$ in which $d_H(a) = n^+_a := \kappa - m^+_a$ for any $a \in A$ and $d_H(b) = n^-_b := \kappa - m^-_b$ for any $b \in B$.
\end{lemma}

\subsection{Decomposing regular digraphs into 1-factors} \label{petersensec}

Petersen proved that every regular undirected graph can be decomposed into 1-factors. The corresponding result for directed graphs is well known; for completeness we include the proof (which is algorithmic as perfect matchings can be found in polynomial time).

\begin{proposition} \label{petersen}
Any $r$-regular multidigraph $G$ contains $r$ edge-disjoint 1-factors.
\end{proposition}
\begin{proof}
Define an undirected bipartite graph $J$ with two vertex classes $A$ and $B$, each of which is a copy of $V(G)$, with an edge from $a \in A$ to $b \in B$ for each edge from $a$ to $b$ in $G$. $J$ is $r$-regular so, by Hall's Theorem \cite{hall}, contains a perfect matching $M_1$. Then $J \setminus M_1$ is $(r-1)$-regular so contains a perfect matching $M_2$. Repeating this procedure we can decompose $J$ into $r$ perfect matchings, each of which corresponds to a 1-factor in $G$.
\end{proof}

\subsection{Unwinding cycles} \label{unwindsec}

At two points in the proof, we will partition a blown-up cycle into several longer, thinner blown-up cycles on subclusters of the original clusters. The following section describes how this process is implemented and describes a special approximate decomposition to be used in Section \ref{unwinding}.

Suppose that $D = p \otimes C_n$ is a $p$-fold blow-up of a cycle $C_n$ of length $n$.
Let $X_1, \ldots, X_n$ be the vertex classes of $D$.
We call any edge-disjoint collection $C^1 , \ldots , C^{p'}$ of $p'$ Hamilton cycles of $D$ a $p'$-\emph{unwinding} of $D$. 
The following lemma guarantees a $(p-1)$-unwinding in which, for each $C^d$ and each $i$, the $i$th vertices of two distinct classes $X_j$ and $X_{j'}$ have distance at least $p$ on $C^d$.

\begin{lemma} \label{unwind}
Suppose that $p>2$ is a prime, suppose $n \in \mathbb{N}$ and let $D = p \otimes C_n$ be a $p$-fold blow-up of a cycle $C_n$ of length $n$. Denote the vertex classes of $D$ by $X_1 , \ldots , X_n$ where, for all $j$ with $1 \leq j \leq n$, we have $X_j = \lbrace x_j^1 , \ldots , x_j^p \rbrace$. Then $D$ contains a $(p-1)$-unwinding $C^1 , \ldots , C^{p-1}$ such that for every $1 \leq d \leq p-1$ and every $1 \leq i \leq p$,
\begin{itemize}

\item[(i)] if $p$ is coprime to $n$, then the vertices $x_1^i, \ldots, x_n^i$ have pairwise distance at least $p$ on $C^d$; 
\item[(ii)] if $p$ is not coprime to $n$, then the vertices $x_1^i, \ldots, x_{n-2}^i$ have pairwise distance at least $p$ on $C^d$.
\end{itemize}
\end{lemma} 

\begin{proof}
We first prove (i). 
Let $\lbrace a \rbrace$ denote the residue of $a$ modulo $p$ and $[b]$ the residue of $b$ modulo $n$ where we adopt the convention that $\lbrace \ell p \rbrace := p$ and $[\ell n] := n$ for any $\ell \in \mathbb{N}$.
For $1 \leq d \leq p-1$ define the modular arithmetic progression 
\[
P(d) := \left( \lbrace 1 \rbrace , \lbrace 1+d \rbrace , \ldots , \lbrace 1+(np-1)d \rbrace \right)
\]
in $\mathbb{Z}_p$. For each $1 \leq k \leq np$ and $1 \leq d \leq p-1$, define the edge
$$
e_k^d := x_{[k]}^{P(d)_k} x_{[k+1]}^{P(d)_{k+1}},
$$
where $P(d)_k$ denotes the $k$th term of $P(d)$ and $P(d)_{np+1} := P(d)_1$. We define $C^d$ to be the digraph with vertex set $V(D)$ and edges $e_1^d , \ldots , e_{np}^d$ (see Figure \ref{fig:unwind}). Note that $C^d$ is clearly a closed walk in $D$.

\tikzstyle{every node}=[circle, draw, fill=black!50,
                        inner sep=0pt, minimum width=2pt]

\begin{center}
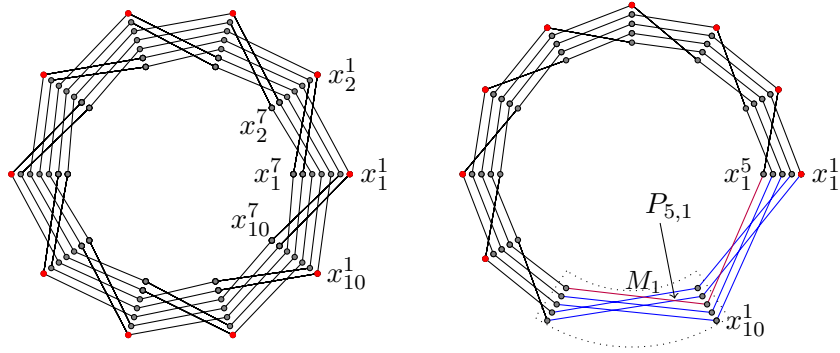
\begin{figure} 
\centering
\begin{tikzpicture}[scale=0.5]
\begin{scope}
    \draw \foreach \x in {0,36,...,324} \foreach \y in {4.5, 4.25, 4, 3.75, 3.5}
    {
        %(\x:\y)  -- (\x+36:\y-0.25)
        %(\x:3)  -- (\x+36: 4.5)
        (\x:\y)  -- (\x+36:\y-0.5)
        (\x:3.25)  -- (\x+36:4.5)
        (\x:3)  -- (\x+36:4.25)
        (\x:\y) node {}
        (\x:3) node {}
        (\x:3.25) node {}
    %    (\x:4.5) node [fill=black, minimum width=3pt] {}
    };
    \draw (0:4.5) node[label=right:$x_1^1$] {};
    \draw (0:3) node[label=left:$x_1^7$] {};
    \draw (36:4.5) node[label=right:$x_2^1$] {};
    \draw (36:3) node[label=225:$x_2^7$] {};
    \draw (324:4.5) node[label=right:$x_{10}^1$] {};
    \draw (324:3) node[label=135:$x_{10}^7$] {};
    \draw[color=red] \foreach \x in {0,36,...,324}
      {
         (\x:4.5) node[fill=red] {}
      };
\end{scope}

\begin{scope}[xshift=12cm]
\draw (275:3.55) node[color=white,label=above:\small $M_1$] {};
    \draw \foreach \x in {0,30,...,210} \foreach \y in {4.5, 4.25, 4, 3.75}
    {
        (\x:\y)  -- (\x+30:\y-0.25)
        (\x:3.5)  -- (\x+30:4.5)
        (\x:\y) node {}
    %    (\x:4.5) node [fill=black, minimum width=3pt] {}
    };
 \draw[color=blue] \foreach \x in {240}
    {
%        (\x:3.5)  -- (\x+60:4)
        (\x:3.75)  -- (\x+60:4.25)
        (\x:4) -- (\x+60:4.5)
        (\x:4.25) -- (\x+60:3.5)
        (\x:4.5) -- (\x+60:3.75)
    %    (\x:4.5) node [fill=black, minimum width=3pt] {}
    };
 \draw[color=blue] \foreach \x in {300}
    {
        (\x:3.5)  -- (\x+60:4.25)
        (\x:3.75)  -- (\x+60:4.5)
   %     (\x:4) -- (\x+60:3.5)
        (\x:4.25) -- (\x+60:3.75)
        (\x:4.5) -- (\x+60:4)
    %    (\x:4.5) node [fill=black, minimum width=3pt] {}
    };
\draw[color=purple] (240:3.5) -- (300:4) -- (360:3.5);
\draw (300:2) node[color=white,label=90:{$P_{5,1}$}] {} ;
\draw[->] (300:1.5) -- (289:3.5);

\draw[dotted] (236:3.1) -- (236:4.6)
                        arc (236:304:4.6) -- (304:4.6)
                        (304:4.6) -- (304:3.1)
                        arc (304:236:3.1) -- (236:3.1);

    \draw (0:4.5) node[label=0:$x_{1}^1$] {};
    \draw (0:3.5) node[label=180:$x_{1}^5$] {};
    \draw (300:4.5) node[label=right:$x_{10}^1$] {};
    \draw \foreach \x in {0,30,...,240,300} \foreach \y in {4.5, 4.25, 4, 3.75, 3.5}
      {
         (\x:\y) node {}
      };
    \draw[color=red] \foreach \x in {0,30,...,210}
      {
         (\x:4.5) node[fill=red] {}
      };
\end{scope}

\end{tikzpicture}\quad
\caption{Illustrating Lemma~\ref{unwind}(i) with $n = 10, p=7 , d = 2$ and Lemma~\ref{unwind}(ii) with $n=10,p=5,d=1$.}\label{fig:unwind}
\end{figure}
\end{center}

\noindent {\sc Claim A.}~
\emph{For each $1 \leq d,d' \leq p-1$ and $1 \leq k,k' \leq np$ the following hold:
\begin{itemize}
\item[(a)] $P(d)$ is periodic with period $p$.
\item[(b)] Suppose $P(d)_k = P(d')_{k'} ~~~\text{and}~~~ P(d)_{k+1} = P(d')_{k'+1}$.
Then $d=d'$.
\end{itemize}}

\noindent
We first show that the claim implies (i). 
First note that (a) and the fact that $n$ is coprime to $p$ imply that every vertex is visited exactly once in the closed walk $C^d$, so $C^d$ must in fact be a Hamilton cycle.
Now suppose $e_k^d = e_{k'}^{d'}$.
Then (b) implies that also $d=d'$. 
Thus no two $C^d$ share an edge; thus $C^1 , \ldots , C^{p-1}$ is a collection of edge-disjoint Hamilton cycles. (a) implies that, on each $C^d$, the distance between $x_\ell^i$ and $x_{\ell'}^i$ is a multiple of $p$ for any $1 \leq \ell,\ell' \leq n$. Therefore $C^1 ,\ldots , C^{p-1}$ have the required property. 

To prove (a) of the claim, note that
$P(d)_k = P(d)_{k'}$ if and only if $1 + kd \equiv 1 + k'd \mod p$ if and only if $k \equiv k' \mod p$ since $d$ is coprime to $p$.
To prove (b), note that $P(d)_k = P(d')_{k'}$ and $P(d)_{k+1} = P(d')_{k'+1}$ imply that
\begin{align}
\label{mod1} 1 + kd &\equiv 1 + k'd' \mod p\\
\label{mod2} 1 + (k+1)d &\equiv 1 + (k'+1)d' \mod p.
\end{align}
Subtracting (\ref{mod1}) from (\ref{mod2}) gives $d \equiv d' \mod p$; but $1 \leq d,d' \leq p-1$ so $d=d'$. This proves the claim and completes the proof of (i).

\medskip
We now prove (ii). So suppose instead that $n$ and $p$ are not coprime. Then $n' := n-2$ is coprime to $p$ since $p>2$. The idea is to use paths derived from the cycles defined above for the first $n'$ clusters and extend them into Hamilton cycles via the remaining clusters. To this end, form an auxiliary blown-up cycle $\tilde{D}$ from $D$ by identifying $x_j^i$ with $x_{j'}^i$ whenever $1 \leq i \leq p$ and $j,j' \in \lbrace n-1,n,1 \rbrace$ and call this vertex $x_{1}^i$ in $\tilde{D}$.  
Now remove any resulting loops from $\tilde{D}$. So $\tilde{D} = p \otimes C_{n-2}$. Next, apply (i) to $\tilde{D}$ to obtain $\tilde{C}^1, \ldots, \tilde{C}^{p-1}$. Now, for each $1 \leq d \leq p-1$, obtain $E_1(C^d)$ from $E(\tilde{C}^d)$ by replacing any edge $x_{n-2}^ix_1^{i'}$ by $x_{n-2}^ix_{n-1}^{i'}$.
Note that, in $D$, $E_1(C^1) , \ldots, E_1(C^{p-1})$ is an edge-disjoint collection of $p-1$ paths each of length $n'$.

\noindent {\sc Claim B.}~
\emph{The collections $E_1(C^1), \ldots , E_1(C^{p-1})$ of edges can be extended into $p-1$ edge-disjoint Hamilton cycles $C^1, \ldots, C^{p-1}$ respectively such that $C^d$ is a subdivision of $\tilde{C}^d$ for each $1 \leq d \leq p-1$. }

\noindent
To see how this implies the lemma, note that since $C^d$ is a subdivision of $\tilde{C}^d$, the distance between any two vertices in $C^d$ is at least the distance in $\tilde{C}^d$. This immediately gives the required property.

It remains to prove the claim. For each $1 \leq d \leq p-1$ we will need to find a collection of edge-disjoint paths from $x_{n-1}^i$ to $x_1^i$ for $1 \leq i \leq p$ to extend $E_1(C^d)$ to $C^d$. Moreover, these collections must be pairwise edge-disjoint.
By Hall's Theorem, we can find $p-1$ edge-disjoint perfect matchings $M_1 , \ldots , M_{p-1}$ in the complete bipartite subgraph $(X_{n-1} , X_{n})$ of $D$. For each $1 \leq d \leq p-1$ and $1 \leq i \leq p$, define
\[
P_{i,d} = x_{n-1}^i x_{n}^{i'} x_1^i
\]
whenever $x_{n-1}^i x_{n}^{i'}$ is an edge in $M_d$. Since the $M_d$ are edge-disjoint matchings, the $P_{i,d}$ are edge-disjoint paths with the required property. Thus, for each $1 \leq d \leq p-1$, defining
\[
E(C^d) := E_1(C^d) \cup \bigcup_{1 \leq i \leq p}{P_{i,d}}
\]
gives $p-1$ edge-disjoint Hamilton cycles $C^1 , \ldots , C^{p-1}$. This proves the claim and completes the proof of (ii).
\end{proof}

\subsection{Merging 1-factors in blown-up cycles} 
In Section~\ref{aldecomp} we will have found an approximate decomposition of a robustly expanding digraph
into 1-factors. The following lemma will use the special structure of the 1-factors  to merge their cycles into a single Hamilton cycle.
It is a special case of Lemma~6.5 in~\cite{KOKelly},  which in turn is based on an idea in~\cite{CKKO}.
As noted in~\cite{KOKelly}, the cycle guaranteed by the lemma can be found in polynomial time.
Roughly speaking, the lemma asserts that if we have a $1$-regular digraph $F$ where most of the edges wind around a `blown-up' cycle $C=V_1\dots V_k$, then under certain circumstances
we can turn $F$ into a (single) cycle by replacing a few edges of $F$ by edges from a digraph $G$ whose edges all wind around $C$.

\begin{lemma}\label{trick}
Let $0<1/m\ll \eps\ll d< 1$.
Let $V_1,\dots,V_k$ be pairwise disjoint clusters, each of size $m$ and let $C=V_1\dots V_k$ be a directed cycle on these clusters.
Let $J\subseteq E(C)$.
Let $G$ be a digraph on $V_1\cup \dots\cup V_k$ such that $G[V_i,V_{i+1}]$ is $(\eps, d)$-superregular
for every $V_i V_{i+1} \in J$.
Suppose that $F$ is a $1$-regular digraph with $V_1\cup \dots \cup V_k\subseteq V(F)$ such that the following properties hold:
\begin{itemize}
\item[\rm{(i)}]For each edge $V_iV_{i+1}\in J$ the digraph $F[V_i,V_{i+1}]$ is a perfect matching.
\item[\rm{(ii)}] For each cycle $D$ in $F$ there is some edge $V_iV_{i+1}\in J$ such that $D$ contains a vertex
in $V_i$.
\item[\rm{(iii)}] Whenever $V_iV_{i+1}, V_jV_{j+1}\in J$ are such that $J$ avoids all edges in the segment $V_{i+1}CV_j$ of
$C$ from $V_{i+1}$ to $V_j$, then $F$ contains a path $P_{ij}$ joining some vertex $u_{i+1}\in V_{i+1}$ to some
vertex $u'_j\in V_j$ such that $P_{ij}$ winds around~$C$.
\end{itemize}
Then we can obtain a cycle on $V(F)$ from $F$ by replacing $F[V_i,V_{i+1}]$ with a suitable perfect matching
in $G[V_i,V_{i+1}]$ for each edge $V_iV_{i+1}\in J$. 
\end{lemma}

It will also be convenient to use the following result from~\cite{KOTchvatal}, which guarantees a Hamilton cycle in a robustly expanding digraph.
The proof of Lemma~\ref{trick} actually consists of repeated applications of Theorem~\ref{1hc} to a suitable auxiliary digraph.
The proof of Theorem~\ref{1hc} can be made algorithmic but this is not needed here as we only apply it to a `reduced' digraph, obtained from the regularity lemma.

\begin{theorem} \label{1hc} 
Let $n_0$ be a positive integer and $\alpha , \nu , \tau$ be positive constants such that $1/n_0 \ll \nu \leq \tau \ll \alpha < 1$. Let $G$ be a digraph on $n \geq n_0$ vertices with $\delta^0(G) \geq \alpha n$ which is a robust $(\nu , \tau)$-outexpander. Then $G$ contains a Hamilton cycle.
\end{theorem}
%%%%%%%%%%%%%%%%%%%%%%%%%%%%%%%%%%%%%%%%%%%%%%%%%%%%%%%%%%%%%%%%%%%%%%%%%%%%%%%%%%%%%%%%%%%%%
%%%%%%%%%%%%%%%%%%%%%%%%%%%%%%%%%%%%%%%%%%%%%%%%%%%%%%%%%%%%%%%%%%%%%%%%%%%%%%%%%%%%%

\section{The proof of Theorem~\ref{main}} \label{sec:mainproof}

\subsection{Applying the diregularity lemma} \label{begin}
We choose $\tau$ so that $\tau \ll \alpha$.
Without loss of generality we may assume that $\nu \ll \tau$ as
any robust $(\nu, \tau)$-outexpander is also a robust $(\nu', \tau)$-outexpander for any $\nu' \leq \nu$.
We may also assume that
$0 < \eta \ll \nu$ as a collection of $(1 - \eta^\prime)r$ 
edge-disjoint Hamilton cycles certainly contains a collection of $(1 - \eta)r$ edge-disjoint Hamilton cycles if $\eta^\prime \leq \eta$. 
Define further constants satisfying
\begin{align} \label{hierarchy}
0 &< 1/n_0 \ll 1/M \ll 1/M^\prime \ll \tilde{\eps} \ll \eps \ll \eps^\prime \ll \xi \ll 1/p  \nonumber\\
&\ll \beta \ll d \ll 1/s \ll \gamma \ll d' \ll \eta \ll \nu \ll \tau \ll \alpha,
\end{align}
where $s \in \mathbb{N}$ is even and $p$ is a prime.

Let $G$ be a digraph of order $n \ge n_0$ such that $G$ is an $r$-regular robust $(\nu, \tau)$-outexpander
with $r \ge \alpha n$.
Define $\tilde{\alpha}$ by $r=\tilde{\alpha} n$. 
Apply the diregularity lemma (Lemma~\ref{direg}) to $G$ with parameters $\tilde{\eps}^{12} , d, M^\prime$ 
to obtain clusters $\tilde{V}_1 , \ldots , \tilde{V}_{\tilde{L}}$ of size $\tilde{m}$, an exceptional set $V_0$, 
a pure digraph $G'$ and a reduced digraph $\tilde{R}$. So $\vert \tilde{R} \vert = \tilde{L}$ and $M' \le \tilde{L} \le M$.
We denote the above partition of $G$ by $\tilde{\mathcal{P}}$ and call the $\tilde{V}_j$ the 
\emph{clusters} of $\tilde{\mathcal{P}}$, frequently referred to as \emph{base primary clusters} (to distinguish them from other types of cluster defined later on). 
Let $\tilde{R}^\prime$ be the spanning subdigraph of $\tilde{R}$ whose edges correspond to pairs of density at least $d^\prime$. 
So $\tilde{V}_i \tilde{V}_j$ is an edge of $\tilde{R}^\prime$ if $(\tilde{V}_i , \tilde{V}_j)_{G^\prime}$ has density at least $d^\prime$. 

When $\tilde{E}$ is an edge of $\tilde{R}$ from $\tilde{V}_i$ to $\tilde{V}_j$ we write $G^\prime(\tilde{E})$ for the subdigraph $(\tilde{V}_i , \tilde{V}_j)_{G^\prime}$ and $d_{ij}$ for the density of this pair. Then by Lemma~\ref{direg}, $G^\prime(\tilde{E})$ is $(\tilde{\eps}^{12},d_{ij})$-regular. Let $\tilde{R}(\beta)$ denote the reduced multidigraph of $G$ (obtained from $\tilde{R}$) with 
parameters $\tilde{\eps}^{12} , \beta, d$ and $M^\prime$. 
Let $\tilde{R}^\prime(\beta)$ be the multidigraph obtained from $\tilde{R}(\beta)$ by including only those edges which also correspond to an edge of $\tilde{R}'$. 
Roughly speaking, our aim is to find an approximate decomposition of $\tilde{R}(\beta)$ into edge-disjoint 1-factors $\tilde{F}$, and then find an approximate Hamilton decomposition of a subdigraph of $G$ consisting mainly of edges that correspond to a pair in $\tilde{F}$. 

For each edge $\tilde{E}$ of $\tilde{R}$, apply Lemma~\ref{randomregslice}(i) to $G^\prime(\tilde{E})$ with parameters $K := \lfloor d_{ij}/\beta \rfloor$ and $\gamma_k := \beta$ for each $ 1\leq k \leq K$ to obtain  $K$ edge-disjoint $(\tilde{\eps}, \beta)$-regular subdigraphs. We associate each of these with a unique edge $E$ from $\tilde{V}_i$ to $\tilde{V}_j$ of $\tilde{R}(\beta)$ and call the corresponding digraph $G'(E)$.

Let $A$ be a cluster of $\tilde{R}$ and let $E(A)$ denote the set of edges incident to $A$ in $\tilde{R}(\beta)$.
For an edge $E$ in $E(A)$ and $x \in A$, we say that the pair $(x,E)$ is \emph{good} if
\begin{itemize}
\item $A$ is the initial cluster of $E$ and $d_{G^\prime(E)}^+(x) =(\beta \pm  2\tilde{\eps})\tilde{m}$; or
\item $A$ is the final cluster of $E$ and $d_{G'(E)}^-(x) = (\beta \pm 2\tilde{\eps})\tilde{m}$
\end{itemize}
and say it is \emph{bad} otherwise
(recall that $\tilde{m}$ is the cluster size). We say that $x$ is \emph{good}
if $x$ forms a bad pair with at most $\xi |E(A)|$ edges in $E(A)$.
Note that for a fixed edge $E$ in $E(A)$, at most $\tilde{\eps} \tilde{m}$ vertices $x \in A$ are bad.
So by double counting the number of bad pairs, it is easy to see that the number of bad vertices in $A$ is at most
$\tilde{\eps} \tilde{m}/ \xi$. %
\COMMENT{($\# bad vertices) \cdot \xi |E(A)| \le \# bad pairs \le |E(A)| \tilde{\eps} \tilde{m}$}

We remove every bad vertex from its cluster as well as possibly some more arbitrary vertices so that exactly $\tilde{\eps}\tilde{m}/\xi$ vertices have been removed from each cluster. 
We then remove at most $2sp$ further vertices from each cluster in order to guarantee that the cluster size is divisible by $2sp$.
We still denote the cluster size by $\tilde{m}$ and still call the clusters \emph{base primary}. 
Each vertex removed here is added to the exceptional set $V_0$, which we now call the \emph{core exceptional set}. So
\begin{equation} \label{V0bound}
\vert V_0 \vert \leq (\tilde{\eps}^{12} + \tilde{\eps}/\xi)n + 2sp\tilde{L} \stackrel{(\ref{hierarchy})}{\le} \sqrt{\tilde{\eps}} n/2.
\end{equation}
We still denote the partition of $V(G)$ into $V_0$ and these clusters by $\tilde{\mathcal{P}}$. 
Note that for each edge $E$ of $\tilde{R}(\beta)$, the digraph $G^\prime(E)$ is still $(\sqrt{\tilde{\eps}} , \beta)$-regular by Proposition~\ref{superslice}(i) (at most $\tilde{\eps}\tilde{m}/4$ vertices were removed from each cluster). %
Lemma~\ref{reduced} implies that
\begin{align} \label{eqreduced}
\nonumber \delta^0(\tilde{R}^\prime) \geq (\tilde{\alpha} - 3d^\prime)\tilde{L} \ \ &\mbox{ and } \ \ \delta^0(\tilde{R}^\prime(\beta)) \geq (\tilde{\alpha} - 4d^\prime)\dfrac{\tilde{L}}{\beta},\\
\delta^0(\tilde{R}(\beta)) \geq (\tilde{\alpha} - 4d)\dfrac{\tilde{L}}{\beta}
 \ \ &\mbox{ and } \ \
 \Delta^0(\tilde{R}(\beta)) \leq (\tilde{\alpha} + 2 \tilde{\eps}^{12})\dfrac{\tilde{L}}{\beta}.
\end{align}
By Lemma~\ref{inherit}, $\tilde{R}^\prime$ is a robust $(\nu/4, 3\tau)$-outexpander.  Note that it is a subdigraph of 
$\tilde{R}^\prime(\beta) \subseteq \tilde{R}(\beta)$ and that all of its edges have multiplicity at least $q := d^\prime/\beta$ in $\tilde{R}^\prime(\beta)$. 
Let 
\begin{equation} \label{eq:defr}
\tilde{r} := (\tilde{\alpha} - \gamma)\tilde{L}/\beta. 
\end{equation} 
Let $n^{\pm}_U := d^{\pm}_{\tilde{R}(\beta)}(U) - \tilde{r}$ and let $\rho := \gamma/\beta$, so $\rho \leq q\nu^2/3$. Note that 
$$
(1 - \frac{4d}{\gamma})\rho \tilde{L} = (\gamma-4d)\frac{\tilde{L}}{\beta} 
\le n_U^\pm \le (\gamma+2\tilde{\eps}^{12}) \frac{\tilde{L}}{\beta} = (1 + \frac{2\tilde{\eps}^{12}}{\gamma})\rho \tilde{L}.
$$ 
So we can 
apply Lemma~\ref{regrobust} to $(G,Q) := (\tilde{R}^\prime , \tilde{R}^\prime(\beta))$  
to obtain a sub-multidigraph $W$ of $\tilde{R}^\prime(\beta)$ (and hence of $\tilde{R}(\beta)$) such that 
 the in- and outdegrees of each cluster $U$ are exactly $n^{\pm}_U$.
 So
$\tilde{R}(\beta) \setminus W$ is a spanning $\tilde{r}$-regular sub-multidigraph of $\tilde{R}(\beta)$. 
Apply Proposition~\ref{petersen} to decompose $\tilde{R}(\beta) \setminus W$ into $\tilde{r}$ 1-factors
$\tilde{F}_1 , \ldots , \tilde{F}_{\tilde{r}}$ of $\tilde{R}(\beta)$. So each $\tilde{F}_t$ corresponds to a 
collection of blown-up cycles spanning $V(G) \setminus V_0$.
Note that this step would not work if we only considered $\tilde{R}$ and $\tilde{R}(\beta)$ and tried to apply Lemma~\ref{regrobust} to find $W$ in $\tilde{R}(\beta)$ directly.

%%%%%%%%%%%%%%%%%%%%%%%%%%%%%%%%%%%%%%%%%%%%%%%%%%%%%%
\subsection{Thin auxiliary digraphs $H$} \label{sec:H}

We now define edge-disjoint subdigraphs $H_0^+$, $H_0^-$, $H_1^+$, $H_1^-$ and $H_2$ of $G$, which are sparse `shadows' of the reduced multidigraph. They act as reservoirs 
of well-distributed edges which will be used at various stages in the proof.
The role of $H_0^\pm$ is to connect blown-up cycles (in Section~\ref{connect}) to ensure that our final merging procedure does indeed yield Hamilton cycles.
In Section \ref{exceptional} edges will be taken from $H_1^\pm$ to connect the vertices in the special exceptional sets $V_{0,i}$ 
(defined later)
to the non-exceptional vertices in each slice $G_i$ (defined in Section \ref{unwinding}). 
$H_2$ will be used to construct `balancing edges' which will be introduced in Section \ref{shadow}.
We choose these subdigraphs already at this point because if we remove them later then this might destroy the superregularity of the pairs in the $G_i$.%
\COMMENT{we might want to rename the indices so that they appear in the correct order.Renaming $H_1$ might be messy, so probably best to rename $H_0$ to $H_0$ and $H_2$ to $H_0$ and $H_3$ to $H_2$?}

We obtain $H_0^+, H_0^-, H_1^+, H_1^-, H_2$ as follows. Each has vertex set $V(G)$ and initially contains no edges.
Then, for each edge $E$ of $\tilde{R}(\beta)$, $G'(E)$ is a $(\sqrt{\tilde{\eps}},\beta)$-regular pair and we can apply Lemma~\ref{randomregslice}(i) to $G'(E)$ with $\gamma_1 :=\beta_1$ where
\begin{equation} \label{eq:beta_1}
\beta_1:=(1-5\gamma)\beta
\end{equation}
and $\gamma_2 := \ldots = \gamma_6 := \gamma\beta$,
to obtain six edge-disjoint pairs $J_1, \ldots, J_6$, where $J_k$ is $(\tilde{\eps}^{1/24},\gamma_k)$-regular, and we call these digraphs $G^*(E)$, $H_0^+(E)$, $H_0^-(E)$, $H_1^+(E)$, $H_1^-(E)$ and $H_2(E)$ respectively. 
We denote the union of $H(E)$ over all edges $E$ of $\tilde{R}(\beta)$ by $H$.
We will only use the weaker bounds that the `remaining' subdigraph $G^*(E)$ of $G'(E)$ is $(\eps/8,\beta_1)$-regular and for each $H = H_0^+, H_0^-, H_1^+, H_1^-, H_2$ we have that $H(E)$ is $(\eps,\gamma\beta)$-regular.
Moreover, Lemma~\ref{randomregslice}(i) implies that if $E$ is an edge from $A$ to $B$ and if $x \in A$ and $y \in B$ are good for $E$, then
\begin{equation} \label{xgoodE}
d_{H(E)}^+(x), d_{H(E)}^-(y) = (\gamma\beta \pm 2\eps)\tilde{m}.
\end{equation}
Note also that $V_0$ is isolated in each $H$.
We now derive some further properties of these digraphs which we will need later.
Firstly, we have the following property for $H_0^+$ and $H_0^-$:
\begin{itemize}
\item[(H0)] Suppose that $\tilde{A}\tilde{B}$ is an edge of $\tilde{R}$. Then for at least $(1-\eps')|\tilde{A}|$ of the vertices $x \in \tilde{A}$ and $(1-\eps')|\tilde{B}|$ of the vertices $y \in \tilde{B}$ we have
$$
|N^+_{H_0^+}(x) \cap \tilde{B} \vert \ge  \gamma d \tilde{m}/2\ \ \mbox{ and }\ \ |N^-_{H_0^-}(y) \cap \tilde{A} \vert \ge  \gamma d \tilde{m}/2.
$$
\end{itemize}
To see this, note first that every edge $E$ of $\tilde{R}$ has multiplicity at least $d/\beta$ in $\tilde{R}(\beta)$.
Let $E_1,\dots,E_\ell$ be the edges of $\tilde{R}(\beta)$ corresponding to $E$. So $d/\beta \leq \ell \leq 1/\beta$.
Recall that $H_0^+(E_i)$ is $(\eps,\gamma \beta)$-regular. 
Let $A'$ be the set of all vertices $x \in \tilde{A}$ such that $x$ has outdegree at least $(\gamma \beta -2\eps) \tilde{m}$ in each of $H_0^+(E_1),\dots,H_0^+(E_\ell)$.
Then $|A'| \ge (1-\ell\eps)\tilde{m} \ge (1-\eps')\tilde{m}$. Moreover, for all $x \in A'$, we have
$$|N^+_{H_0^+}(x) \cap \tilde{B}| \ge \ell(\gamma \beta -2\eps) \tilde{m} \ge \frac{d}{\beta} \frac{\gamma \beta}{2} \tilde{m} \ge \frac{\gamma d \tilde{m}}{2}.
$$
The proof of the second inequality is similar.

We also have the following property of $H_1^+$ and $H_1^-$:
\begin{itemize}
\item[(H1)] For all $x \in V(G) \setminus V_0$, we have
$
\gamma \tilde{\alpha} n/3 \leq d_{H_1^+}^\pm(x),d_{H_1^-}^\pm(x) \leq 2 \gamma \tilde{\alpha} n.
$
\end{itemize}
(H1) follows from the fact that $V_0$ contains all the bad vertices (in the sense of Section \ref{begin}).
Indeed, since any vertex $x \in V(G)\setminus V_0$ is good we have
$$
d^+_{H_1^+}(x) \stackrel{(\ref{xgoodE})}{\ge} \delta^+(\tilde{R}(\beta))(1- \xi) (\gamma \beta -2\eps)\tilde{m}
\stackrel{(\ref{eqreduced})}{\ge} \frac{\tilde{\alpha} \tilde{L}}{2\beta} \gamma \beta \tilde{m} 
\ge \gamma \tilde{\alpha} n/3.
$$
The other bounds in (H1) follow similarly.%
\COMMENT{$d^+_{H_1^+}(x) \stackrel{(\ref{xgoodE})}{\le} \Delta^+(\tilde{R})(\beta)(\gamma \beta +2\eps)\tilde{m}$ }

%%%%%%%%%%%%%%%%%%%%%%%%%%%%%%%%%%%%%%%%%%%%%%%%%%%%%%%%

\subsection{Unwinding cycles} \label{unwinding}

For each $1 \leq t \leq \tilde{r}$ we now apply Lemma~\ref{supreg} to each cycle in $\tilde{F}_t$ to remove vertices from each cluster, so that they now have 
size exactly $(1-\eps/4)\tilde{m}$ and such that each edge $E$ of $\tilde{F}_t$ corresponds to an
$(\eps/2 , \beta_1)$-superregular pair $G^*(E)$.
By removing at most $2sp$ further vertices from each cluster we obtain clusters of size $m$ such that $2sp \mid m$. We call these \emph{adapted primary clusters} or \emph{adapted primary} $(t)$-\emph{clusters} if we wish to emphasise the dependence on $t$, and say that each such cluster is \emph{associated} with the base primary cluster from which it was formed.
Since $2sp \leq \eps \tilde{m}/4$ it is easy to see that now each edge $E$ of $\tilde{F}_t$ 
corresponds to an $(\eps , \beta_1)$-superregular pair $G^*(E)$.
Note that
\begin{equation} \label{boundL}%
\COMMENT{This relation is we discussed a couple of times, but somehow always forgot to put in the paper.}
\frac{1}{m} \le \frac{2 \tilde{L}}{n} \le \frac{2M}{n_0} \ll \frac{1}{\tilde{L}}\ \ \mbox{ and } \ \ (1-\eps)n \stackrel{(\ref{V0bound})}{\leq} m\tilde{L} \leq \tilde{m}\tilde{L} \leq n. 
\end{equation}

Let $\tilde{V}_{0,t}^{\rm spec}$ denote the set of all those vertices in $G$ which were removed from the clusters in this step.
We call them the \emph{special exceptional vertices} (for the original slice $t$).
So $|\tilde{V}_{0,t}^{\rm spec}| \le \eps n/4 + 2sp\tilde{L} \leq \eps n/2$.
Let $\tilde{V}_{0,t}= V_0 \cup \tilde{V}_{0,t}^{\rm spec}$.
Then
\begin{equation} \label{tildeV0}
\vert \tilde{V}_{0,t} \vert \stackrel{(\ref{V0bound})}{\leq} \dfrac{2\eps n}{3}.
\end{equation}

We denote the collection of the adapted primary $(t)$-clusters together with the exceptional set $\tilde{V}_{0,t}$ by $\mathcal{P}(t)$. Note that $\mathcal{P}(t)$ and $\tilde{\mathcal{P}}$ are $2\eps/3$-close for every $1 \leq t \leq \tilde{r}$ (recall that this notion was defined before Lemma~\ref{equipartition2}).

For each cycle $C$ in a given $1$-factorisation, we would like to ensure that the outneighbourhood of an exceptional vertex is well-distributed on each cycle, in the sense that each cluster $V$ of $C$ contains only a small fraction of the neighbours of any exceptional vertex.
Currently, we cannot guarantee this. But we will be able to achieve this property by considering a refinement of the partition $\mathcal{P}(t)$ for each $t$.
As associated clusters in each $\mathcal{P}(t)$ only differ slightly from one another, we can find this refinement in such a way that the subclusters are also similar by ensuring that all such refinements are close to a refinement of $\tilde{\mathcal{P}}$.

Let $\mathcal{G} = \lbrace G, H_0^+, H_0^-, H_1^+, H_1^-, H_2 \rbrace$.
We now apply Lemma~\ref{equipartition} to our base primary clusters and exceptional set $V_0$ to obtain an $\tilde{\eps}$-uniform 
$s$-refinement $\mathcal{P}_s'$ of $\tilde{\mathcal{P}}$ with respect to $\mathcal{G}$, and we call the resulting subclusters \emph{base} $s$-\emph{clusters}. 
So we have $L_s := s \tilde{L}$ base $s$-clusters. 
Apply Lemma~\ref{equipartition} to $\mathcal{P}_s'$ to obtain an $\tilde{\eps}$-uniform $p$-refinement $\mathcal{P}_p'$ 
of $\mathcal{P}_s'$ with respect to $\mathcal{G}$. Let
\begin{equation} \label{L}
L_p := pL_s = sp\tilde{L}.
\end{equation}
We call the $L_p$ subclusters obtained from an $s$-cluster \emph{base} $p$-\emph{clusters}. 
By the remark before (\ref{doublerefinement}), $\mathcal{P}_p'$ is also a $3\tilde{\eps}$-uniform $sp$-refinement of $\tilde{\mathcal{P}}$. Finally apply 
Lemma~\ref{equipartition} to $\mathcal{P}_p'$ to obtain an $\tilde{\eps}$-uniform $2$-refinement $\mathcal{P}_{2p}'$
of $\mathcal{P}_p'$ with respect to $\mathcal{G}$. The argument before~(\ref{doublerefinement}) implies that $\mathcal{P}_{2p}'$ is a $4\tilde{\eps}$-uniform $2p$-refinement of $\mathcal{P}_s'$ and a $5\tilde{\eps}$-uniform $2sp$-refinement of $\tilde{\mathcal{P}}$. %
\COMMENT{it's a triple refinement, which increases the error bound}
We call the subclusters obtained from an $s$-cluster \emph{base} $2p$-\emph{clusters}. 

Define constants $\eps_s, \eps_p, \eps_{2p}$ such that $\eps \ll \eps_s \ll \eps_p \ll \eps_{2p} \ll \eps'$.
Now do the following for each $t$ with $1 \leq t \leq \tilde{r}$. Apply Lemma~\ref{equipartition2} to $\tilde{\mathcal{P}}$ 
%(with $\eps_1=\tilde{\eps},\eps_2=\eps,\eps'=\eps_s$) 
to obtain an $\eps_s$-uniform $s$-refinement $\mathcal{P}_s(t)$ of $\mathcal{P}(t)$ that is $\eps_s$-close to $\mathcal{P}_s'$. 
Next apply Lemma~\ref{equipartition2} to $\mathcal{P}_s'$ 
%(with $\eps_1=\eps_s,\eps_2=\eps_s,\eps'=\eps_p$) 
to obtain an $\eps_p$-uniform $p$-refinement $\mathcal{P}_p(t)$ of $\mathcal{P}_s(t)$ that is $\eps_p$-close to $\mathcal{P}_p'$. By the observation at the end of Section \ref{unirefs}, $\mathcal{P}_p(t)$ is also an $\eps'$-uniform $sp$-refinement of $\mathcal{P}(t)$.
Finally apply Lemma~\ref{equipartition} to $\mathcal{P}_p'$ 
%(with $\eps_1=\eps_p,\eps_2=\eps_p,\eps'=\eps_{2p}$) 
to obtain an $\eps_{2p}$-uniform $2$-refinement $\mathcal{P}_{2p}(t)$ of $\mathcal{P}_p(t)$ that is $\eps_{2p}$-close to $\mathcal{P}_{2p}'$. Again, $\mathcal{P}_{2p}(t)$ is an $\eps'$-uniform $2p$-refinement of $\mathcal{P}_s(t)$ and an $\eps'$-uniform $2sp$-refinement of $\mathcal{P}(t)$. 
For $j=s,p,2p$ we call the clusters of $\mathcal{P}_j(t)$ the \emph{(adapted)} $j$\emph{-clusters} or $j$-$(t)$-\emph{clusters} if we wish to emphasise the dependence on $t$. 
For each such $j$ we have that $\mathcal{P}_j(t)$ is an $\eps'$-uniform refinement of $\mathcal{P}(t)$ that is $\eps'$-close to $\mathcal{P}_j'$, so each adapted $j$-cluster in $\mathcal{P}_j(t)$ is associated with a unique base $j$-cluster in $\mathcal{P}_j'$.
Write
\begin{equation} \label{eq:m}
m_s := m/s\ \ \mbox{ and }\ \ m_p := m/sp
\end{equation}
for the respective sizes of the $s$- and $p$-clusters (which are the same for all $t$ though the clusters themselves are different). Note that
\begin{equation} \label{mp}
m_p \leq \dfrac{n}{L_p} \leq 2m_p.
\end{equation}

By a slight abuse of notation we can consider $\tilde{R}$ and $\tilde{R}(\beta)$ as digraphs on either base or adapted $(t)$-clusters, depending on the context. For each $0 \leq t \leq \tilde{r}$, we now define corresponding reduced digraphs for the refinements defined above, where, for convenience, $\mathcal{P}_j(0) := \mathcal{P}_j'$. 

Let $R_s= s \otimes \tilde{R}$ be the $s$-fold blow-up of $\tilde{R}$, where for a vertex $W$ of $\tilde{R}$ (which is an adapted primary $(t)$-cluster if $t \geq 1$),
the corresponding vertices in $R_s$ are the subclusters of $W$ in $\cP_s(t)$. Define $R_s(\beta) = s \otimes \tilde{R}(\beta)$ analogously.
Also let $R_p = p \otimes R_s$, where for a vertex $U$ of $R_s$ the corresponding vertices in $R_p$ are the subclusters of $U$ in $\mathcal{P}_p(t)$. 
So the vertices of $R_p$ are precisely the $p$-clusters and also $R_p = sp \otimes \tilde{R}$. 
Define $R_p(\beta) = p \otimes R_s(\beta)$ $= sp \otimes \tilde{R}(\beta)$ analogously.
Note that apart from the fact that the clusters which form their vertex sets are slightly different for different values of $t$, these digraphs are the same, so there is no need for any dependence on $t$ in the notation.

Suppose that $\tilde{E}$ is an edge of $\tilde{R}(\beta)$ from $\tilde{U}$ to $\tilde{W}$
and that $U$ is an $s$-cluster which is a subcluster of $\tilde{U}$ and $W$ is an $s$-cluster which a subcluster of $\tilde{W}$.
Note that there is a unique edge $E$ in $R_s(\beta)$ from $U$ to $W$ corresponding to $\tilde{E}$.
Thus to each edge $E$ of $R_s(\beta)$ we can associate the digraph 
\begin{equation} \label{Gstar}
G^*(E) := G^*(\tilde{E})[U,W].
\end{equation}
We make a similar association for each edge $F$ of $R_p(\beta)$ by defining $G^*(F)$ analogously. 

We now use the 1-factors $\tilde{F}_t$ to define edge-disjoint 1-factors $F_j'$ in the reduced digraph $R_s(\beta)$ and then use the $F_j'$ to find edge-disjoint 1-factors $F_i$ in $R_p(\beta)$.
Note that each cycle $C$ of $\tilde{F}_t$ corresponds to an $s$-fold blow-up $C'$ of $C$ in $R_s(\beta)$. So
for each cycle $C$ in $\tilde{F}_t$ we can 
apply Lemma~\ref{unwind} to obtain an $(s-1)$-unwinding $C_1 , \ldots , C_{s-1}$ of $C'$. 
Here we do not need the special properties of the $(s-1)$-unwinding which are guaranteed by Lemma~\ref{unwind}; in fact any unwinding yielding edge-disjoint Hamilton cycles will do.
So $\tilde{F}_t$ corresponds to a set of $(s-1)$ $1$-factors $F'_j$ (with $(t-1)(s-1)+1 \le j \le t(s-1)$) of $R_s(\beta)$. We say that such an $F_j'$ has \emph{original factor type} $t$ (and that $t$ is the \emph{original type} of $j$).
Note that for each cluster $W$ of $\tilde{R}$, there are $s$ clusters of $F'_j$ which are subclusters of $W$.
Moreover, all of these lie on the same cycle of $F_j'$.
Let 
\begin{equation} \label{rs}
r_s := (s-1) \tilde{r}.
\end{equation}
Then altogether this gives us a set of $r_s$ edge-disjoint $1$-factors
$F'_1 , \ldots , F'_{r_s}$ of $R_s(\beta)$.

Consider any cycle $C_\ell$ of a $1$-factor $F'_j$ obtained from a cycle $C$ of $\tilde{F}_t$
as above. Let $K$ be the length of $C$; so $C_\ell$ has length $K s$.
We say that an $s$-cluster lying on $C_\ell$ is \emph{clean for $F'_j$} if it belongs to the last $K$ clusters of $C_\ell$
(where for each cycle we pick a consistent ordering of its vertices in advance).
Note that $K \ge 2$ and so $C_\ell$ has at least two clean $s$-clusters.%
\COMMENT{Previously, we had `at most' 2 clean clusters but needed exactly two later on.}
Moreover,  for each adapted primary cluster $W$, exactly one subcluster of $W$ in $\cP_s(t)$ is clean for $F_j'$.%
\COMMENT{This needs more detail}
Note that for different $1$-factors $F_j'$, the set of clean clusters will usually be different.

It turns out that we actually need a stronger property than the one described above, namely we need that 
($\star$) below holds.
(This will enable us to ensure that, in the digraphs $G_i$ that we consider later, only a few clusters will contain vertices sending or receiving an edge from the exceptional set and these will be sufficiently far apart.)
For this, we use our refinement $\mathcal{P}_p(t)$ of each $s$-cluster into
$p$ subclusters and unwind the cycles in the above $1$-factorisation again.

For every $V \in \mathcal{P}_s(t)$, let $V^1, \ldots , V^p$ be the $p$-clusters contained in $V$. Note that the collection of all $V^k$ over all $s$-clusters $V$ contained in an adapted primary cluster $W$ are precisely the $p$-clusters refining $W$. For each cycle $D=V_1 \ldots V_K$ in $F'_j$ 
(where this is the same ordering we specified above)
let $D'$ be the $p$-fold blow-up of $D$ whose vertex classes are the $p$-clusters $V_\ell^k$ contained in $V_1, \ldots, V_K$.
Apply Lemma~\ref{unwind} to $D'$ to find a $(p-1)$-unwinding $D_1 , \ldots , D_{p-1}$ of $D'$  with $V_\ell^k$ playing the role of $x_\ell^k$.
We have the following property:

\begin{itemize}
\item[($\star$)] For each $1 \leq d \leq p-1$ and $1 \leq k \leq p$, the $p$-clusters $V_1^k, \ldots, V_{K-2}^k$ have pairwise distance at least $p$ on $D_d$.  
\end{itemize}

Note that ($\star$) holds only for the $p$-clusters in $V_1, \ldots, V_{K-2}$ and not necessarily $V_{K-1}$ or $V_K$. This is the reason for introducing the clean $s$-clusters: recall that $V_{K-1}$ and $V_K$ are clean. This will mean that we will never introduce any edges between their vertices and the exceptional set (see (b) in Section~\ref{redclusters}).

Moreover, for all $1 \leq \ell \leq K$, we have
from Lemma~\ref{unwind} 
that $V_{\ell}^1, \ldots, V_{\ell}^p$ lie on the same cycle $D_d$.
Additionally, their successors on $D_d$ all belong to a single adapted primary cluster $V_{\ell+1}$.
Also $F'_j$ corresponds to a set of $(p-1)$ edge-disjoint $1$-factors $F_i$ (with $(j-1)(p-1)+1 \le i \le j(p-1)$) of $R_p(\beta)$.
We say that such an $F_i$ has \emph{intermediate factor type} $j$
and \emph{original factor type} $t$ where $t$ is the original type of $F_j'$.
For each $i$, write $V_{0,i}^{\rm spec} := \tilde{V}_{0,t}^{\rm spec}$ for the special exceptional set associated with $F_i$,
where $t$ is the original factor type of $F_i$.
Note that for every $i$, every vertex in $G$ is contained either in a $p$-cluster of $F_i$, in $V_0$ or in $V_{0,i}^{\rm spec}$. 
Note also that for each adapted primary cluster $W$ of $\tilde{R}$, there are $sp$ clusters of $F_i$ which are subclusters of $W$. 
Also
let 
\begin{equation} \label{rp}
r_p := (p-1) r_s.
\end{equation}
Then altogether this gives us a set of $r_p$ edge-disjoint $1$-factors
$F_1 , \ldots , F_{r_p}$ of $R_p(\beta)$. 
Note that for each $t$, there are exactly $(s-1)(p-1)$ of the $F_i$ which have original factor type $t$.
Furthermore,
\begin{equation} \label{eq:boundrp}
r_p \stackrel{(\ref{eq:defr})}{\le} \frac{\tilde{\alpha}sp \tilde{L} }{\beta} \ \ \mbox{ and  so } \ \ r_p \stackrel{(\ref{L})}{\leq} \dfrac{\tilde{\alpha}L_p}{\beta}; \ \ \mbox{ also } \ \ 1/m \le 1/m_p \stackrel{(\ref{boundL})}{\ll} 1/r_p.
\end{equation}

For each edge $E \in E(F_i)$ from $A$ to $B$, let $G_i(E) := G^*(E)$, where $G^*(E)$ was defined just after~(\ref{Gstar}). 
Let $G_i$ denote the union of the digraphs $G_i(E)$ over all $E$ with $E \in E(F_i)$ and call it the $i$th \emph{slice}.
Clearly $G_1 , \ldots , G_{r_p}$ are edge-disjoint.
Given $E \in E(F_i)$ for some $1 \leq i \leq r_p$, let $\tilde{E} \in E(\tilde{F}_t)$ be the unique edge such that $E$ is in the blow-up of $\tilde{E}$, where $F_i$ has original factor type $t$.
As noted directly before (\ref{boundL}),
$G^*(\tilde{E})$ is $(\eps , \beta_1)$-superregular and hence $G_i(E)=G^*(E)$ is $(\eps^{\prime} , \beta_1)$-superregular by Lemma~\ref{equipartition}(i).

Recall that since $V_{0,i}^{\rm spec}$ is different for each $i$, the vertex set of a $p$-cluster will be slightly different in $G_i$ and $G_{i'}$ when $F_i$ and $F_{i'}$ have different original factor types. 
Note that if $U$ is a base $2p$-cluster (of size $m_p/2$), and $U_{(t)}$ is the associated $2p$-$(t)$-cluster, then
\begin{equation} \label{largeint}
|U \cap U_{(t)}| \geq (1-\eps')m_p/2.
\end{equation}
as the corresponding partitions are $\eps'$-close.
(On the other hand, $\bigcap_{1 \leq t \leq \tilde{r}}{U_{(t)}}$ may be empty.)
The same statements hold for $s$- and $p$-clusters. 
When adding edges incident to exceptional vertices in Section \ref{exceptional} we need to be careful about distinguishing between base $2p$-clusters and the $2p$-$(t)$-clusters which are actually contained in the clusters of our slices.

%%%%%%%%%%%%%%%%%%%%%%%%%%%%%%%%%%%%%%%%%%%%%%%%%%%%%%%%%%%%%%%%

\subsection{Red clusters and edges} \label{redclusters}

The aim of this Section is to lay some groundwork for Sections \ref{connect}, \ref{exceptional} and \ref{exceptional2}
by specifying the properties that the edges between the exceptional vertices and the rest of $V(G)$ need to satisfy. 
In Section \ref{connect} our aim is to remove a bounded number of \emph{bridge vertices} $V_{0,i}^{\rm bridge}$ from each $G_i \setminus (V_0 \cup V_{0,i}^{\rm spec})$ and 
change their neighbourhoods in such a way that the blown-up cycles in $G_i$ are connected via bridge vertices. Some additional vertices will also be removed and added to $V_{0,i}^{\rm spec}$ to keep the cluster sizes equal.
In Section~\ref{exceptional} we will add edges to $G_i$ which are incident to $V_0$.
In Section~\ref{exceptional2} we will do the same for $V_{0,i}^{\rm spec}$.
We will then let
\[
V_{0,i} := V_0 \cup V_{0,i}^{\rm spec} \cup V_{0,i}^{\rm bridge}.
\]
$V_{0,i}$ is then the exceptional set for the slice $G_i$: each vertex will lie either in a cluster of $G_i$ or in $V_{0,i}$. Any edge incident to a vertex in $V_{0,i}$ and any vertex in a cluster of $G_i$ incident to such an edge will be called $i$-\emph{red} (or \emph{red} if this is unambiguous).
Roughly speaking, when adding red edges to $G_i$, we will need to ensure that $G_i$ is a spanning almost-regular digraph, 
that no non-exceptional vertex has large $i$-red degree and that the set of red vertices is small and well-distributed.

To achieve this, for each $i$ we will only add red edges incident to some carefully selected $2p$-clusters
and then apply property ($\star$).
More precisely, for fixed $i$, let $j=j(i)$ and $t=t(i)$ respectively be the intermediate and original factor types of $G_i$.
For each $s$-cluster $U$ of $\cP_s(t)$, let $U_1,\dots ,U_p$ denote the $p$-clusters of $\cP_{p}(t)$ which are subclusters of $U$. For $1 \leq \ell \leq p$, let $U(\ell)$ and $U(\ell+p)$ be the $2p$-clusters contained in $U_\ell$.
In $G_i$, we will add red edges between $V_{0,i}$ and $U(k)$ only if 
\begin{itemize}
\item[(a)] $t \equiv k \mod 2p$ and 
\item[(b)] $U$ is not a clean cluster in $F'_{j}$.
\end{itemize}
We call such a $2p$-cluster $U(k)$ \emph{$i$-red} and we call a $p$-cluster $i$-\emph{red} if it contains an $i$-red $2p$-cluster (or simply \emph{red} if this is unambiguous).
Note that (a) implies that every $s$-cluster $U$ which is not clean contains exactly one red $2p$-cluster (and thus exactly one red $p$-cluster).
Moreover, recall that any adapted primary cluster contains exactly one clean $s$-cluster; thus it contains exactly $s-1$ red $p$-clusters.

All red vertices will be contained in red $2p$-clusters, but note that we do not require every red cluster to contain a red vertex.
Let 
\begin{equation} \label{eq:kappa}
\kappa := (1 - \gamma)\beta_1 m_p.
\end{equation}
We would like to find exactly $\kappa$ edge-disjoint Hamilton cycles in each of the $G_i$.
For this, we will first need to add edges so that $G_i$ satisfies the following for all $i$ with $1 \leq i \leq r_p$:

\medskip

\begin{enumerate}
\item[(Red0)] There exists a sequence $D_1x_1D_2x_2\dots x_{\ell-1}D_\ell x_\ell D_1$
with the following properties:
\begin{itemize}
\item Each $D_j$ is a cycle of $F_i$ and every cycle of $F_i$ appears at least once in the sequence;
\item $V_{0,i}^{\rm bridge}:=\{x_1,\dots,x_\ell\}$ and each $x_j$ has exactly $\kappa$ outneighbours in $D_{j+1}$ and exactly $\kappa$ inneighbours in $D_{j}$.
\end{itemize}
\item[(Red1)] $d^\pm_{G_i} (x) = \kappa$  for all $x \in V_{0,i}$;
\item[(Red2)] $V_{0,i}$ is an independent set in $G_i$;
\item[(Red3)] $| N^\pm_{G_i}(y) \cap V_{0,i} | \leq \sqrt{\xi} \beta_1 m_p$ for all $y \notin V_{0,i}$;
\item[(Red4)] For every red $p$-cluster $V$, all red edges of $G_i$ are incident to a single $2p$-cluster contained in $V$. In particular, $| N^\pm_{G_i}(V_{0,i}) \cap V | \leq m_p /2$ for all clusters $V \in R_p$;
\item[(Red5)] If $V,V'$ are red $p$-clusters on a cycle $C$ of $F_i$, then they have distance at least $p$ on $C$;
\item[(Red6)] If a $p$-cluster $V$ contains the final vertex of a red edge in $G_i$, then it contains no initial vertices of red edges in $G_i$, and vice versa;
\item[(Red7)]  $G_1 , \ldots , G_{r_p}$ are edge-disjoint and $G_i(E)$ is $(2\eps',\beta_1)$-superregular for all $E \in E(F_i)$.
\end{enumerate}

\medskip

Roughly speaking, given a 1-factor $f$ of $G_i$, (Red0) and (Red1) will ensure that $f$ has a path between any pair of successive cycles $D_j$ of $F_i$. (Red2)--(Red6) imply that the red edges are well-distributed. This will be crucial when applying Lemma~\ref{trick} to transform $f$ into a Hamilton cycle in Section~\ref{mergeH}.

Suppose that $V$ is a red $p$-cluster. Since only one of the two $2p$-clusters contained in $V$ is red,
it follows that (Red4) will automatically be satisfied for $V$ if we add red edges according to (a) and (b).
It is easy to see that this will also satisfy (Red5). Indeed, recall that every non-clean $s$-cluster contains exactly one red $p$-cluster. Moreover, if $U_\ell$ and $U_{\ell'}$ are non-clean $s$-clusters then the $p$-cluster $U_\ell^k$ is red if and only if $U_{\ell'}^k$ is red.
Suppose a cycle $C$ in $F_i$ was obtained by unwinding the blow-up of $C'$ in $F_j'$; then the last two $s$-clusters in $C'$ (using the same ordering as in Section \ref{unwinding}) are clean and hence contain no red $p$-clusters by (b). So ($\star$) implies that the red clusters on $C$ will have distance at least $p$ apart.%
\COMMENT{Add detail on how the red clusters correspond to vertices in the statement of Lemma~\ref{unwind}}

Let $F = r_p/4p$. For each $k$ with $1 \le k \le 2p$, 
note that the number of all digraphs $G_i$ whose original type $t$ satisfies (a)
is $2F$. For each $k$, consider an ordering of all these graphs.

We will now fix which of the (red) $2p$-clusters will receive red edges and which of them will send out red edges. 
For each $i=1,\dots,r_p$, let $t=t(i)$ be the original type of $G_i$ and let $k$ with $1 \le k \le 2p$ satisfy (a).
Suppose that $G_i$ is the $f$th graph with original type $t$ (so $1 \le f \le 2F$). 
For each adapted primary cluster $W \in \mathcal{P}(t)$, let $W_1,\dots,W_s$ denote the set of all $s$-clusters contained in $W$.
Recall that exactly one of the $W_j$ is clean.
Now choose a set $S_W^+$ of $s$-clusters from  $W_1,\dots,W_{s/2}$ 
so that none of the $s$-clusters in $S_W^+$ is clean and so that  $|S_W^+|=s/2-1$ (recall that $s$ is even).
Let $I_W^+$ denote the set of indices of the $s$-clusters in $S_W^+$.%
\COMMENT{the previous definition via `Removing clusters' might give the impression that we through them into the exceptional set}
Similarly choose $S_W^-$ from $W_{s/2+1},\dots,W_{s}$ with $|S_W^-|=s/2-1$ which avoids the clean cluster and 
let $I_W^-$ be the corresponding set of indices.

For each $s$-cluster $W_\ell$ contained in $W$, let $W_\ell(k)$ denote the $k$th $2p$-cluster contained in $W_\ell$, where $k$ is defined as in the previous paragraph. 
We call  $W_\ell(k)$ \emph{in-red} (for $i$) if 
\begin{itemize}
\item $1 \leq f \leq F$ and $\ell \in I_W^+$, or $F < f \leq 2F$ and $\ell \in I_W^-$.
\end{itemize}
We call $W_\ell(k)$ \emph{out-red} (for $i$) if 
\begin{itemize}
\item $1 \leq f \leq F$ and $\ell \in I_W^-$, or $F < f \leq 2F$ and $\ell \in I_W^+$.
\end{itemize}
If a $p$-cluster $V$ contains an in-red $2p$-cluster, we say that $V$ is \emph{in-red} (and similarly for \emph{out-red} clusters).
So the number of in-red $p$-clusters in each adapted primary cluster $W$ is exactly
\begin{equation} \label{inred}
|I_W^-| = \dfrac{s}{2}-1
\end{equation} 
and similarly for out-red clusters.

%%%%%%%%%%%%%%%%%%%%%%%%%%%%%%%%%5

\subsection{Connecting blown-up cycles} \label{connect}

In the final section of the proof we will successively find $1$-factors in each $G_i$ and then turn each of these into a Hamilton cycle. 
As mentioned earlier, (Red0) will be used to ensure that each 1-factor $f$ of $G_i$ has a path connecting any pair of consecutive cycles of $F_i$,
which will make it possible to merge the cycles of $f$ into a Hamilton cycle. In this section we will modify $G_i$ so that (Red0) holds. 

We will join cycles by choosing \emph{bridge vertices} $x_{i,j}$ in $V(G)\setminus (V_0 \cup V_{0,i}^{\rm spec})$ whose neighbourhoods will be chosen from the sparse digraphs $H_0^\pm$ defined in Section~\ref{sec:H}. In what follows, we write $A_j^-$ for the predecessor of the $p$-cluster $A_j$ in $F_i$. 

\tikzstyle{every node}=[circle, draw, fill=white,
                        inner sep=0pt, minimum width=9.5pt]
\tikzstyle{every line}=[width=10mm]
\begin{center}
\scalefont{0.75}
{
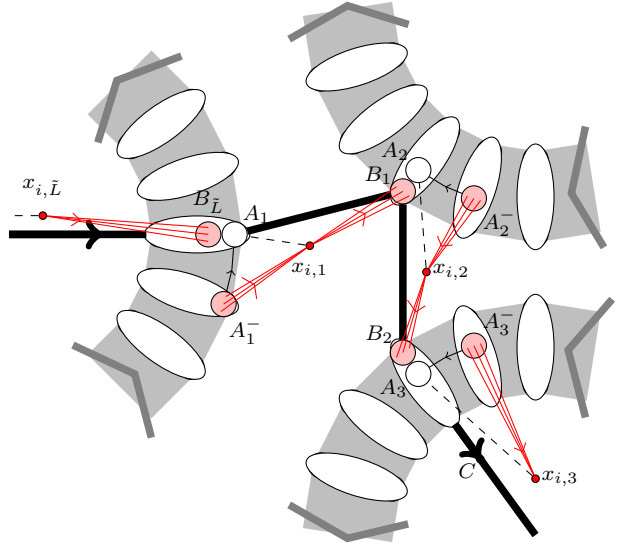
\begin{figure}
\centering
\begin{tikzpicture}[scale=1]
\begin{scope}[xshift=0cm,yshift=1cm]
    \draw[line width=12mm,color=gray!50]
     \foreach \x in {-54,-36,-18,0,18,36}
    {
        (\x:2.5) -- (\x+18:2.5)
    };

           \draw
 \foreach \x in {-36,-18,0,18,36}
    {
        (\x:3) node (\x four) {}
        (\x:2) node (\x four2) {}
        (\x:2.65) node (\x four3) {}
    };
\draw[fill=white,rotate=-36] (0:2.5) ellipse (20pt and 7pt);
\draw[fill=white,rotate=-18] (0:2.5) ellipse (20pt and 7pt);
\draw[label=left:{$\tilde{A}_1$},fill=white,rotate=0] (0:2.5) ellipse (20pt and 7pt);
\draw[fill=white,rotate=18] (0:2.5) ellipse (20pt and 7pt);
\draw[fill=white,rotate=36] (0:2.5) ellipse (20pt and 7pt);

\draw[->] (-18:3) -- (-9:3);
\draw (-9:3) -- (0:3);

\node at (0four) [label=40:{$A_1$}] {};
\node at (0:2.65) [fill=red!25,label={$B_{\tilde{L}}$}] {};
\node at (-18four) [fill=red!25,label=300:{$A_1^-$}] {};

\draw[line width=1mm,color=gray] (-60:1.7) -- (-48:2.5) -- (-55:3.3);
\draw[line width=1mm,color=gray] (46:1.7) -- (55:2.5) -- (46:3.3);
\end{scope}
\begin{scope}[yshift=4cm,xshift=7cm,]
    \draw[line width=12mm,color=gray!50]
     \foreach \x in {180,198,216,234,252,270}
    {
        (\x:2.5) -- (\x+18:2.5)
    };
           \draw
 \foreach \x in {198,216,234,252,270}
    {
        (\x:3) node (\x five) {}
        (\x:2) node (\x five2) {}
        (\x:2.65) node (\x five3) {}
    };

\draw[fill=white,rotate=198] (0:2.5) ellipse (20pt and 7pt);
\draw[fill=white,rotate=216] (0:2.5) ellipse (20pt and 7pt);
\draw[fill=white,rotate=234] (0:2.5) ellipse (20pt and 7pt);
\draw[fill=white,rotate=252] (0:2.5) ellipse (20pt and 7pt);
\draw[fill=white,rotate=270] (0:2.5) ellipse (20pt and 7pt);

\draw[->] (252:2.65) -- (243:2.65);
\draw (243:2.65) -- (234:2.65);

\node at (252:2.65) [fill=red!25,label=320:{$A_2^-$}] {};
\node at (234five) [fill=red!25,label=160:{$B_1$}] {};
\node at (234:2.65) [label=150:{$A_2$}] {};

\draw[line width=1mm,color=gray] (187:1.7) -- (179:2.5) -- (187:3.3);
\draw[line width=1mm,color=gray] (293:1.7) -- (281:2.5) -- (286:3.3);
\end{scope}
\begin{scope}[yshift=-3cm,xshift=7cm,]
\node at (126:1.25) [minimum width=1mm,label=210:{$C$}] {};
\draw[<-,line width=1mm,color=black] (126:1.25) -- (126:2.5);
    \draw[line width=12mm,color=gray!50]
     \foreach \x in {72,90,108,126,144,162}
    {
        (\x:2.5) -- (\x+18:2.5)
    };
\draw
 \foreach \x in {90,108,126,144,162}
    {
        (\x:3) node  (\x six) {}
        (\x:2) node (\x six2) {}
        (\x:2.65) node (\x six3) {}
    };

\draw[fill=white,rotate=90] (0:2.5) ellipse (20pt and 7pt);
\draw[fill=white,rotate=108] (0:2.5) ellipse (20pt and 7pt);
\draw[fill=white,rotate=126] (0:2.5) ellipse (20pt and 7pt);
\draw[fill=white,rotate=144] (0:2.5) ellipse (20pt and 7pt);
\draw[fill=white,rotate=162] (0:2.5) ellipse (20pt and 7pt);

\draw[->] (108:2.65) -- (117:2.65);
\draw (117:2.65) -- (126:2.65);

\node at (126six) [fill=red!25,label=150:{$B_2$}] {};
\node at (126:2.65) [label=190:{$A_3$}] {};
\node at (108:2.65) [fill=red!25,label=40:{$A_3^-$}] {};

\draw[line width=1mm,color=gray] (67:1.7) -- (80:2.5) -- (72:3.3);
\draw[line width=1mm,color=gray] (175:1.7) -- (181:2.5) -- (173:3.3);
\end{scope}
\draw[line width=1mm,color=black] (0,1) -- (0four2);
\draw[line width=1mm,color=black] (0four) -- (234five);
\draw[line width=1mm,color=black] (234five) -- (126six);
\draw[line width=1mm,color=black] (126six2) -- (7,-3);
\draw[->,line width=1mm,color=black] (0,1) -- (1.25,1);

\node at (4,0.85) [fill=red,draw,circle,minimum width=1mm,
label=below:{$x_{i,1}$}] (x1){};
\node at (5.55,0.5) [fill=red,draw,circle,minimum width=1mm,
label=right:{$x_{i,2}$}] (x2){};
\node at (0.45,1.25) [fill=red,draw,circle,minimum width=1mm,
label=above:{$x_{i,\tilde{L}}$}] (xL){};
\node at (7,-2.25) [fill=red,draw,circle,minimum width=1mm,
label=right:{$x_{i,3}$}] (x3){};
\draw[style=dashed] (x1) -- (0four);
\draw[style=dashed] (x2) -- (234five3);
\draw[style=dashed] (x3) -- (126six3);
%\draw[style=dashed] (x3) -- (7.15,-2.75);
\draw[style=dashed] (xL) -- (0,1.25);

\begin{scope}[xshift=0cm,yshift=1cm]
%\draw[color=red] (x1) -- (-14:3);
\draw[color=red] (x1) -- (-16:3);
\draw[color=red] (x1) -- (-18:3);
\draw[color=red] (x1) -- (-20:3);
%\draw[color=red] (x1) -- (-22:3);

\draw[color=red] (xL) -- (-2:2.65);
\draw[color=red] (xL) -- (0:2.65);
\draw[color=red] (xL) -- (2:2.65);

\end{scope}

\begin{scope}[yshift=-3cm,xshift=7cm]
\draw[color=red] (x3) -- (106:2.65);
\draw[color=red] (x3) -- (108:2.65);
\draw[color=red] (x3) -- (110:2.65);
\end{scope}

\begin{scope}[shift={(x1)}]
\draw[color=red] (205:1) -- (215:0.9) -- (225:1);
\draw[color=red] (20:0.75) -- (30:0.85) -- (40:0.75);
\end{scope}

\begin{scope}[xshift=7cm,yshift=4cm]
%\draw[color=red] (x2) -- (248:2.65);
\draw[color=red] (x2) -- (250:2.65);
\draw[color=red] (x2) -- (252:2.65);
\draw[color=red] (x2) -- (254:2.65);
%\draw[color=red] (x2) -- (256:2.65);

\begin{scope}[shift={(x2)}]
\draw[color=red] (46:0.5) -- (60:0.4) -- (73:0.5);
\draw[color=red] (236:0.45) -- (255:0.55) -- (268:0.45);
\end{scope}

\begin{scope}[shift={(x3)}]
\draw[color=red] (105:0.5) -- (115:0.4) -- (125:0.5);
\end{scope}

\begin{scope}[shift={(xL)}]
\draw[color=red] (5:0.4) -- (-7:0.5) -- (-20:0.4);
\end{scope}

%\draw[color=red] (256:3.3) -- (252:3.4) -- (248:3.3);
%\draw[color=red] (254.5:4.4) -- (252.5:4.5) -- (250.5:4.4);

%\draw[color=red] (x1) -- (230:3);
\draw[color=red] (x1) -- (232:3);
\draw[color=red] (x1) -- (234:3);
\draw[color=red] (x1) -- (236:3);
%\draw[color=red] (x1) -- (238:3);

\end{scope}

\begin{scope}[xshift=7cm,yshift=-3cm]
%\draw[color=red] (x2) -- (122:3);
\draw[color=red] (x2) -- (124:3);
\draw[color=red] (x2) -- (126:3);
\draw[color=red] (x2) -- (128:3);
%\draw[color=red] (x2) -- (130:3);
\end{scope}
\end{tikzpicture}\quad
\caption{Bridge vertices $x_{i,1} , x_{i,2}, x_{i,3}$ chosen from $p$-clusters $A_1, A_2, A_3$ respectively.}\label{fig:bridge}
\end{figure}
}
\end{center}
 
\begin{claim} \label{connect1} 
There is a sequence $A_1B_1A_2B_2 \ldots A_{\tilde{L}}B_{\tilde{L}}A_1$ of $p$-clusters in $R_p$ such that, for each $1 \leq j,j' \leq \tilde{L}$ the following hold:
\begin{itemize}
\item[(i)] Let  $\tilde{A}_j$ and $\tilde{B}_j$ be the adapted primary clusters containing $A_j$ and $B_j$ respectively.
Then there is an edge from $\tilde{A}_j$ to $\tilde{B}_j$ in $\tilde{R}$.
\item[(ii)] $A^-_j$ is out-red and $B_j$ is in-red;
\item[(iii)] $B_j$ and $A_{j+1}$ lie in the same adapted primary cluster (where $A_{\tilde{L}+1}:=A_1$).
\item[(iv)] Every adapted cluster contains exactly one $A_j$ and exactly one $B_{j^\prime}$.
\item[(v)] All the $A_j$ and $B_{j'}$ are distinct.
\end{itemize}
\end{claim}

To prove the claim, observe that, by Lemma~\ref{inherit}, $\tilde{R}$ is a robust $(\nu/4 , 3\tau)$-outexpander, so Theorem~\ref{1hc} implies that $\tilde{R}$ contains a 
Hamilton cycle~$C=\tilde{A}_1\dots \tilde{A}_{\tilde{L}}$.
We will choose $A_j$ in $\tilde{A}_j$ and $B_j$ in $\tilde{A}_{j+1}$.
This will automatically satisfy (i), (iii) and (iv) and ensures that they will be distinct, except possibly $A_j = B_{j-1}$. 
Now recall that, by (\ref{inred}), each adapted primary cluster contains exactly $s/2-1$ in-red and $s/2-1$ out-red $p$-clusters. Moreover, as noted after ($\star$), they all lie on the same cycle in $F_i$
and $p$-clusters directly preceding those in $\tilde{A}_j$ on $F_i$ all lie in the same adapted primary cluster, which we call $\tilde{A}_j^-$. 
Thus we can always choose an in-red $B_{j-1}$ in $\tilde{A}_j$ and an out-red $A_{j}^- \in \tilde{A}_j^-$ whose successor $A_j$ on $F_i$ lies in $\tilde{A}_j$, proving (ii). Moreover, we have $s/2-1 > 1$ choices for ${A}_j$ so we may assume that $A_j$ and $B_{j-1}$ are distinct. This proves (v) and thus the claim. 

\medskip

\noindent We will choose the bridge vertices in the sets $A_j$. The next claim guarantees many candidates for these bridge vertices whose neighbourhoods have the required properties.

\begin{claim} \label{connect2} 
For each $i$ with $1 \leq i \leq r_p$, whenever a $p$-cluster $A_j \in V(F_i)$ is joined by an edge in $R_p$ to a $p$-cluster $B_j \in V(F_i)$ the following holds.%
\COMMENT{we need to be more careful about referring to $p$-clusters, as they are slightly different for each slice}
Let $(A_j^-)'$ and $B_j'$ be $2p$-clusters contained in $A_j^-$ and $B_j$ respectively.
Then $A_j$ contains at least $m_p/2$ vertices $x$ such that both $\vert N_{H_0^-}^-(x) \cap (A_j^-)' \vert > \kappa$%
\COMMENT{we do need to know this inneighbourhood lies in $A_j^-$} 
and $\vert N_{H_0^+}^+(x) \cap B_j' \vert > \kappa$.
\end{claim}

\noindent
We say that a vertex $x$ as in Claim~\ref{connect2} is \emph{$(i,j)$-useful}.
To prove the claim, note that (H0)%
\COMMENT{, (\ref{boundL}) to show that $m$ and $\tilde{m}$ are close in size} and (URef) imply that, for at least $3m_p/4$ of the vertices $x \in A_j$, we have
$$
\vert N^+_{H_0^+}(x) \cap B_j' \vert \geq \gamma d m_p/5 \stackrel{(\ref{eq:kappa})}{>}  \kappa.
$$ 
As $F_i$ (and thus $R_p$) contains the edge $A_j^-A_j$, for at least $3m_p/4$ of the vertices $x \in A_j$ we similarly have
$$
\vert N^-_{H_0^-}(x) \cap (A_j^-)' \vert >  \kappa.
$$ 
So at least $m_p/2$ the vertices in $A_j$ satisfy both inequalities, which proves the claim.

\medskip

\noindent 
Now we choose the set $V_{0,i}^{\rm bridge}$ satisfying (Red0).
For each $1 \leq i \leq r_p$, consider the sequence guaranteed by Claim \ref{connect1}
and for each $1 \leq j \leq \tilde{L}$, let $(A_j^-)^{\rm red}$ and $B_j^{\rm red}$ be the unique red $2p$-clusters contained in $A_j^-$ and $B_j$ respectively.
So $(A_j^-)^{\rm red}$ is out-red and $B_j^{\rm red}$ is in-red.
For each $1 \leq j \leq \tilde{L}$, apply Claim~\ref{connect2} to the pair $(A_j,B_j)$ with $(A_j^-)^{\rm red}, B_j^{\rm red}$ playing the roles of $(A_j^-)', B_j'$ respectively, to obtain a vertex $x_{i,j} \in A_j$ which is $(i,j)$-useful and which is distinct from all vertices
chosen so far. Note that the latter is possible since Claim~\ref{connect1}(v) implies that for each $i$, we only choose one vertex from  $A_j$. So altogether, we choose at most $r_p$
vertices from each $A_j$, which is at most $m_p/3$ by 
(\ref{eq:boundrp}). 
In each $G_i$, remove each $x_{i,j}$ from $A_j$ and denote the collection of all $x_{i,j}$ with $1 \leq j \leq \tilde{L}$ by $V_{0,i}^{\rm bridge}$. 
This process is illustrated in Figure \ref{fig:bridge}.

Now, for each $1 \leq i \leq r_p$, there are exactly $sp-1$ of the $p$-clusters in each adapted primary $(t)$-cluster of $G_i$ from which no vertices have been removed in this step, where $t$ is the original type of $F_i$.%
\COMMENT{Because each adapted cluster contains $sp$ $p$-clusters and in $G_i$, we only took a bridge vertex from one of these}
We still need to ensure that $p$-clusters of $G_i$ have equal size, so we choose a further $(sp-1)r_p\tilde{L} \leq \eps n/3$ 
distinct vertices such that exactly one is removed from each untouched $p$-cluster in each $G_i$. 
Each such vertex is moved from its cluster  into $V_{0,i}^{\rm spec}$.
The final inequality in~(\ref{eq:boundrp}) implies that we can assume that each
vertex $x$ is moved into $V_{0,i}^{\rm spec}$ for at most one $1 \leq i \leq r_p$ in this step.%
 \COMMENT{the way this was described in the previous version, this created loads of red vertices, because (as we discussed some time ago) one has to be careful how to connect the further exceptional vertices}
Now adapted primary $(t)$-clusters become \emph{adapted primary $[i]$-clusters} and (adapted) $s$-, $p$- and $2p$-$(t)$-clusters become \emph{(adapted)} $s$-, $p$- and $2p$-$[i]$-\emph{clusters} respectively (or $s$-, $p$-, $2p$-\emph{clusters} if this is unambiguous).
This will not overlap with previous notation as from now on we never refer to $(t)$-clusters and only ever refer to base and $[i]$-clusters. (\ref{largeint}) implies that if $U$ is a base $2p$-cluster and $U_{[i]}$ is the associated $2p$-$[i]$-cluster, then
\begin{equation} \label{largeint2}
| U \cap U_{[i]} | \geq (1-\eps')m_p/2 -1.
\end{equation}
We still refer to the cluster sizes $m, m_s$ and $m_p$ in the same way since each one has only lost at most $sp$ vertices (which does not affect any calculations).
The $2p$-clusters may no longer have exactly the same size, but this also does not affect any of the calculations.
We call $V_{0,i}^{\rm bridge}$ the set of \emph{bridge vertices} and say that every edge incident to a bridge vertex is $i$-red. We now have that
\begin{equation} \label{V0ibound}
|V_{0,i}| \stackrel{(\ref{tildeV0})}{\leq} \eps n.
\end{equation}
Since in $G_i$, we removed exactly one vertex from each $p$-cluster, we still have 
$$
\vert N^+_{H_0^+}(x_{i,j}) \cap B_j^{\rm red} \vert \geq \kappa\ \ \mbox{ and }\ \ |N^-_{H_0^-}(x_{i,j}) \cap (A_j^-)^{\rm red} \vert \ge \kappa.
$$  
Since the $x_{i,j}$ are all distinct,  it follows that for each $x_{i,j}$, we can choose $\kappa$
of these outedges from $H_0^+$ and add them to $G_i$. Similarly, we can choose $\kappa$ of these inedges from $H_0^-$ and add them to $G_i$, whilst also removing every other edge incident to $x_{i,j}$ in $G_i$.%
\COMMENT{(Note that we could have chosen inedges to $x_{i,j}$ from its neighbourhood in $G_i$ but then we could only guarantee half the edges required from the $2p$-cluster $(A_j^-)^{\rm red}$.) }
So (Red1) is satisfied for $V_{0,i}^{\rm bridge}$.

It is now easy to verify (Red0). For each $1 \leq i \leq r_p$, consider the sequence given by Claim \ref{connect1}. Let $D_j$ be the cycle of $F_i$ containing the adapted $p$-cluster $A_j$ for each $1 \leq j \leq \tilde{L}$ and let $x_j := x_{i,j}$ be the bridge vertex which was removed from $A_j$ in $G_i$. Note that each cycle of $F_i$ appears several times in the sequence. 
We claim that $D_1x_1D_2x_2\dots x_{\tilde{L}}D_\ell x_{\tilde{L}} D_1$ is a sequence satisfying (Red0). The first property is immediate from Claim \ref{connect1}(iv). Each $x_j$ has inneighbourhood contained in $A_j^-$ which is in $D_j$ since $A_j$ is, and its outneighbourhood is contained in $B_j$ which lies in the same adapted cluster as $A_{j+1}$ and thus in the cycle $D_{j+1}$. Therefore the second property is also satisfied.

(Red2) follows since the in- and outedges incident to bridge vertices were chosen from edge-disjoint subdigraphs $H_0^-$ and $H_0^+$ respectively.
Furthermore, by Claim~\ref{connect1}(v), any $y \notin V_{0,i}$ is incident to at most one $i$-red edge so (Red3) holds.
In each red $p$-cluster $V$, red edges were only added to the unique red $2p$-cluster $W$ contained in $V$ so (Red4) is satisfied.
(Red5) is satisfied by the comments after the statement of (Red7).
Moreover every out-red $p$-cluster only sends out red edges and every in-red $p$-cluster only receives red edges so (Red6) holds.  
The edge-disjointness in (Red7) is immediate from the construction.
Finally, note that any vertex in $V(G) \setminus V_{0,i}$ lost at most one inneighbour and one outneighbour in $G_i$, so for each edge $E$ of $F_i$, $G_i(E)$ is certainly still $(2\eps',\beta_1)$-superregular. 
Therefore (Red0) and (Red2)--(Red7) are all satisfied. 
Note that (Red1) holds for all vertices in $V_{0,i}^{\rm bridge}$.
The aim of the next two sections is to maintain these properties while also achieving (Red1) for all vertices in $V_{0,i}$.

%%%%%%%%%%%%%%%%%%%%%%%%%%%%%%%%%%%%%%%%%%%%%%%%%%%%%%%%%%%%%%%%

\subsection{Incorporating the core exceptional set $V_0$} \label{exceptional}
 
Note that so far, $G_i$ contains no edges with initial or final vertex in $V_0 \cup V_{0,i}^{\rm spec}$.
In this section and the next we will add edges incident to these vertices into the $G_i$. 
Recall that we call such edges and any incident vertices $i$-\emph{red} or \emph{red} if this is unambiguous. Throughout both sections we will refer to (a) and (b) in Section \ref{redclusters}.
To achieve (Red1), we consider the core exceptional set $V_0$ and the special exceptional set $V_{0,i}^{\rm spec}$ separately.
In this section we consider the core exceptional set. 
Roughly speaking, the set of edges between $V_0$ and $G_i \setminus V_0$ will consist of a random subdigraph of $G$ induced by $V_0$ and the red $2p$-clusters of $G_i$.
The following claim guarantees the existence of suitable edge-disjoint random subdigraphs.
Recall from Section~\ref{redclusters} that $F=r_p/4p$.

\begin{claim}\label{claim1} Let $X$ be a base $2p$-cluster which is a subcluster of a base primary cluster $W$. 
Then in $G$ we can find $F$ edge-disjoint bipartite graphs $E_1^+(X),\dots,E_F^+(X)$ with all edges oriented from $V_0$ to $X$ so that for all
$1 \le f \le F$ the following hold:
\begin{itemize}
\item[(i)] For all $x \in V_0$ we have $d^+_{E_f^+(X)}(x) \ge \frac{1-\eps}{2sp F}\left(|N^+_G(x)  \cap W| - 5\tilde{\eps}\tilde{m}\right)$.
\item[(ii)] For all $y \in X$ we have $d^-_{E_f^+(X)}(y) < \sqrt{\xi}\beta_1m_p/2$.
\end{itemize}
We can also find $E_1^-(X),\dots,E_F^-(X)$ satisfying analogous properties for the inneighbourhoods.%
\end{claim}

\noindent
To prove the claim, let $E^+(X)$ denote the digraph induced by the set of edges from $V_0$ to $X$ in $G$.
Now consider a random partition of the edges of $E^+(X)$ into $F$ parts $E_f^+(X)$.
More precisely, assign each edge of $E^+(X)$ to $E_f^+(X)$ with probability $1/F$, independently of all other edges.
There are several cases to consider. Say that $x \in V_0$ is \emph{prolific} if $\vert N_G^+(x) \cap W \vert > 5\tilde{\eps}\tilde{m}$. Say that $V_0$ is \emph{large} if $|V_0| \geq \sqrt{\xi}\beta_1m_p/2$ and \emph{small} otherwise.
Every $x \in V_0$ which is not prolific satisfies the condition in (i) with probability 1, and the inequality in (ii) is satisfied with probability 1 if $V_0$ is small. Suppose that $x$ is prolific. 
Then since $\cP_{2p}'$ is a $5\tilde{\eps}$-uniform $2sp$-refinement of $\tilde{\mathcal{P}}$,  (URef) implies that 
$d^+_{E^+(X)}(x) \ge \frac{1-5\tilde{\eps}}{2sp }|N^+_G(x) \cap W|.
$ 

Then for each $1 \leq f \leq F$, each prolific $x \in V_0$ and each $y \in X$,
$$
\ex \left( d^+_{E^+_f(X)}(x) \right) \ge \frac{1-5\tilde{\eps}}{2sp F}|N^+_G(x) \cap W| \ \ \mbox{ and } \ \ \ex \left( d^-_{E_f^+(X)}(y) \right) \le \frac{|V_0|}{F}.
$$ 

By Proposition~\ref{chernoff} (with $a := \eps/2$) we have that, for fixed $f$ and prolific $x \in V_0$,
\begin{eqnarray*}
\mathbb{P} \left( d^+_{E_f^+(X)}(x) \leq \dfrac{1-\eps}{2spF} |N_G^+(x) \cap W | \right) &\leq& \exp \left(-\dfrac{5\eps^2\tilde{\eps} \tilde{m}(1-5\tilde{\eps})}{24spF}\right)\\
\nonumber &\stackrel{(\ref{eq:boundrp})}{\leq}& \exp \left( -\dfrac{\tilde{\eps}^2\beta n}{s^2p\tilde{L}^2} \right) \stackrel{(\ref{hierarchy})}{\leq} e^{-\sqrt{n}}
\end{eqnarray*}%
\COMMENT{$\dfrac{5\eps^2\tilde{\eps}\tilde{m}(1-5\tilde{\eps})}{24spF} = \dfrac{5\tilde{\eps}\eps^2\tilde{m}(1-5\tilde{\eps})}{6sr_p} \geq \dfrac{\tilde{\eps}^2\beta n}{s^2 p\tilde{L}^2}$ since $r_p \leq \tilde{\alpha}sp\tilde{L}/\beta \leq sp\tilde{L}/\beta$.}
and $|V_0|F \leq n^2$. So taking a union bound over all $f$ and all $x \in V_0$ we see that the probability that (i) fails for some $f$ in this partition is at most $n^2e^{-\sqrt{n}}$. Similarly, for large $V_0$, fixed $f$ and $y \in X$, Proposition~\ref{chernoff} implies that
\begin{align} \label{ii}
\mathbb{P} \left( d^-_{E_f^+(X)}(y) > \dfrac{2|V_0|}{F} \right) &\leq \exp \left( -\dfrac{\sqrt{\xi}\beta_1m_p}{6F} \right)\\
\nonumber &= \exp \left( -\dfrac{2\sqrt{\xi}\beta_1m}{3sr_p} \right) \stackrel{(\ref{hierarchy}),(\ref{eq:boundrp})}{\leq} e^{-\sqrt{n}}.
\end{align}
Note that $n \leq 2L_p m_p$ by (\ref{mp}) and 
$$
r_p \stackrel{(\ref{rp})}{\geq} sp\tilde{r}/2 \stackrel{(\ref{eq:defr})}{\geq} sp\tilde{\alpha}\tilde{L}/3 \stackrel{(\ref{L})}{=} \tilde{\alpha}L_p/3.
$$
Thus
\begin{equation} \label{claim1y}
\frac{2|V_0|}{F}
\stackrel{(\ref{V0bound})}{\le} 4 \sqrt{\tilde{\eps}} n \frac{p}{r_p} 
\le \dfrac{24\sqrt{\tilde{\eps}}p m_p}{\tilde{\alpha}}
\stackrel{(\ref{hierarchy})}{<} \dfrac{\sqrt{\xi} \beta_1 m_p}{2}.
\end{equation}
Furthermore, $|X|F \leq n^2$, so (\ref{ii}) and (\ref{claim1y}) imply that the probability that (ii) fails for this partition is at most $n^2e^{-\sqrt{n}}$. 
Therefore the partition satisfies both (i) and (ii) with probability $1 - 2n^2 e^{-\sqrt{n}} \geq 1/2$.
This proves the claim.

\medskip

For each $k$ with $1 \le k \le 2p$, 
recall from the end of Section~\ref{redclusters} that the number of all graphs $G_i$ whose original type $t$ satisfies (a)
is $2F$. %where $F$ is as defined in Claim \ref{claim1}. 
For each $k$, consider the ordering of all these digraphs as chosen in Section \ref{redclusters} and suppose that $G_i$ is the $f$th digraph with original type $t$.
We now define the edges of $G_i$ between $V_0$ and $V(G) \setminus V_{0,i}$.
For each base $s$-cluster $W_\ell$, let $W_\ell(k)$ denote the $k$th base $2p$-cluster contained in $W_\ell$.
Apply Claim~\ref{claim1} to obtain $F$ bipartite digraphs $E_f^+(W_\ell(k))$ and $F$ bipartite digraphs $E_f^-(W_\ell(k))$ for each $W_\ell(k)$. Now let $W_\ell'$ denote the $s$-$[i]$-cluster associated with $W_\ell$ and $W_\ell(k)'$ denote the $2p$-$[i]$-cluster associated with $W_\ell(k)$.
Let
$E_f^+(W_\ell(k)')$
be the subdigraph of $E_f^+(W_\ell(k))$ consisting of all edges whose final vertex lies in $W_\ell(k)'$
and let
$E_f^-(W_\ell(k))')$
be the subdigraph of $E_f^-(W_\ell(k))$ consisting of all edges whose initial vertex lies in $W_\ell(k)'$.
Then, by (\ref{largeint2}), for all $x \in V(G)$ we have
\begin{equation} \label{degdiff}
d^+_{E_f^+(W_\ell(k)')}(x) \geq d^+_{E_f^+(W_\ell(k))}(x)-\eps' m_p/2 -1.
\end{equation}
An analogous statement is true for the indegrees in $E_f^-$.
Recall that $k$ with $1 \leq k \leq 2p$ is defined by the fact that $G_i$ has original type $t$ and $t \equiv k \mod 2p$, and that $I_W^+$ and $I_W^-$ are the indices of the non-clean $s$-clusters in $W$ defined at the end of Section~\ref{redclusters}.
If $1 \le f \le F$, we add the following edges to $G_i$:
\begin{itemize}
\item all edges lying in the digraphs $E^+_f (W_\ell(k)')$ with $\ell \in I_W^+$;
\item all edges lying in the digraphs $E^-_f (W_\ell(k)')$ with $\ell \in I_W^-$.
\end{itemize}
If $F < f \le 2F$, we add the following edges to $G_i$:
\begin{itemize}
\item all edges lying in the digraphs $E^+_{f-F} (W_\ell(k)')$ with $\ell \in I_W^-$;
\item all edges lying in the digraphs $E^-_{f-F} (W_\ell(k)')$ with $\ell \in I_W^+$.
\end{itemize}

\COMMENT{We've used $(t)$-clusters here and not $[i]$-clusters because the way edges are assigned depends on the original type $t$ of $i$, and $[i]$-clusters are different for different $i$. EDIT: It's okay to use $[i]$-clusters here because our allocation of edges is done separately for each $i$.
}
Note this implies that all edges from $G_i \setminus V_0$ to $V_0$ have initial vertex in an out-red cluster and similarly for the in-red clusters. Moreover, the sets of edges assigned to $G_i$ and $G_{i'}$ are disjoint for $i \neq i'$. Indeed, this follows from the fact that, for $j \neq k$, $E_f^\pm(W_\ell(j))$ and $E_f^\pm(W_\ell(k))$ are clearly edge-disjoint; that for $f \neq f'$, $E_f^\pm(W_\ell(k)')$ and $E_{f'}^\pm(W_\ell(k)')$ are also edge-disjoint; and that each $E_f^\pm(W_\ell(k))$ is used for at most one of the $G_i$.

Therefore Claim~\ref{claim1},~(\ref{degdiff}) and~(\ref{inred}) imply that for all $x \in V_0$, we have that
\begin{equation*} 
d_{G_i}^+(x) \geq (s/2 -1 ) \sum_{W \in \tilde{\cP}} \left(\frac{1-\eps}{2sp F}(|N^+_G(x) \cap W|-5\tilde{\eps}\tilde{m}) - \dfrac{\eps' m_p}{2} -1 \right).
\end{equation*}
%The final `$-1$' comes from the fact that $2p$-$[i]$-clusters were obtained from $2p$-$(t)$-clusters by removing at most one vertex. 
Note also that
\begin{equation} \label{sKF}
2spF = \frac{s r_p}{2} \stackrel{(\ref{eq:boundrp})}{\le} \frac{s}{2} \frac{\tilde{\alpha} L_p}{\beta} 
\stackrel{(\ref{mp})}{\le} \frac{s}{2} \frac{\tilde{\alpha} n}{\beta m_p}. 
\end{equation}
So
\begin{eqnarray} \label{eq:xGi}
\nonumber d_{G_i}^+(x) &\ge& (s/2 -1)  \frac{(1-\eps')}{2sp F} (\tilde{\alpha} n -|V_0| - 2{\eps}'n)\\
&\stackrel{(\ref{V0bound}),(\ref{sKF})}{\ge}& (1-4 \eps')  \beta m_p \stackrel{(\ref{eq:beta_1})}{\ge} \beta_1 m_p \stackrel{(\ref{eq:kappa})}{\ge} \kappa,
\end{eqnarray}
and we have an analogue for indegrees. So we can delete edges from each 
$x \in V_0$ so that $d_{G_i}^\pm(x) = \kappa$ in each slice and hence (Red1) holds for all vertices in $V_0$.

%%%%%%%%%%%%%%%%%%%%%%%%%%%%%%%%%%%%%%%%%%

\subsection{Incorporating the special exceptional set $V_{0,i}^{\rm spec}$} \label{exceptional2}

We now prove a claim which will be used to achieve (Red1) for the set $V_{0,i}^{\rm spec}$ of special exceptional vertices.
Before this, we first need to derive a further property (H1$'$) of $H_1^\pm$ from (H1).

Write $S_i^+$ for the collection of vertices contained in the out-red $2p$-$[i]$-clusters and define $S_i^-$ analogously. Note that each of $S_i^\pm$ consists of the vertices in exactly $s/2-1$ of the $2p$-$[i]$-clusters in each adapted $s$-$[i]$-cluster.

 For every $1 \leq k \leq 2p$ and every base $s$-cluster $U \in \mathcal{P}_s'$, let $U(k)$ be the $k$th base $2p$-cluster of $U$, and write $H_{1,k}^+$ for the spanning subdigraph of $H_1^+$ consisting of all edges whose final vertex lies in $\bigcup_{U \in \mathcal{P}_s'}{U(k)}$. Also define $H_{1,k}^-$ to be the spanning subdigraph of $H_1^-$ consisting of all edges whose initial vertex lies in $\bigcup_{U \in \mathcal{P}_s'}{U(k)}$.
We have the following property of $H_1^\pm$:
\begin{itemize}
\item[(H1$'$)] For all $x \in V(G) \setminus V_0$, whenever $i$ has original type $t$ and $k$ satisfies $1 \leq k \leq 2p$ and (a) we have that
$$
\dfrac{\gamma \tilde{\alpha}n}{20p} \leq |N_{H_{1,k}^+}^+(x) \cap S_i^-| ~~,~~ |N_{H_{1,k}^-}^-(x) \cap S_i^+| \leq \dfrac{\gamma \tilde{\alpha}n}{p}.
$$
\end{itemize}

\noindent
To prove (H1$'$), note that since $\mathcal{P}_{2p}'$ was a $5\tilde{\eps}$-uniform $2sp$-refinement of $\tilde{\mathcal{P}}$, (URef) implies that, for each $x \in V(G)\setminus V_0$, each $1 \leq k \leq 2p$ and $U \in \mathcal{P}_s'$,
$$
	|N^+_{H_{1,k}^+}(x) \cap U(k)| \geq \dfrac{(1-5\tilde{\eps})}{2sp}\left(|N^+_{H_1^+}(x) \cap \tilde{U}| - 5\tilde{\eps}\tilde{m}\right)
$$
where $\tilde{U}$ is the base primary cluster containing $U$. 
If $U(k)_i$ is the $2p$-$[i]$-cluster associated with $U(k)$, (\ref{largeint2}) implies that
$$
|N_{H_{1,k}^+}^+(x) \cap U(k)_i| \geq |N_{H_{1,k}^+}^+(x) \cap U(k)|-\eps'm_p/2-1.
$$%
\COMMENT{(Recall that (\ref{largeint}) still holds after the removal of bridge vertices). }
But whenever $i$ has original type $t$ and $k$ satisfies (a), $S_i^-$ contains all the vertices from exactly $s/2-1$ of the $2p$-$[i]$-clusters $U(k)_i$ contained in each adapted $[i]$-cluster $\tilde{U}_i$ associated with $\tilde{U}$, so
\begin{align*}
|N_{H_{1,k}^+}^+(x) \cap S_i^- \cap \tilde{U}_i| &\geq (s/2-1)\left( \dfrac{1-5\tilde{\eps}}{2sp}(|N^+_{H_{1}^+}(x) \cap \tilde{U}| - 5\tilde{\eps}\tilde{m}) - \dfrac{\eps'm_p}{2} -1\right)\\
&\geq |N^+_{H_{1}^+}(x) \cap \tilde{U}|/6p - \eps'sm_p.
\end{align*}

Therefore,%
\COMMENT{$\tilde{U}_t \subseteq \tilde{U}$ so $S_i^- \cap \tilde{U}_t = S_i^- \cap \tilde{U}$}
summing over all base primary clusters $\tilde{U}$ and recalling that $V_0$ is an isolated set in $H_1^+$ we have that 
$$
|N^+_{H_{1,k}^+}(x) \cap S_i^-| \geq \dfrac{d_{H_1^+}^+(x)}{6p} - \eps' s m_p\tilde{L} \stackrel{(\text{H}1)}{\geq} \dfrac{\gamma\tilde{\alpha}n}{20p}.
$$
The other bounds in (H1$'$) follow similarly. 
\medskip

\begin{claim} \label{claim2} For each $i$ with $1 \le i \le r_p$, there are subdigraphs $Q^+_i$ of $H^+_1$ and $Q_i^-$ of $H^-_1$ each consisting of edges between
$V_{0,i}^{\rm spec}$ and $V(G) \setminus V_{0,i}$
so that 
\begin{itemize}
\item[(i)] for all $x \in V_{0,i}^{\rm spec}$ we have $|N^+_{Q_{i}^+}(x) \cap S_i^-|, |N^-_{Q_i^-}(x) \cap S_i^+| \ge  \kappa$.
\item[(ii)] For all $y \in V(G) \setminus V_{0,i}$ we have $d^+_{Q^-_i}(y), d^-_{Q^+_i}(y) \leq \sqrt{\xi}\beta_1 m_p/3$.
\item[(iii)] all the $Q_i^\pm$ are pairwise edge-disjoint.
\end{itemize}
\end{claim}

\noindent
To prove the claim, for each vertex $x$ in $V(G) \setminus V_0$, we let
$T(x) := \lbrace i : x \in V_{0,i}^{\rm spec} \rbrace$.
Recall that $x \in V_{0,i}^{\rm spec}$ if and only if
\begin{itemize}
\item[(A)] $x \in \tilde{V}_{0,t}^{\rm spec}$ and $i$ has original type $t$; or
\item[(B)] $x$ was removed to compensate for the removal of a bridge vertex.
\end{itemize}%
\COMMENT{It is not enough for know that each $x$ is good. This only gives a bound on the number of special exceptional sets in which $x$ lies if $x$ was removed to ensure superregularity (A). If it was removed to compensate for a bridge vertex (B), we could have chosen this same $x$ in every slice, so even though $x$ is good, it lies in many special exceptional sets. However, we specified in Sec 5.5 that every `compensating vertex' is unique for a particular slice and $p$-cluster within that slice. I have made this more explicit in Sec 5.5. So we do have $|T(x)|=1$ whenever $x$ is removed due to (B).}
Note that $x$ can satisfy both (A) and (B).
 Suppose that $x$ satisfies (A). Let $\mathcal{L}_x = \lbrace t : x \in \tilde{V}_{0,t}^{\rm spec}\ \rbrace$. 
Note $x \notin V_0$. So $x$ is good in the sense of Section \ref{begin}, and hence $|\cL_x| \leq \xi \tilde{L}/\beta$. 
As observed before (\ref{V0ibound}),
any $x \in V(G) \setminus V_0$ is in at most one set $V_{0,i}^{\rm spec}$ due to (B).
Therefore 
$$
|T(x)| \leq |\cL_x| (s-1)(p-1) +1 \leq \xi \tilde{L} sp/\beta \stackrel{(\ref{L})}{=} \xi L_p/\beta.
$$
For each $1 \leq i \leq r_p$ and each $1 \leq k \leq 2p$ we define digraphs $Q^+_{i,k}$ as follows. For each $k$, we randomly assign each edge of $H_{1,k}^+$ whose initial vertex is $x$ to one of the digraphs $Q_{i,k}^+$ with $i \in T(x)$
with probability $q := \beta/\xi L_p$ (independently of all other edges, and each edge is assigned to at most one of the $Q_{i,k}^+$).
The sum of the probabilities is at most $1$.
Note that $V_0$ is an isolated set in $H_{1,k}^\pm$.
Now define $Q_i^+ := Q_{i,k}^+$ where $i$ has original type $t$ and $k$ satisfies (a).
Then (iii) certainly holds, and for all $x \in V_{0,i}^{\rm spec}$, we have
\begin{equation} \label{degQix}
\ex \left(|N^+_{Q_i^+}(x) \cap S_i^-| \right) = \frac{\beta |N^+_{H^+_{1,k}}(x) \cap S_i^-|}{\xi L_p} \stackrel{(\text{H}1')}{\ge} \frac{\gamma \tilde{\alpha} \beta n}{20p\xi L_p} 
\stackrel{(\ref{mp})}{\ge} 2\beta   m_p.
\end{equation}
Proposition~\ref{chernoff} implies that, for fixed $1 \leq i \leq r_p$ and fixed $x \in V_{0,i}^{\rm spec}$,
$$
\mathbb{P} \left( |N^+_{Q_i^+}(x) \cap S_i^-| < \beta_1 m_p \right) \leq \exp\left(-\dfrac{\beta m_p}{6} \right) \stackrel{(\ref{L}),(\ref{mp})}{\leq} \exp\left( -\dfrac{\beta n}{12 sp\tilde{L}} \right) \leq e^{-\sqrt{n}}.
$$
So a union bound implies that the probability that there exist $i$ and $x$ not satisfying this inequality is at most $n^2 e^{-\sqrt{n}} < 1/4$. 
(i) now follows since $\kappa \leq \beta_1 m_p$ by (\ref{eq:kappa}).
For (ii), note that for \emph{any} vertex $y \in V(G)$ we have
\begin{equation} \label{degQiy}
\mathbb{E} \left( d^-_{Q_i^+}(y) \right) \leq q|V_{0,i}^{\rm spec}| \stackrel{(\ref{V0ibound})}{\leq} \dfrac{\beta}{\xi L_p}\eps n \stackrel{(\ref{mp})}{\leq} \dfrac{2\eps}{\xi}\beta m_p \leq \sqrt{\xi} \beta m_p/4.
\end{equation}
Proposition~\ref{chernoff} shows (as in Claim \ref{claim1}) that the probability that the condition in (ii) fails for some $i$ and some $y \in V(G)$ is at most%
\COMMENT{$r_p n \exp (-\sqrt{\xi}\beta m_p/108) \stackrel{(\ref{eq:boundrp})}{\leq} n^2 e^{-\sqrt{n}}$} $1/4$.
So there is a choice of $Q_1^+ , \ldots , Q_{r_p}^+$ so that all the conditions hold, and similarly for $Q_1^- , \ldots , Q_{r_p}^-$, which proves Claim \ref{claim2}.

\medskip

\noindent
It is now easy to obtain the edges of $G_i$ between $V_{0,i}^{\rm spec}$ and $V(G) \setminus V_{0,i}$. Apply Claim \ref{claim2} to find edge-disjoint digraphs $Q_i^\pm$ for each $1 \leq i \leq r_p$. Recall that $S_i^\pm \subseteq V(G) \setminus V_{0,i}$ and so (Red6) will follow if we add $i$-red edges with initial vertex in $S_i^+$ or final vertex in $S_i^-$. So for each $x \in V_{0,i}$ we add exactly $\kappa$ edges in $Q_i^+$ going from $x$ to $S_i^-$ and exactly $\kappa$ edges in $Q_i^-$ going to $x$ from $S_i^+$.

We have now incorporated $V_{0,i}$ into each $G_i$. It remains to verify that (Red0)--(Red7) hold.
Recall that we partially verified these properties for the red vertices incident to bridge vertices at the end of Section~\ref{connect}.
In particular,
(Red0) was achieved in Section \ref{connect} and the edges we have added here do not affect it. The previous paragraph shows that (Red1) holds for all vertices in $V_{0,i}^{\rm spec}$. Since we already verified it for the bridge vertices $V_{0,i}^{\rm bridge}$ in Section~\ref{connect} and for $V_0$ in Section~\ref{exceptional}, it now 
holds for all vertices in $V_{0,i}$.
Clearly, our construction satisfies (Red2).
(Red3) follows from Claims \ref{claim1}(ii) and \ref{claim2}(ii) and the fact that each non-exceptional vertex is incident to at most one bridge vertex in each slice.
Recall that, in Section \ref{redclusters}, we showed how (Red4) and (Red5) follow from (a) and (b) of the construction.
(Red6) follows from the fact that in constructions including $V_0$ and $V_{0,i}^{\rm spec}$ and $V_{0,i}^{\rm bridge}$, the outedges from $V_{0,i}$ always went to in-red clusters and the inedges
to $V_{0,i}$ came from out-red clusters. 
(Red7) follows immediately from the edge-disjointness of the digraphs in Claims~\ref{claim1} and~\ref{claim2} and the observation in the final paragraph of Section~\ref{connect}.

Note that Theorem \ref{derandom} implies that the proofs of Claims \ref{claim1} and \ref{claim2} can be `derandomised' and so red edges satisfying (Red0)--(Red7) can be found in polynomial time. 

%%%%%%%%%%%%%%%%%%%%%%%%%%%%%%%%%%%%%%%%%%%%%%%%%%%%%%%%%%%%%%%%%%%%%%%%%%%%%%%%%%%%%

\subsection{Finding shadow balancing sequences} \label{shadow}

We have now incorporated all the exceptional vertices to form $r_p$ edge-disjoint slices $G_i$ of $G$, together containing almost all edges, 
such that each slice is a spanning almost-regular subdigraph of $G$. The main aim of this section is to add further red edges to each slice $G_i$ so that the number of red edges sent out by vertices in each cluster $V$ equals the number received by its successor $V^+$ on the cycle of $F_i$ containing $V$.

This `balancing property' is necessary for the following reason. Suppose that $V$ is out-red and suppose that we have a 1-factor $f$ containing a red edge sent out to $V_{0,i}$ by a vertex $x \in V$.
If $V^+$ is not red, any edge of $f$ to $V^+$ must have its initial vertex in $V$. So $f[V,V^+]$ must be a perfect matching, which is impossible since there can be no edge in $f$ from $x$ to a vertex in $V^+$. Note that the absence of red edges incident to $V^-$ does not give rise to the above problem. But we observe a similar problem for $U , U^-$ when $U$ is in-red. So the above `balancing property' is certainly necessary to obtain even a single 1-factor. We will see in Section \ref{aldecomp} that, combined with our other properties, it is also sufficient.

We will add `balancing edges' between non-exceptional vertices to achieve the above property while also ensuring that no vertex is incident to many red edges. 
As indicated above, it will turn out to be sufficient to only add such edges to either the predecessor or successor of existing red clusters. By the end of Section~\ref{balance} our new red clusters will consist of consecutive pairs, well-spaced around each blown-up cycle.

We will first find `shadow balancing edges' in the reduced digraph
between suitable cluster pairs. For this, we will use the fact that $R_p$ is a robust outexpander. Then we will choose the required number of edges from the sparse pre-reserved subdigraph $H_2$ induced on these pairs.
When doing this, we need to be careful to maintain (Red6) with $p$ replaced by $p-1$.

Given $G_i$, we denote the set of red $p$-clusters by $T$ (so we suppress the dependence on $i$ here).
Let $T_{\rm in}$ denote the set of in-red clusters and define $T_{\rm out}$ similarly, so $T = T_{\rm in} \cup T_{\rm out}$.
For a set $S \subseteq T$ of $p$-clusters, we let $S^-$ denote the predecessors of $S$ on $T$ and define $S^+$ similarly. 

Now, for each $1 \leq i \leq r_p$ and each $p$-cluster $V$, let
\begin{equation} \label{siV}
s_i^\pm(V) := \sum\limits_{y \in V} {\vert N^\pm_{G_i}(y) \cap V_{0,i} \vert}
\end{equation}
be the number of red edges entering/leaving $V$. 
So $s^+_i(V)  \neq 0$ only if $V \in T_{\rm out}$ and $s^-_i(V) \neq 0$ only if $V \in T_{\rm in}$.
Note that (Red1) implies that 
\begin{equation} \label{eqbalance}
\sum_{V \in R_p} s_i^+(V)=\sum_{V \in R_p} s_i^-(V).
\end{equation}
Let
\begin{equation} \label{b}
b := \frac{ \xi^{1/6} \beta_1 m_p^2}{ L_p }\ \ \mbox{ and } \ \ c := \xi^{1/5} \beta_1 m_p^2. 
\end{equation}%
\COMMENT{the powers of $\xi$ have swapped}
A \emph{balancing sequence} $B_i$ with respect to $G_i$ is a spanning subdigraph of $H_2$ with the following properties:
\begin{enumerate}
\item[(B1)] $d^\pm_{B_i}(y) \leq 8\xi^{1/6}  \beta_1 m_p$ for every $y \notin V_{0,i}$;
\item[(B2)] We have the following degree conditions: \[
d^+_{B_i}(V) =
\left\{
	\begin{array}{ll}
		s_i^-(V^+)+c & \mbox{if } V \in T_{\rm in}^-\\
        c  & \mbox{if } V \in T_{\rm out} \\
		0 & \mbox{otherwise } 
	\end{array}
\right.
\]
\[
d^-_{B_i}(V) =
\left\{
	\begin{array}{ll}
		c  & \mbox{if } V \in T_{\rm in} \\
		s_i^+(V^-)+c & \mbox{if } V \in T_{\rm out}^+\\
		0 & \mbox{otherwise } 
	\end{array}
\right.
\]%
\COMMENT{this needs prettifying. Okay? Also $d^+_{B_i}(V)$ needs to be defined. Defined in the notation section}
\end{enumerate} 

We will use so called `shadow balancing sequences' as a framework to find balancing sequences. For this, define an auxiliary digraph $R^*$ with $V(R^*)=T$ as follows. Let
\begin{equation} \label{defin*}
N^+_{R^*}(V) =
\left\{
	\begin{array}{ll}
		\left( N^+_{R_p}(V^-) \cap T_{\rm in}\right) \cup \left(  N^+_{R_p}(V^-) \cap T_{\rm out}^+ \right)^-  & \mbox{if } V \in T_{\rm in} \\
		\left( N^+_{R_p}(V) \cap T_{\rm in} \right) \cup \left( N^+_{R_p}(V) \cap T_{\rm out}^+ \right)^- & \mbox{if } V \in T_{\rm out}
	\end{array}
\right.
\end{equation}
This definition reflects the fact that red edges entering $V \in T_{\rm in}$ will be balanced by edges leaving $V^-$ (and entering either $T_{\rm in}$ or the successor $W^+$ of some $W \in T_{\rm out}$). Similarly an edge leaving $V \in T_{\rm out}$ will be balanced by an edge entering $V^+$.
Note that $R^*$ depends on $i$. If we need to emphasise this, we write $R^*_i$.

Define a \emph{shadow balancing sequence} $B_i^\prime$ to be a multidigraph with vertex set $V(R^*)$ whose edges are copies of edges of $R^*$ as follows. Let
\[
n_V^+ :=
\left\{
	\begin{array}{ll}
		s_i^-(V)+c  & \mbox{if } V \in T_{\rm in} \\
		c & \mbox{if } V \in T_{\rm out}
	\end{array}
\right.
\ \ \mbox{ and } \ \ 
n_V^- :=
\left\{
	\begin{array}{ll}
		c  & \mbox{if } V \in T_{\rm in} \\
		s_i^+(V)+c & \mbox{if } V \in T_{\rm out}
	\end{array}
\right.
\]
Then $B_i'$ has the following properties:
\begin{itemize}
\item[(B1$'$)] no edge of $R^*$ appears more than $b$ times in $B_i^\prime$.
\item[(B2$'$)] For every $V \in V(R^*)$, we have $d^+_{B'_i}(V)= n^+_V$ and $d^-_{B'_i}(V)= n^-_V$.
\end{itemize} 
Note that (\ref{eqbalance}) implies that 
\begin{equation} \label{eqbalance2}
\sum_{V \in R^*} n_V^+=\sum_{V \in R^*} n_V^-.
\end{equation}
To find these shadow balancing sequences, we will need that $R^*$ is a robust outexpander with sufficiently large minimum semidegree.

\begin{claim} \label{claimbal} 
Let $\nu'=\nu^3/64$. Then
\begin{itemize}
\item[(i)] $R^*$ is a robust $(\nu',12\tau)$-outexpander.
\item[(ii)] $\delta^0(R^*) \ge \tilde{\alpha} |R^*|/4.$
\end{itemize} 
\end{claim}

\noindent
To prove part (i) of the claim, we will use the fact that an $(s/2-1)$-fold blow-up of a robust $(\nu/4,3\tau)$-outexpander is a $(\nu',6\tau)$-robust outexpander
(see Lemma~\ref{expanderblowup}).
Let $R^{\rm in}_p=R_p[T_{\rm in}]$ and $R^{\rm out}_p=R_p[T_{\rm out}^+]$.
Since every adapted primary cluster contains exactly $s/2-1$ out-red $p$-clusters, it follows that $R^{\rm in}_p$ is an $(s/2-1)$-fold blow-up of $\tilde{R}$.
So it is a robust $(\nu',6\tau)$-outexpander.
Similarly, $R^{\rm out}_p$ is a robust $(\nu',6\tau)$-outexpander.

Consider any $S \subseteq T_{\rm in}$ with $6\tau |T_{\rm in}| \le |S| \le (1- 6\tau ) |T_{\rm in}|$.
Note that $T_{\rm in}$ and $T_{\rm out}$ are disjoint (see e.g.~(Red6)).
So $T_{\rm in}$ and $(T_{\rm out}^+)^-$ are disjoint and hence (\ref{defin*}) implies that
\begin{equation} \label{RNS}
|RN^+_{\nu',R^*}(S)|=|RN^+_{\nu',R_p}(S^-) \cap T_{\rm in}| + |(RN^+_{\nu',R_p}(S^-) \cap T_{\rm out}^+)^-|.
\end{equation}
Now let $S^-_{\rm in}$ be obtained from $S^-$ by replacing each $p$-cluster $V \in S^-$ by an arbitrary (but distinct) $p$-cluster $V_{\rm in} \in T_{\rm in}$
which lies in the same adapted primary cluster as $S^-$. Note this is possible as $S \subseteq T_{\rm in}$ implies that $S$ (and thus $S^-$) contains at most $s/2-1$ of the $p$-clusters from each adapted $s$-cluster.
Note that in $R_p$, each cluster receives an edge from $V_{\rm in}$ if and only if it receives an edge from $V$.
So (\ref{defin*}) implies that 
\begin{align*}
|RN^+_{\nu',R_p}(S^-) \cap T_{\rm in}| & = |RN^+_{\nu',R_p}(S^-_{\rm in}) \cap T_{\rm in}|
= |RN^+_{\nu',R_p^{\rm in}}(S^-_{\rm in})| \\
& \ge |S^-_{\rm in}|+ \nu'|R_p^{\rm in}|=|S|+ \nu' |R^*|/2.
\end{align*}
Similarly, let $S^-_{\rm out}$ be obtained from $S^-$ by replacing each $p$-cluster $V \in S^-$ by an arbitrary (but distinct) cluster $V_{\rm out} \in T_{\rm out}^+$ which lies in the same adapted $s$-cluster as $V$.
Then we have
\begin{align*}
|(RN^+_{\nu',R_p}(S^-) \cap T_{\rm out}^+)^-| 
& = |RN^+_{\nu',R_p}(S^-) \cap T_{\rm out}^+|\\
& = |RN^+_{\nu',R_p}(S^-_{\rm out}) \cap T_{\rm out}^+|
 = |RN^+_{\nu',R_p^{\rm out}}(S^-_{\rm out})| \\
&  \ge |S^-_{\rm out}|+ \nu'|R_p^{\rm out}|=|S|+ \nu' |R^*|/2.
\end{align*}
So altogether, we have $|RN^+_{\nu',R^*}(S)| \ge 2|S|+\nu'|R^*|$.

Now suppose that $S \subseteq T_{\rm out}$ with $6\tau |T_{\rm out}| \le |S| \le (1- 6\tau ) |T_{\rm out}|$.
Similarly as above, (\ref{defin*}) implies that 
\begin{equation} \label{RNS2}
|RN^+_{\nu',R^*}(S)|=|RN^+_{\nu',R_p}(S) \cap T_{\rm in}| + |(RN^+_{\nu',R_p}(S) \cap T_{\rm out}^+)^-| \ge 2|S|+\nu'|R^*|.
\end{equation}
Now consider any $S \subseteq V(R^*)$ with $6\tau |R^*| \le |S| \le (1- 6\tau ) |R^*|$. 
Then either $|S \cap T_{\rm in}| \ge |S|/2$ or $|S \cap T_{\rm out}| \ge |S|/2$.
In either case, we get $|RN^+_{\nu',R^*}(S)| \ge |S|+\nu'|R^*|$.
This proves part (i) of the claim.

To prove part (ii),  suppose that $V \in T_{\rm in}$.
Note that $R_p^{\rm in}$ satisfies $\delta^0(R_p^{\rm in}) \ge \tilde{\alpha} |R_p^{\rm in}|/2$ by Lemma~\ref{expanderblowup}(i).
Choose any $V_{\rm in}^- \in T_{\rm in}$ which lies in the same adapted primary cluster as $V^-$.
Then, similarly as observed above, $V_{\rm in}^-$ has the same outneighbours within the set $T_{\rm in}$ 
as $V^-$ (both in the digraph $R_p$). So the degree bound follows for $V$.
The case when $V \in T_{\rm out}$ is similar. This proves Claim~\ref{claimbal}.
\medskip

It is now easy to find shadow balancing sequences $B_i'$ satisfying (B1$'$) and (B2$'$).
Indeed, note that $c \leq n_V^\pm \leq c + \sqrt{\xi} \beta_1 m_p^2$ by (Red3). 
In particular, (\ref{b}) implies that $n_V^+ = c \left( 1 \pm \xi^{3/10} \right)$ and similarly for $n_V^-$. 
Let $R'$ be obtained from $R^*$ by replacing each of the edges of $R^*$ by $b$ copies of this edge and let $n' := |R^*| = (s-2)\tilde{L}$. We will apply Lemma~\ref{regrobust} as follows:
\medskip
\begin{center}
  \begin{tabular}{ r | c | c | c | c | c | c | c }
     & $R^*$ & $R'$ & $n'$ & $b$ & $\xi^{3/10}$ & $\nu'$ & $c/n'$ \\ \hline
    playing the role of & $G$ & $Q$ & $n$ & $q$ & $\eps$ & $\nu$ & $\rho$ \\
  \end{tabular}
\end{center}
\medskip
\noindent
Then
$$
\rho := \dfrac{c}{n'} \stackrel{(\ref{b})}{=} \dfrac{\xi^{1/5}\beta_1 m_p^2}{(s-2)\tilde{L}} \stackrel{(\ref{hierarchy})}{\leq} \dfrac{\xi^{1/6}\beta_1m_p^2 \nu'^2}{3sp\tilde{L}} \stackrel{(\ref{L}),(\ref{b})}{=} \dfrac{b\nu'^2}{3}
$$
as required by Lemma \ref{regrobust}, and we obtain a 
spanning subdigraph $B_i'$ of $R'$ with $d^\pm_{B_i'}(V) = n^\pm_V$ for each $V \in V(R')=V(R^*)$.%

\subsection{Adding balancing sequences} \label{balance}

Note that for each edge $E'$ of $R^*_i$, there is a unique edge $E$ of $R_p$ (from a $p$-cluster $A$ to a $p$-cluster $B$) which corresponds to $E'$. More precisely, (\ref{defin*}) shows that if $E' = VW \in E(R^*_i)$ then
\begin{equation} \label{correspondence}
E =
\left\{
	\begin{array}{ll}
		V^-W  & \mbox{if } V \in T_{\rm in},W \in T_{\rm in} \\
		V^-W^+  & \mbox{if } V \in T_{\rm in},W \in T_{\rm out} \\
		VW  & \mbox{if } V \in T_{\rm out},W \in T_{\rm in} \\
		VW^+ & \mbox{if } V \in T_{\rm out},W \in T_{\rm out}.
	\end{array}
\right.
\end{equation}
(As before, $V^-$ denotes the predecessor of $V$ on $F_i$.)
So for each edge of $B_i'$, we can choose the corresponding edge of $R_p$. For each $i$ and each edge $E$ of $R_p$, let $c_i(E)$ denote the number of times that the edge $E$ is chosen due to $B_i'$.
So $c_i(E) \le b$ by (B1$'$). %We denote by $H_2[E]$ the subdigraph of $H_2$ consisting of all edges from vertices of $A$ to vertices of $B$. 
If we now replace the chosen edges $E$ of $R_p$ with $c_i(E)$ edges in $H_2(E)$, this will give the required balancing sequence $B_i$.
However, we need to be careful to ensure that we can do this for every $i$ with $1 \le  i \le r_p$ so that all edges are disjoint. We also wish to maintain (Red4) and (Red6).

We now need to consider the dependence on $i$ again, as clusters in different slices are not quite the same.
Given a base $p$-cluster $A$ in $R_p$, let $A[i]$ be the associated $p$-$[i]$-cluster. 
Each $p$-$[i]$-cluster $A{[i]}$ contains at most one red $2p$-$[i]$-cluster by (Red4). If there is such a subcluster, denote it by $A^*{[i]}$. If there is no such subcluster, let $A^*{[i]}$ be an arbitrary subcluster of $A[i]$. We will only add balancing edges incident to $A^*{[i]}$. 
Let $A^*$ be the base $2p$-cluster associated with $A^*{[i]}$. 
Suppose that $E$ is an edge of $R_p$ from $A$ to $B$. 
Let $\tilde{E} \in E(\tilde{R}(\beta))$ be one of the edges whose blow-up contains $E$;
then $H_2(\tilde{E})$ is $(\eps,\gamma \beta)$-regular as observed in Section~\ref{sec:H}. 
Write $H_2(E^*)$ for the subdigraph of $H_2(\tilde{E})$ induced on $(A^* , B^*)$; then by Lemma~\ref{equipartition}(i) we have that $H_2(E^*)$ is $(\eps',\gamma\beta)$-regular. 

Write $H_2(E^*{[i]})$ for the subdigraph of $H_2(E^*)$ induced on $\left(A^*{[i]}, B^*{[i]}\right)$. Whenever $E$ is chosen due to $B_i'$, we will add balancing edges to $G_i$ from $H_2(E^*{[i]})$. By (\ref{largeint2}) we have that, for all $i$ with $1 \leq i \leq r_p$, $H_2(E^*{[i]})$ is a subdigraph of $H_2(E^*)$ obtained by removing at most $\eps'm_p/2 +1$ vertices from each vertex class.

\begin{claim} \label{claim5} Let $d_0:= 8b/m_p^2$ where $b$ is defined in \emph{(\ref{b})}. Suppose that $H$ is a subdigraph of $H_2(E^*)$ obtained by removing at most $\eps'm_p/2 +1$ vertices from each of $A^*$ and $B^*$ and at most $r_p d_0 m_p$ edges at every vertex.
Then $H$ is $(\xi^{1/15},\gamma\beta)$-regular.
\end{claim}

\medskip

\noindent
To prove the claim, note first that 
\begin{equation} \label{eqd0}
d_0=\frac{8b}{m_p^2}=\frac{8 \xi^{1/6} \beta_1}{L_p}.
\end{equation}
So 
$$
2 r_p d_0  \stackrel{(\ref{eq:boundrp})}{\le} \frac{16 \xi^{1/6} \beta_1}{L_p} \frac{\tilde{\alpha} L_p}{\beta} \le 16 \xi^{1/6} \tilde{\alpha} \le \xi^{1/7}.
$$
Also $\eps' \ll \xi^{1/7}$.
So Proposition~\ref{superslice}(i) with $\xi^{1/7}$ playing the role of $d'$ implies the claim.%
\COMMENT{cluster size is $m_p/2$ hence the 2s}

\medskip

\noindent
Now for each $i$ in succession we aim to apply Lemma~\ref{embed} to find a set $C_i(E)$ of $c_i(E)$ edges in $H_2(E^*)$, and remove the edges of $C_i(E)$ 
from further consideration. Suppose we have found $C_1(E),\dots,C_{i-1}(E)$ in $H_2(E^*)$. Suppose further that each of these has maximum degree at most $d_0m_p$ and that the edges are from $A^*$ to $B^*$. We now wish to find $C_i(E)$.

Let $H^{i-1}_2(E^*)$ denote the subdigraph of $H_2(E^*)$ obtained by removing the edges of $C_1(E),\dots,C_{i-1}(E)$ and removing any vertex not present in $H_2(E^*{[i]})$.
So $H^{i-1}_2(E^*)$ is also a subdigraph of $H_2(E^*[i])$.
By (\ref{largeint2}), the number of vertices in each vertex class of $H_2^{i-1}(E^*)$ is at most $\eps' m_p/2+1$ less than that in $H_2(E^*)$. Moreover, at most $r_p d_0 m_p$ edges have been removed from each vertex. 
Then Claim \ref{claim5} implies that $H^{i-1}_2(E^*)$ is $(\xi^{1/15},\gamma\beta)$-regular.
So we can apply Lemma~\ref{embed} to find $C_{i}(E)$, with a maximum degree of at most 
$8 c_{i}(E)/m_p \le 8 b/m_p = d_0 m_p$. We continue inductively until we have found $C_1(E),\dots,C_{r_p}(E)$.

%\COMMENT{The red $2p$-clusters may be different in each slice and $H_2(E^*)$ is dependent on their position.}

Now let $B_i$ be the union of all $C_i(E)$ over all edges $E$ of $R_p$.
Note that the $B_i$ are edge-disjoint by construction.
%Note that if $E{[i]} \in E(R_p{[i]})$ and $E'{[i']} \in E(R_p{[i']})$ are edges which are not associated, the corresponding balancing edges come from disjoint subdigraphs $H_2(E^*)$ and $H_2(E'^*)$ of $H_2$. Therefore the $B_i$ are edge-disjoint. %
To verify (B1), note that for all $y \in V(G)\setminus V_{0,i}$,
$$
d^\pm_{B_i}(y) \le  L_p d_0 m_p \stackrel{(\ref{eqd0})}{=} 8 \xi^{1/6} \beta_1 m_p,
$$
as required. (\ref{correspondence}) implies that the clusters that send out shadow balancing edges are precisely $T_{\rm in}^- \cup T_{\rm out}$ and the clusters that receive shadow balancing edges are precisely $T_{\rm in} \cup T_{\rm out}^+$. Suppose that $V \in T_{\rm in}^-$. Then we have that 
$$
d_{B_i}^+(V) \stackrel{(\ref{correspondence})}{=} d_{B_i'}^+(V^+) \stackrel{(\text{B}2')}{=} n_{V^+}^+ = s_i^-(V^+)+c
$$ 
so (B2) holds in this case. The other cases follow similarly.%
\COMMENT{
\begin{align*}
V \in T_{out}:~~~d_{B_i}^+(V) = d_{B_i'}^+(V) = n_{V}^+ = c\\
V \in T_{out}^+:~~~d_{B_i}^-(V) = d_{B_i'}^-(V^-) = n_{V^-}^- = s_i^+(V^-)+c\\
V \in T_{in}:~~~d_{B_i}^-(V) = d_{B_i'}^-(V) = n_{V}^- = c\\
\end{align*}
}
Therefore $B_i$ satisfies (B1) and (B2). Note that
only vertices in a single $2p$-subcluster of each $p$-cluster (which is the red subcluster if one of them is red)
are incident to a balancing edge.

For each $1 \leq i \leq r_p$ we add the edges of $B_i$ to $G_i$. So now $E(G_i)$ consists of edges from each cluster to its unique successor on $F_i$ together with the $i$-red edges incident to $V_{0,i}$ and the balancing edges $B_i$.

%%%%%%%%%%%%%%%%%%%%%%%%%%%%%%%%%%%%%%%%%%%%%%%%%%%%%%%%%%%%%%%%%%%%%%%%%%

\subsection{Almost decomposing into 1-factors} \label{aldecomp}

Our aim now is to use Lemma~\ref{almostreg2} to find a $\kappa$-regular spanning subdigraph of each $G_i$. 
For this, the `balancing property' achieved in Section~\ref{balance} will be crucial. 

Before this, for each $i$, we first remove a subdigraph $H_{3,i}$ of $G_i$, which will be needed in Section~\ref{mergeH}.
We do this as follows.
For each edge $E$ of $F_i$, recall that $G_i(E)$ is $(2\eps',\beta_1)$-superregular by (Red7). 
Apply Lemma~\ref{randomregslice}(ii) to 
$G_i(E)$ with parameters $K := 2$ and $\gamma_1 := \gamma^2\beta_1, \gamma_2 := \beta_2$ where
\begin{equation} \label{eq:beta2}
\beta_2 := (1-\gamma^2)\beta_1
\end{equation}
to obtain two edge-disjoint subdigraphs of $G_i(E)$: a $(2\eps'^{1/12},\gamma^2 \beta_1)$-superregular digraph $H_{3,i}(E)$ and a $(2\eps'^{1/12},\beta_2)$-superregular `remainder' subdigraph which we still denote by $G_i(E)$. We let $H_{3,i}$ have vertex set $V(G)$ and edge set given by the union of $H_{3,i}(E)$ over all edges $E$ of $F_i$.

\medskip

We now continue with finding a $\kappa$-regular spanning subdigraph of each $G_i$. 
Denote the collection of $i$-red edges incident to $V_{0,i}$ by $\mathcal{T}_i$. For each $1 \leq i \leq r_p$ we call the edges in $\mathcal{T}_i \cup B_i$ and any $p$-cluster containing a vertex incident to such an edge $i$-\emph{red} or \emph{red} (so balancing edges are also regarded as red now). Write $d_{i}^\pm(x) := d_{\mathcal{T}_i}^\pm(x) + d_{B_i}^\pm(x)$ for each $x \in V(G_i)$ and define $d_i^\pm(V) = \sum_{x \in V}{d_i^\pm(x)}$ for $V \in V(F_i)$. 
So by (\ref{siV}) we have that, for each $V \in V(F_i)$,
\begin{equation} \label{di}
d_i^\pm(V) = s_i^\pm(V) + d_{B_i}^\pm(V).
\end{equation}
For each $1 \leq i \leq r_p$ we now have the following properties:
\medskip
\begin{enumerate}
\item[(Red0$'$)] There exists a sequence $D_1 x_1 D_2 x_2 \ldots x_{\ell-1} D_{\ell} x_\ell D_1$ with the following properties:
\begin{itemize}
\item Each $D_j$ is a cycle of $F_i$ and every cycle of $F_i$ appears at least once in the sequence;
\item $V_{0,i}^{\rm bridge} = \lbrace x_1 , \ldots , x_\ell \rbrace$ and each $x_j$ has exactly $\kappa$ outneighbours in $D_{j+1}$ and exactly $\kappa$ inneighbours in $D_j$;
\end{itemize}
\item[(Red1$'$)] $d_i^\pm(x) = \kappa$ for each $x \in V_{0,i}$;
\item[(Red2$'$)] $V_{0,i}$ is an independent set in $G_i$;
\item[(Red3$'$)] $d_i^\pm(y) \leq \xi^{1/7} \beta_2 m_p$ for each $y \in G_i \setminus V_{0,i}$;
\item[(Red4$'$)] For every red cluster $V \in R_p$, all $i$-red edges are incident to a single $2p$-cluster contained in $V$. In particular, at most $m_p/2$ vertices in $V$ are incident to an $i$-red edge;
\item[(Red5$'$)] In $F_i$ any out-red $p$-cluster $V$ is preceded by $p-3$ $p$-clusters which are neither out-red nor in-red, and is succeeded by an in-red $p$-cluster. Any in-red $p$-cluster $V$ is succeeded by $p-3$ $p$-clusters which are neither out-red nor in-red, and is preceded by an out-red $p$-cluster;
\item[(Red6$'$)] Each $p$-cluster is either out-red, in-red or contains no vertices incident to a red edge;
\item[(Red7$'$)] $G_1 , \ldots , G_{r_p}$ are edge-disjoint and $G_i(E)$ is $(2\eps^{\prime 1/12} , \beta_2)$-superregular for all $E \in E(F_i)$;
\item[(B2$''$)]  $d^+_i(V) = d^-_i(V^+)$ for all $p$-clusters $V \in V(F_i)$.
\end{enumerate}
\medskip
\noindent 
(Red0$'$), (Red1$'$) and (Red2$'$) follow immediately from (Red0), (Red1) and (Red2) respectively.
(Red3$'$) follows from summing the degrees given by (Red3) and (B1) and using (\ref{eq:beta2}). (Red4$'$) is a consequence of (Red4) and our choice of edges in Section \ref{balance}. (Red5$'$) follows from (Red5) and (B2): indeed, the (red) clusters in $T=T_{\rm in} \cup T_{\rm out}$ are separated by exactly $p-1$ non-red clusters by (Red5), and by (B2), the only other red clusters are precisely those in $T_{\rm in}^- \cup T_{\rm out}^+$.
(Red6$'$) and edge-disjointness in (Red7$'$) follow from (Red6) and edge-disjointness in (Red7), as well as the construction of $B_i$ in Sections~\ref{shadow} and \ref{balance}.
The second part of (Red7$'$) was verified directly after~(\ref{eq:beta2}).
(B2$''$) is a direct consequence of (B2) and (\ref{di}). So for example, if $V \in T_{\rm out}$ then 
$$
d_i^+(V) = s_i^+(V) + c = d_{B_i}^-(V^+) = d_i^-(V^+).
$$%
\COMMENT{$V \in T_{in}$: $d_i^-(V) = s_i^-(V) + c = d_{B_i}^+(V^-) = d_i^+(V^-)$}

Consider any edge $E$ from $V$ to $V^+$ in $F_i$. We wish to find a subdigraph $G_i(E)^*$ of $G_i(E)$ such that, together with the red edges incident to $V$ and $V^+$, every vertex in $V$ has outdegree $\kappa$ and every vertex in $V^+$ has indegree $\kappa$. The union of these subdigraphs over all edges $E \in E(F_i)$, together with the red edges $B_i \cup \mathcal{T}_i$, will form a $\kappa$-regular spanning subdigraph $G_i^*$ of $G_i$.
(Recall that $\kappa$ was defined in~(\ref{eq:kappa}).)

Given any $x \in V$, let $m^+_x = d_i^+(x)$ and given any $y \in V^+$, let $m^-_y = d_i^-(y)$. By (Red3$'$) we have that $m^+_x , m^-_y \leq \xi^{1/7} \beta_2 m_p$ and by (B2$''$) we have that
\[
\sum\limits_{x \in V}{m^+_x} = \sum\limits_{y \in V^+}{m^-_y}.
\]
Let $\hat{\eps} := 2 \eps'^{1/12}$ and $\hat{\beta} := \beta_2 - \hat{\eps}$. So (Red7$'$) implies that $G_i(E)$ is $(\hat{\eps},\hat{\beta} + \hat{\eps})$-superregular for every $E \in E(F_i)$. Let
$$
\hat{\alpha} := 1 - \dfrac{(1-\gamma)\beta_1}{\beta_2 - \hat{\eps}}.
$$
So $\kappa = (1-\hat{\alpha})\hat{\beta} m_p$, and it is easy to see that $\gamma/2 \leq \hat{\alpha} \leq 2\gamma$, so that $\hat{\beta} \ll \hat{\alpha} \ll 1$.
Thus we can apply Lemma~\ref{almostreg2} to $G_i(E)$
with $\hat{\eps}$ playing the role of $\eps$, $\hat{\beta}$ playing the role of $\beta$ and $\hat{\alpha}$ playing the role of $\alpha$.
Then we obtain a spanning subdigraph $G_i(E)^*$ of $G_i(E)$ in which each $x \in V$ has outdegree $\kappa - m^+_x$ and each $y \in V^+$ has indegree $\kappa - m^-_y$.  
Then
\[
G_i^* := \bigcup\limits_{E \in E(F_i)}{G_i(E)^*} \cup B_i \cup \mathcal{T}_i
\]
is a $\kappa$-regular spanning subdigraph of $G_i$ as required. Moreover $G_1^* , \ldots , G_{r_p}^*$ are edge-disjoint subdigraphs of $G$ by (Red7$'$). Now apply Proposition~\ref{petersen} to each $G_i^*$ to obtain $\kappa$ edge-disjoint 1-factors $f_{i,1}, \ldots , f_{i,\kappa}$ of each $G_i$. 

%%%%%%%%%%%%%%%%%%%%%%%%%%%%%%%%%%%%%%%%%%%%%%%%%%%%%%%%%%%%%%%%%%%%%%%%%%%

\subsection{Merging 1-factors into Hamilton cycles} \label{mergeH}

The final step is to use edges disjoint from our collection of 1-factors to merge cycles such that each 1-factor is transformed into a Hamilton cycle. Then we will have found an approximate decomposition into edge-disjoint Hamilton cycles.
The argument will be exactly the same for each $G_i$. So 
since we will work within a fixed $G_i$, we will label the $\kappa$ factors obtained from $G_i$ as $f_{1}, \ldots , f_{\kappa}$.
We wish to use Lemma~\ref{trick} and edges from our pre-reserved digraph $H_{3,i} $ to merge the cycles in each $f_j$.

We say that a non-red cluster is \emph{black} and we say that an edge of $F_i$ is \emph{black} if
both the initial cluster and final cluster are black.
So for all black edges $VV^+$ in $F_i$ we have that $f_j[V,V^+]$ is a perfect matching for each $f_j$, since in $G_i$ every edge from a vertex in $V$ goes to a vertex in $V^+$.
(Red5$'$) implies that every pair $U_{\rm out}U_{\rm in}$ of consecutive red clusters on any cycle of $F_i$ is followed by $p-3$ consecutive black clusters. Denote the path of length $p-4$ from the first of these black clusters to the last by $I_U$, so every edge in $I_U$ is black.
So we can choose $p-4$ disjoint sets of edges $J_1,\dots,J_{p-4}$ of $F_i$ so that for each pair of consecutive red clusters $U_{\rm out}U_{\rm in}$, $J_q$ contains exactly one edge of $I_U$. So each $J_q$ consists of exactly $|T| = |T_{\rm in}| + |T_{\rm out}| = (s-2)\tilde{L}$ edges of $F_i$ and has non-empty intersection with any cycle of $F_i$.

The idea is to apply Lemma~\ref{trick} repeatedly to transform each of the $f_j$ into a Hamilton cycle. Each time $H_{3,i}$ will play the role of $G$, and each $J_q$ will play the role of $J$ roughly $\kappa/p$ times. If $\mathcal{E}$ is a set of edges in $F_i$, we write $H_{3,i}(\mathcal{E}) := \bigcup_{E \in \mathcal{E}}{H_{3,i}(E)}$.

We now describe the merging procedure for $f_1$.
Denote the cycles of $F_i$ by $D_1,\dots,D_\ell$.
Let $K_1$ be the $1$-regular digraph consisting of all cycles of $f_1$ which contain a vertex in a cluster of $D_1$.
Now apply Lemma~\ref{trick} as follows: $D_1$ plays the role of $C$, $J_1 \cap E(D_1)$ plays the role of $J$,
$K_1$ plays the role of $F$ and $H_{3,i}(J_1)$ plays the role of $G$.

Condition (i) in Lemma~\ref{trick} is clearly satisfied since every edge of $J_1$ is black.
To verify condition (ii), let $D$ be any cycle of $K_1$.
We claim that $D$ contains a vertex $x$ from a black cluster $B$. 
To see this, suppose that $D$ contains a vertex $y$ which lies in an in-red cluster.
Then the next vertex of $D$ lies in a black cluster.
Similarly, if $y$ lies in an out-red cluster, then the vertex preceding $y$ on $D$ lies in a black cluster, which proves the claim.
Now let $I_U$ be the black interval containing $B$; then there is a path in $D$ (containing $x$) which contains at least one vertex from each cluster in $I_U$. But $J_1 \cap E(D_1)$ contains an edge of $I_U$, as required.

To verify (iii), let $VV^+$ and $WW^+$ be edges of $J_1 \cap E(D_1)$ such that $J_1$ avoids all edges in the segment $V^+D_1W$. Then there is exactly one pair of successive red clusters $U_{\rm out}U_{\rm in}$ in this segment. So for each $v_a \in V^+$ there is a path $P_a$ in $f_1$ from $v_a$ to a distinct vertex $u^{\rm out}_a$ in $U_{\rm out}$ which winds around $D_1$. Similarly, for each $u_{a'}^{\rm in} \in U_{\rm in}$ there is a path $P_{a'}'$ in $f_1$ from $u_{a'}$ to a distinct vertex $w_{a'} \in W$ which winds around $D_1$. But by (Red4$'$), for at least half of the vertices $u_a^{\rm out} \in U_{\rm out}$, there is an edge in $f_1$ to some $u_{a'}^{\rm in} \in U_{\rm in}$. So $f_1$ contains at least one path $v_aP_au_a^{\rm out}u_{a'}^{\rm in}P_{a'}'w_{a'}$ from $v_a \in V^+$ to $w_{a'} \in W$ which winds around $D_1$, as required.  

So we can find a matching $M_1$ in $H_{3,i}(J_1)$ and a cycle $C_1$ with $V(C_1)=V(K_1)$ and $E(C_1) \subseteq K_1 \cup M_1$.
We replace the 1-regular subdigraph $K_1$ of $f_1$ by $C_1$. 
We call the resulting $1$-factor $f_1(1)$ and we denote $H_{3,i} \setminus M_1$ by $H^2_{3,i}$. Note that all cycles of $f_1$ which contained a vertex in $D_1$ have now been merged into a single cycle of $f_1(1)$.

For $2 \leq k \leq \ell$ we define $f_1(k)$ inductively as follows. Let $K_k$ be the $1$-regular digraph consisting of all cycles of $f_1(k-1)$ which contain a vertex in a cluster of $D_k$.
Now let $D_k$ play the role of $C$, $J_1 \cap E(D_k)$ play the role of $J$,
$K_k$ play the role of $F$ and $H_{3,i}(J_1)$ play the role of $G$.
Note that the $k$ choices $J_1 \cap E(D_{k'})$ with $1 \leq k' \leq k$ playing the role of $J$ so far are pairwise vertex-disjoint.
Exactly as above, the conditions (i)-(iii) are satisfied and we can apply Lemma~\ref{trick} to obtain a 1-factor $f_1(k)$ in which all cycles containing a vertex in $D_k$ have been merged. 
Moreover if two vertices $x$ and $y$ lie on a common cycle of $f_1(k-1)$ they lie on a common cycle of $f_1(k)$. 
We repeat this for all $1 \leq k \leq \ell$ to obtain $f_1' := f_1(\ell)$.
We will see below that $f_1'$ is a Hamilton cycle.

We now aim to carry out a similar procedure for $f_2 , \ldots , f_\kappa$ to obtain $f_2' , \ldots , f_\kappa'$. 
The approach will be to use $J_1$ for $f_1, \ldots, f_{\kappa'}$ where $\kappa' := \kappa/(p-4)$ and more generally to use $J_q$ for $f_{(q-1)\kappa'+1}, \ldots, f_{q\kappa'}$. 
Note that, to obtain $f_1'$, we removed exactly one perfect matching from
each $H_{3,i}(E)$ for each edge $E$ of $J_1$.
To reuse $J_1$ we need only check that, at each step and for each edge $E$ of $J_1$, the remainder of the sparse digraph $H_{3,i}(E)$ satisfies the conditions required of $G$ in Lemma~\ref{trick}.  
For this, let $H_{3,i}^t(J_q)$ denote a subdigraph of $H_{3,i}(J_q)$ obtained by removing $t$ arbitrary perfect matchings from $H_{3,i}(E)$ for each $E \in J_q$. 

\begin{claim} \label{minusmatching}
Let $\kappa'$ be defined as above and let $\eps^* := 2\sqrt{\beta_1/p}$. Then $H_{3,i}^{\kappa'}(E)$ is $(\eps^*,\gamma^2\beta_1)$-superregular whenever $E$ is an edge in $J_q$, where $1 \leq q \leq p-4$.
\end{claim}

\noindent To see this, it suffices to consider a single edge $E = XY$ in $J_1$. Write $H := H_{3,i}^{\kappa'}(E)$. Then, since at each stage we removed a perfect matching, in total we removed $\kappa'$ edges incident to each vertex in $X \cup Y$, which is at most $\beta_1m_p/p$ by (\ref{eq:kappa}). Since $H_{3,i}(E)$ is $(2\eps'^{1/12},\gamma^2\beta_1)$-superregular (see directly after (\ref{eq:beta2})), we can apply Proposition~\ref{superslice}(ii) with $H_{3,i}(E)$ playing the role of $G$, $H$ playing the role of $G'$ and $d' := \beta_1/p$ to find that $H$ is $(\eps^*,\gamma^2\beta_1)$-superregular. Note that $\eps^* \ll \gamma^2\beta_1$ by (\ref{hierarchy}). This proves the claim.

\medskip
Suppose that we have constructed $f_1', \ldots, f_t'$ with $t < \kappa'$ in the same way as $f_1'$. Then we will have used $t$ perfect matchings in $H_{3,i}(E)$ for each $E \in J_1$. Let $H_{3,i}^t(J_1)$ denote the subdigraph of $H_{3,i}(J_1)$ consisting of the remaining edges. Then Claim~\ref{minusmatching} implies that $H_{3,i}^t(J_1)$ can still play the role of $G$ in Lemma~\ref{trick}. So we can construct $f_{t+1}'$ in the same way as $f_1'$. 
Thus we can obtain $f_1', \ldots, f_{\kappa'}'$ as described above.

Now for each $2 \leq q \leq p-4$ and each $1 \leq t \leq \kappa'$ we can use $J_q$ to obtain $f_{(q-1)\kappa'+t}'$ from $f_{(q-1)\kappa'+t}$
in exactly the same way (except that we use edges from $H_{3,i}(J_q)$ and so $J_q \cap E(D_k)$ now plays the role of $J$ for $1 \leq k \leq \ell$).%
\COMMENT{More precisely, write $f_{(q-1)\kappa'+t}(0) := f_{(q-1)\kappa'+t}$ and $H_{3,i}^0(J_q) := H_{3,i}(J_q)$. For each $1 \leq j \leq \ell$ let $K_{(q-1)\kappa'+t}$ be the 1-regular digraph consisting of all cycles of $f_{(q-1)\kappa'+t}(j-1)$ which contain a vertex in a cluster of $D_j$. Apply Lemma~\ref{trick} with $D_j$ playing the role of $C$, $J_q \cap E(D_j)$ playing the role of $J$, $K_1$ playing the role of $F$ and $H_{3,i}^{t-1}(J_q)$ playing the role of $G$ to obtain $f_{(q-1)\kappa'+t}(j)$.
By Claim~\ref{minusmatching}, $H_{3,i}^t(J_q)$ satisfies the conditions required of $G$ in the lemma.
Exactly as for $f_1$ above, $f_{(q-1)\kappa'+t}(j)$ is a 1-factor in which all cycles containing a vertex in $D_j$ have been merged, and if two vertices lie on a common cycle of $f_{(q-1)\kappa'+t}(j-1)$ they also lie on a common cycle of $f_{(q-1)\kappa'+t}(j)$.
Write $f_{(q-1)\kappa'+t}' := f_{(q-1)\kappa'+t}(\ell)$. 
Now let $H_{3,i}^j(J_q)$ denote the remainder of $H_{3,i}^{j-1}(J_q)$ after these $\ell$ applications of the lemma. }

We have now obtained $f_1' , \ldots , f_\kappa'$. They are clearly edge-disjoint 1-factors. We claim that $f_j'$ is a Hamilton cycle for each $1 \leq j \leq \kappa$. Indeed, suppose not. It suffices to consider $f_1'$. Let $C$ and $C'$ be cycles in $f_1'$ where $C$ contains a vertex $x$ in some cycle $D$ of $F_i$ and $C'$ contains a vertex $x'$ in some cycle $D'$ of $F_i$. Recall that, by our construction, for all cycles $D_k$ in $F_i$, every vertex in (a cluster of) $D_k$ is contained in a single cycle in $f_1'$. Consider the sequence given by (Red0$'$) as a cyclic sequence and pick an interval
\[
D_{g}x_{g}D_{g+1}x_{g+1}\dots x_{g'-1}D_{g'} x_{g'}
\]
such that $D = D_{g}$ and $D' = D_{g'}$. By (Red0$'$) and (Red1$'$), the inneighbour of $x_{g}$ in $f_1'$ is contained in $D$, so $x_{g} \in V(C)$. But similarly the outneighbour of $x_{g}$ in $f_1'$ is contained in $D_{g+1}$, so all vertices lying in a cluster of $D_{g+1}$ are contained in $V(C)$  and thus $x_{g+1} \in V(C)$. Continuing along the subsequence we conclude that
every vertex lying in a cluster of $D'$ lies on $C$. So $x'$ lies on both $C'$ and $C$; so since $f_1'$ is a 1-factor we must have $C = C'$. 
Thus $f_1'$ is a Hamilton cycle, and the same holds for $f_2' , \ldots , f_\kappa'$.    

Finally, we can bound the total number of Hamilton cycles as follows.
Note that
\begin{align*}
\kappa &\stackrel{(\ref{eq:beta_1}),(\ref{eq:m}),(\ref{eq:kappa})}{=} (1-\gamma)(1-5\gamma)\beta \frac{m}{sp}.\\
r_p &\stackrel{(\ref{eq:defr}),(\ref{rs}),(\ref{rp})}{=} (s-1)(p-1)(\tilde{\alpha}-\gamma) \frac{\tilde{L}}{\beta} \geq (1-\sqrt{\gamma})sp \frac{{\tilde{\alpha}\tilde{L}}}{\beta}.
\end{align*}
So altogether, after repeating the procedure for every $1 \leq i \leq r_p$, we have found
\begin{eqnarray}
\nonumber r_p \kappa &\geq& (1-\gamma)(1-5\gamma)(1-\sqrt{\gamma})\tilde{\alpha}\tilde{L}\tilde{m}\\
\nonumber &\stackrel{(\ref{boundL})}{\geq}& (1-\sqrt{\gamma})^3(1-\eps)\tilde{\alpha} n\\
\nonumber &\stackrel{(\ref{hierarchy})}{\geq}& (1-\eta)r
\end{eqnarray}
edge-disjoint Hamilton cycles, as required. This completes the proof of Theorem~\ref{main}.

\section{The proof of Corollary~\ref{cor}} \label{strengthening}

We now use Theorem~\ref{main} and Lemma~\ref{regrobust} to prove Corollary~\ref{cor}.

\begin{proof}
As in the proof of Theorem~\ref{main}, we may assume without loss of generality that $0 < \eta \ll \nu \ll \tau \ll \alpha$. 
Choose $n_0$ and $\gamma$ so that $0 < 1/n_0 \ll \gamma \ll \eta$.
Suppose that $G$ is a digraph on $n \geq n_0$ vertices satisfying (i) and (ii).
Let
\[
n_x^\pm := d_G^\pm(x) - (\alpha - \sqrt{\gamma})n
\]
for each $x \in V(G)$. We apply Lemma~\ref{regrobust} to $G$ with $\rho = \eps = \sqrt{\gamma}$ and with $Q = G$ (so $q=1$) to obtain a subdigraph $H$ of $G$ such that $\tilde{G} := G \setminus H$ is an $(\alpha-\sqrt{\gamma})n$-regular digraph on $n$ vertices. 
Note that for all $x \in V(G)$ we have $d_{\tilde{G}}^-(x) \geq d_G^-(x) - (\sqrt{\gamma} - \gamma)n \geq d_G^-(x) - \nu n/2$.
So for all sets $S$ of vertices,
\[
RN_{\nu/2 , \tilde{G}}^+(S) \supseteq RN_{\nu , G}^+(S).
\]
Thus $\tilde{G}$ is a robust $(\nu/2 , \tau)$-outexpander. Therefore we can apply Theorem~\ref{main} to $\tilde{G}$ with parameter $\eta^\prime := \eta/2\alpha$ to find $(1 - \eta^\prime)(\alpha - \sqrt{\gamma})n > (\alpha - \eta)n$ edge-disjoint Hamilton cycles in $\tilde{G}$ and hence in $G$.
\end{proof}

\section*{Acknowledgements}
We are extremely grateful to Daniela K\"uhn for helpful discussions throughout the project.

\end{document}